\documentclass[12pt]{article}

\usepackage{amsmath}
\usepackage{amsthm}
\usepackage{amssymb}
\usepackage{color,epsfig}
\usepackage{fullpage}
\usepackage{enumerate}
\usepackage{paralist}
\usepackage{hyperref}
\usepackage{relsize}
\usepackage{exscale}

\usepackage{tikz}
\usepackage{subfig}

\usepackage{ifpdf}

\newtheorem{theorem}{Theorem}
\newtheorem{proposition}[theorem]{Proposition}

\newtheorem{cor}[theorem]{Corollary}
\newtheorem{lemma}[theorem]{Lemma}
\newtheorem{remark}[theorem]{Remark}
 
\newcommand{\dist}{\mathrm{d_h}}
\newcommand{\distp}{\mathrm{d'_h}}
\newcommand{\distpp}{\mathrm{d''_h}}

\newcommand{\Vol}{\mathrm{vol}}
\newcommand{\calA}{\mathcal{A}}
\newcommand{\calB}{\mathcal{B}}
\newcommand{\calC}{\mathcal{C}}

\newcommand{\calL}{\mathcal{L}}
\newcommand{\calO}{\mathcal{O}}
\newcommand{\calP}{\mathcal{P}}
\newcommand{\calU}{\mathcal{U}}
\newcommand{\calQ}{\mathcal{Q}}
\newcommand{\calR}{\mathcal{R}}
\newcommand{\calS}{\mathcal{S}}

\newcommand{\NN}{\mathbb{N}}

\newcommand{\EE}{\mathbb{E}}
\newcommand{\VV}{\mathbb{V}}
\newcommand{\HH}{\mathbb{H}}
\newcommand{\RR}{\mathbb{R}}

\renewcommand{\Pr}{\mathbf{P}}
\newcommand{\Ex}{\mathbf{E}}

\newcommand{\elllow}{\ell_{\mathrm{low}}}
\newcommand{\ellbdr}{\ell_{\mathrm{bdr}}}

\newcommand{\ellmid}{\ell_{\mathrm{mid}}}
\newcommand{\ellmax}{\ell_{\mathrm{max}}}
\newcommand{\ellmin}{\ell_{\mathrm{min}}}

\newcommand{\unifmod}{\mathrm{Unf}}
\newcommand{\poimod}{\mathrm{Poi}}

\newcommand{\myatop}[2]{\genfrac{}{}{0pt}{}{#1}{#2}}

\begin{document}

\title{Spectral Gap of Random Hyperbolic Graphs and Related Parameters}

\author{Marcos Kiwi\thanks{Depto.~Ing.~Matem\'{a}tica \&
  Ctr.~Modelamiento Matem\'atico (CNRS UMI 2807), U.~Chile. 
  Beauchef 851, Santiago, Chile, Email: \texttt{mkiwi@dim.uchile.cl}. Gratefully acknowledges the support of 
  Millennium Nucleus Information and Coordination in Networks ICM/FIC P10-024F
  and CONICYT via Basal in Applied Mathematics.} \\
\and 
Dieter Mitsche\thanks{Universit\'{e} de Nice Sophia-Antipolis, Laboratoire J-A Dieudonn\'{e}, Parc Valrose, 06108 Nice cedex 02, Email: \texttt{dmitsche@unice.fr}. The main part of this work was performed during a stay at Depto.~Ing.~Matem\'{a}tica \& Ctr.~Modelamiento Matem\'atico (CNRS UMI 2807), U.~Chile, and the author would like to thank them for their hospitality.}}

\maketitle 
 
\begin{abstract}
Random hyperbolic graphs have been suggested as a promising model 
  of social networks.
A few of their fundamental parameters have been studied.
However, none of them concerns their spectra.
We consider the random hyperbolic graph model 
  as formalized by~\cite{GPP12} and 
  essentially determine the spectral gap of their normalized Laplacian.
Specifically, we establish that with high probability 
  the second smallest eigenvalue of the 
  normalized Laplacian of the giant component of an $n$-vertex 
  random hyperbolic graph is at least
  $\Omega(n^{-(2\alpha-1)}/D)$, 
  where 
  $\frac12<\alpha<1$ is a model parameter and $D$ is the 
  network diameter (which is known to be at most polylogarithmic in $n$).
We also show a matching (up to a polylogarithmic factor) upper bound of 
  $n^{-(2\alpha-1)}(\log n)^{1+o(1)}$.

As a byproduct we conclude that the conductance upper bound on the 
  eigenvalue gap obtained via Cheeger's inequality is essentially tight.
We also provide a more detailed picture of the collection of vertices
  on which the bound on the conductance is attained, in particular 
  showing that for all subsets whose volume is 
  $O(n^{\varepsilon})$ for $0<\varepsilon< 1$ 
  the obtained conductance is with high probability 
  $\Omega(n^{-(2\alpha-1)\varepsilon+o(1)})$. 
Finally, we also show consequences of our result for the minimum and 
  maximum bisection of the giant component.
\end{abstract}

\section{Introduction}
It has been empirically observed that many networks, in particular
  so called social networks, are typically scale-free and exhibit 
  non-vanishing clustering coefficient.
Several models of random graphs exhibiting either scale freeness or
  non-vanishing clustering coefficient have been proposed. 
A model that seems to naturally exhibit both properties is 
  the one introduced rather recently by Krioukov et al.~\cite{KPKVB10}
  and referred to as random hyperbolic graph model, which 
  is a variant of the classical random geometric graph model
  adapted to the hyperbolic plane.
The resulting graphs have key properties observed in large 
  real-world networks. 
This was convincingly demonstrated by Bogu\~n\'{a} et al.~in~\cite{BPK10}
  where a maximum likelihood fit of the 
  autonomous systems of the internet graph in hyperbolic 
  space is computed.
The impressive quality of the embedding obtained is an
  indication that hyperbolic geometry underlies 
  important real networks.
This partly explains the considerable interest the model has
  attracted since its introduction.

Formally, the random hyperbolic graph model $\unifmod_{\alpha,C}(n)$ is defined
  in~\cite{GPP12} as described next: for $\alpha> \frac12$, $C\in\RR$, $n\in\NN$, set $R=2\log n+C$ ($\log$ denotes here and throughout the paper the natural logarithm), and 
  build $G=(V,E)$ with vertex set $V$ a subset of $n$ points
  of the hyperbolic plane $\HH^2$ chosen as follows:
\begin{itemize}
\item For each $v\in V$, polar coordinates $(r_{v},\theta_{v})$ are 
  generated identically and independently distributed with joint
  density function $f(r,\theta)$, with $\theta_{v}$ chosen uniformly 
  at random in the interval $[0,2\pi)$ and $r_{v}$ with density:
\begin{align*}
f(r) & := \begin{cases}\displaystyle
   \frac{\alpha\sinh(\alpha r)}{C(\alpha,R)}, &\text{if $0\leq r< R$}, \\
   0, & \text{otherwise},
  \end{cases}
\end{align*}
where $C(\alpha,R)=\cosh(\alpha R)-1$ is a normalization constant.

\item For $u,v\in V$, $u\neq v$, there is an edge with endpoints 
  $u$ and $v$ provided the distance (in the hyperbolic plane) between
  $u$ and $v$ is at most $R$, i.e.,  
  the hyperbolic distance
  between two vertices whose native representation
  polar coordinates are $(r,\theta)$ and $(r',\theta')$, denoted by 
  $\dist:=\dist(r_{u},r_{v},\theta_{u}-\theta_{v})$,
  is such that  $\dist\leq R$ where $\dist$ is obtained by solving 
\begin{equation}\label{eqn:coshLaw}
\cosh \dist := \cosh r\cosh r'-\sinh r\sinh r'\cos(\theta{-}\theta').
\end{equation}
\end{itemize}
The restriction $\alpha>\frac12$ and the role of $R$, informally speaking,
  guarantee that the resulting graph has bounded average degree (depending
  on $\alpha$ and $C$ only): intuitively, 
if $\alpha<\frac12$, then the degree sequence is so 
  heavy tailed that this is impossible, and if $\alpha>1$, then
  as the number of vertices grows,
  the largest component of a random hyperbolic graph has sublinear 
  order~\cite[Theorem~1.4]{BFM15}.
In fact, although some of our results hold for a wider range of $\alpha$, 
  we will always assume $\frac12<\alpha<1$, since as 
  already discussed, this is the most interesting regime.
  
\medskip
A common way 
  of visualizing the hyperbolic plane $\HH^2$ is via its
  native representation where the choice for 
  ground space is $\RR^2$.
Here, a point of $\RR^2$ with polar coordinates $(r,\theta)$ has 
  hyperbolic distance to the origin $O$ equal to its Euclidean distance~$r$.
In the native representation, an instance of $\unifmod_{\alpha,C}(n)$
  can be drawn by mapping a vertex $v$ to the point in $\RR^2$ with polar
  coordinate $(r_v,\theta_v)$ and drawing edges as straight lines.
Clearly, the graph drawing will lie within $B_{O}(R)$ 
  (see Figure~\ref{fig:hyperGraph}).

\begin{figure}[ht]
\begin{center}
\scalebox{0.75}{\includegraphics{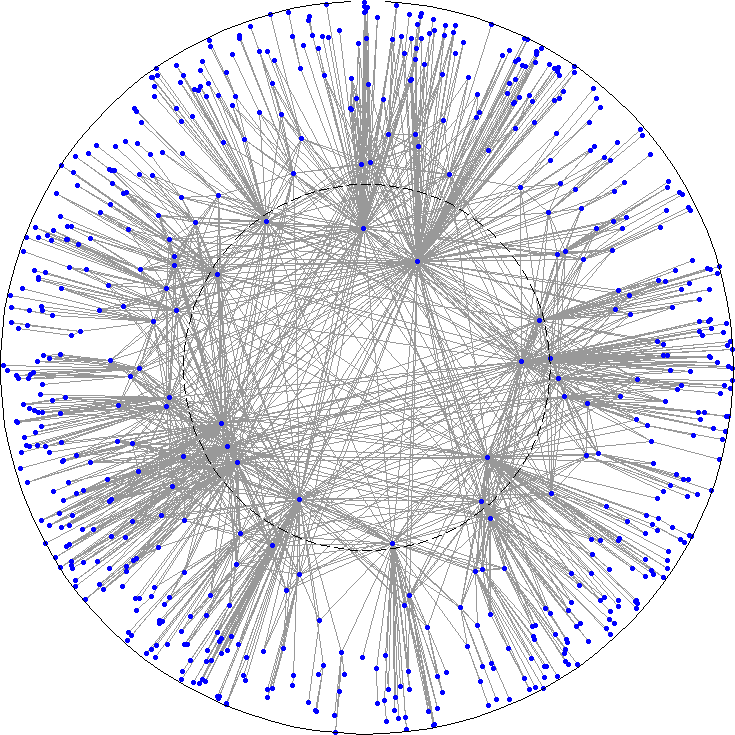}}
\end{center}
\caption{Native representation of the largest connected component 
  (with $621$ vertices) of an instance of 
  $\unifmod_{\alpha,C}(n)$ with $\alpha=0.55$, $C=2.25$ and $n=740$.
The solid (respectively, segmented) circle is the boundary of 
  $B_{O}(R)$ (respectively, $B_{O}(\frac{R}{2})$).}\label{fig:hyperGraph}
\end{figure}

\medskip
The adjacency, Laplacian, and normalized Laplacian
  are three well-known matrices associated to a graph, 
  all of whose spectrum encode important information 
  related to fundamental graph parameters.
For non-regular graphs, such as a random hyperbolic graph $G=(V,E)$
  obtained from $\unifmod_{\alpha,C}(n)$, 
  arguably the most relevant associated matrix is the 
  normalized Laplacian $\calL_G$. Note
  that $\calL_G$ is positive semidefinite and has 
  smallest eigenvalue $0$.
Certainly, the most interesting parameter of 
  $\calL_G$ is its eigenvalue gap $\lambda_1(G)$.
Since for $\frac12<\alpha<1$,
  a typical occurrence of $G$ has $\Theta(|V|)$ isolated vertices,
  the eigenvalue $0$ of $G$ has high multiplicity
  and thus $\lambda_1(G)=0$.
On the other hand, it is known that for the aforesaid range 
  of $\alpha$, most likely the 
  graph $G$ has a component of linear order~\cite[Theorem~1.4]{BFM15} (see also Theorem~\ref{thm:volTobias1} and Corollary~\ref{cor:volTobias} below) 
  and all other components are of
  polylogarithmic order~\cite[Corollary 13]{km15}, 
  which justifies referring to the
  linear size component as \emph{the giant component}.
Thus, the most basic non-trivial question about the 
  spectrum of random hyperbolic graphs is to determine the 
  spectral gap of their giant component.
Implicit in the proof of~\cite[Theorem 1.4]{BFM15} (once more, see also Theorem~\ref{thm:volTobias1} and Corollary~\ref{cor:volTobias} below) is that the giant component of a
  random hyperbolic graph $G$ is the one
  that contains all vertices whose radial coordinates are at most $\frac{R}{2}$,
  which we onward refer to as the \emph{center component} of the
  hyperbolic graph and denote by $H:=H(G)$.

\medskip
The preceding discussion motivates our study of the magnitude of the 
  second largest eigenvalue $\lambda_1=\lambda_1(H)$ of 
  the normalized Laplacian matrix $\calL_{H}$ of the center
  component $H$ of $G$ chosen according to $\unifmod_{\alpha,C}(n)$.
Formally, denoting by $d(v)$ 
  the degree of $v$ in $G$ (which equals $v$'s degree in $H$), 
  the \emph{normalized Laplacian of $H$} is 
  the (square) matrix $\calL_H$ whose rows and columns 
  are indexed by the vertex set of $H$ and whose $(u,v)$-entry takes the value
\[
\calL_H(u,v) := 
\begin{cases}
1, & \text{if $u=v$}, \\
-\displaystyle\frac{1}{\sqrt{d(u)d(v)}}, 
   & \text{if $uv$ is an edge of $H$}, \\
0, & \text{otherwise}.
\end{cases}
\]
Alternatively, $\calL_H := I-D_H^{-1/2}A_HD_H^{-1/2}$, where 
  $A_H$ denotes the adjacency matrix of $H$ and $D_H$ is 
  the diagonal matrix whose $(v,v)$-entry equals~$d(v)$.
It is well known that $\calL_H$ is positive semi-definite and 
  its smallest eigenvalue equals $0$ with geometric multiplicity $1$
  (given that $H$ is by definition connected).
Note that the stochastic matrix associated to the random walk 
  in $H$ is $P_{H}:=D_{H}^{-1}A_{H}=D_{H}^{-1/2}(I-\calL_{H})D_{H}^{1/2}$.
Hence, results concerning the spectra of $\calL_H$ easily translate
  into results about the spectra of $P_{H}$ and thence has implications
  concerning the rate of convergence towards the stationary distribution
  of such random walks and related Markov processes.

\begin{figure}[ht]
\begin{center}
\scalebox{0.75}{\includegraphics{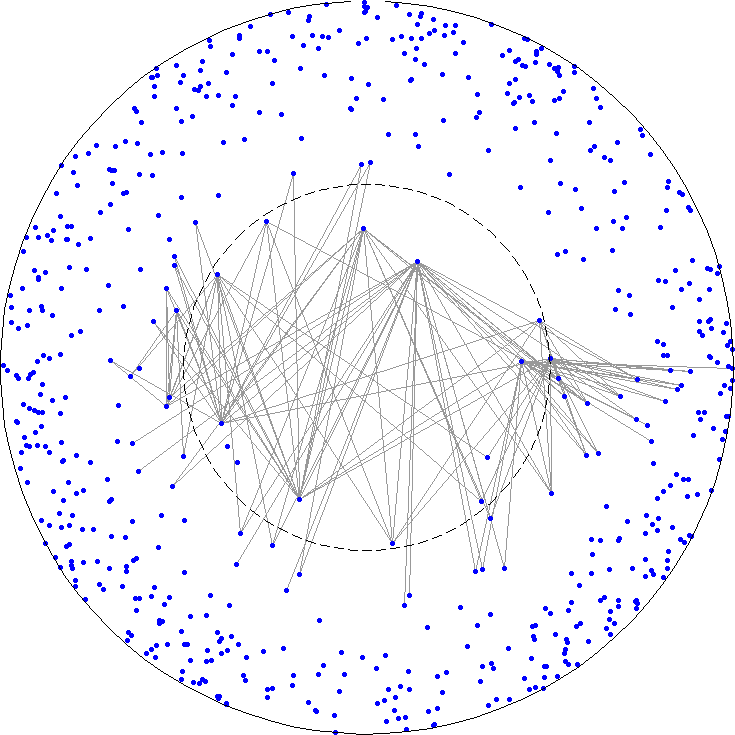}}
\end{center}
\caption{Cut induced in the graph of Figure~\ref{fig:hyperGraph} by 
  vertices of polar coordinate between $0$ and~$\pi$ (angles 
  measured relative to a horizontal axis passing through $\HH^2$'s origin).}\label{fig:cut}
\end{figure}

One often used approach for bounding 
 $\lambda_1(H)$ for a connected graph $H=(U,F)$ is via the so called Cheeger inequality.
To explain this, recall that for $S\subseteq U$,
  the \emph{volume} of $S$, denoted $\Vol(S)$,
  is defined as the sum of the degrees of 
  the vertices in $S$, i.e., $\Vol(S):=\sum_{v\in S}d(v)$.
Also, recall that the \emph{cut} induced by $S$ in $H$, denoted by
  $\partial S$, 
  is the set of graph edges with exactly one endvertex 
  in $S$, i.e., 
  $\partial S:=\{uv \in F: |\{u,v\}\cap S|=1\}$
  (see Figure~\ref{fig:cut}).
The \emph{conductance} of $S$ in $H$, 
  $\emptyset \subsetneq S\subsetneq U$, is defined as 
\begin{equation}\label{eqn:setConductance}
h(S) := \frac{|\partial S|}{\min\{\Vol(S),\Vol(U\setminus S)\}},
\end{equation}
and the \emph{conductance} of $H$ is
 $h(H) := \min\big\{ h(S) : \emptyset \subsetneq S \subsetneq U \}$. 
Cheeger's inequality (see e.g.~\cite[\S 2.3]{chung97})
  states that for an arbitrary connected graph~$G$, 
\begin{equation}\label{eqn:cheeger}
\frac{1}{2}h^2(G)\leq \lambda_1(G) \leq 2h(G),
\end{equation}
and often provides an effective way for bounding the eigenvalue gap
  of graphs.
Our main result gives a stronger characterization of $\lambda_1(H)$ than the 
  one obtained through Cheeger's inequality.
In fact, we show that $\lambda_1(H)$ 
  essentially matches the upper bound given by~\eqref{eqn:cheeger},
  i.e., $\lambda_1(H)$ equals $h(H)$ up to a small polylogarithmic
  factor.
As a byproduct, we obtain an almost tight bound on the conductance
  of the giant component of  random hyperbolic graphs.

Despite the fact that in the original model of Krioukov et al.~\cite{KPKVB10} $n$ points were chosen uniformly at random, it is from a probabilistic point of view arguably more natural to consider the Poissonized version of this model. Specifically, we consider a Poisson point process on the hyperbolic disk of radius $R$ and denote its point set by $\calP$. The intensity function at polar coordinates $(r,\theta)$ for 
  $0\leq r< R$ and $0 \leq \theta < 2\pi$ is equal to
\[
g(r,\theta) := \delta e^{\frac{R}{2}}f(r,\theta)
\]
with $\delta=e^{-\frac{C}{2}}$.
Throughout the paper we denote this model by $\poimod_{\alpha,C}(n)$.
Note in particular that $\int_{0}^R \int_{0}^{2\pi} g(r,\theta) d\theta dr 
  = \delta e^{\frac{R}{2}}=n$, and thus  $\mathbb{E}|{\calP}|=n.$ 
  The main advantage of defining $\calP$ as a Poisson point process is
  motivated by the following two properties: the number of points of
  $\calP$ that lie in any region $A \cap B_O(R)$ follows a Poisson
  distribution with mean given by $\int_A g(r,\theta) drd\theta=n
  \mu(A \cap B_O(R))$, and the numbers of points of $\calP$ in disjoint
  regions of the hyperbolic plane are independently distributed.
  Moreover, by conditioning $\calP$ upon the event $|{\calP}|=n$, we
  recover the original distribution. Therefore, since
  $\Pr(|{\calP}|=n-k)=\Theta(1/\sqrt n)$ for any $k=O(1)$, any event
  holding in $\calP$ with probability at least $1-o(f_n)$ must hold in
  the original setup with probability at least $1-o(f_n \sqrt n)$, and
  in particular, any event holding with probability at least
  $1-o(1/\sqrt{n})$ holds a.a.s.~in the
  original model. Also, an event holding w.e.p.~in 
  $\poimod_{\alpha,C}(n)$ also holds w.e.p.~in $\unifmod_{\alpha,C}(n)$.  
Henceforth, unless stated otherwise, our results will be presented in the Poissonized model only; the corresponding results for the uniform model follow by the above considerations.

\bigskip
\noindent\textbf{Notation.}
All asymptotic notation in this paper is respect to $n$. Expressions given in terms of other variables such as $O({R})$ are still asymptotics with respect to $n$, since these variables still depend on $n$.
We say that an event holds \emph{asymptotically almost surely (a.a.s.)}, if it holds with probability tending to $1$ as $n \to \infty.$ 
We say that an event holds \emph{with extremely high probability
    (w.e.p.)}, if for a fixed (but arbitrary)
  constant $C' > 0$, there exists an $n_0:=n_{0}(C')$
  such that for every $n\geq n_0$ the event holds with
  probability at least $1-n^{-C'}$. 
Throughout the paper, denote by $\upsilon:=\upsilon(n)$ a function 
  tending to infinity arbitrarily slowly with $n$.
  By a union bound, we get that the union of polynomially (in $n$) many events that hold w.e.p.~is also an event that holds w.e.p.  
  For $N\in\NN$, we denote the set 
  $\{1,\ldots,N\}$ by $[N]$.
For a graph $G=(V,E)$ with $S, S' \subseteq V$ and $S \cap S'=\emptyset$, we denote by $E(S,S')$ the set of edges having one endvertex in $S$, and one endvertex in $S'$. For $v \in V$, we refer to the neighborhood of $v$ inside $S$ by $N_{S}(v)$, i.e., $N_{S}(v)=\{w \in S: vw \in E\}$.
Finally, 
  we will often consider a subset $S$ of vertices of a connected component 
  of a given graph in which case $\overline{S}$ will denote 
  its complement with respect to the vertex set of the component.

\subsection{Main contributions}
The following theorem is the main result of this paper.
It bounds from below the spectral gap of random hyperbolic graphs.
\begin{theorem}\label{th:main}
If $H$ is the center component of 
  $G$ chosen according to $\poimod_{\alpha,C}(n)$
  and $D(H)$ denotes the diameter of $H$, then 
  w.e.p., 
\[
\lambda_1(H) = \Omega(n^{-(2\alpha-1)}/D).
\]
\end{theorem}
We also have the following complementary result. 
We remark that a similar upper bound, slightly less precise 
  but in the more general setup of geometric inhomogeneous random 
  graphs, was obtained in~\cite{BKL1}. 
\begin{lemma}\label{lem:upperBound}
Let $H=(U,F)$ be the center component of 
  $G=(V,E)$ chosen according to $\poimod_{\alpha,C}(n)$ or $\unifmod_{\alpha,C}(n)$.
Then, a.a.s.~$h(H)\le \upsilon n^{-(2\alpha-1)}\log n.$
\end{lemma}

Whereas Theorem~\ref{th:main} gives a global lower bound on the conductance of a random hyperbolic graph, we obtain additional information from the next theorem. By classifying subsets of vertices according to their structure and their volume, we can show the following theorem:
\begin{theorem}\label{thm:conductancia}
Let $H=(U,F)$ be the center component of 
  $G=(V,E)$ chosen according to $\poimod_{\alpha,C}(n)$, and let 
  $0 < \varepsilon < 1.$ 
W.e.p., for every set $S \subseteq U$ with 
  $\Vol(S)=O(n^{\varepsilon})$, we have $h(S) =\Omega(n^{-(2\alpha-1)\varepsilon+o(1)}).$
\end{theorem}


  
  We also obtain the following corollary regarding minimum and maximum sizes of arbitrary bisectors (recall that a bisection of a graph is a bipartition of its vertex 
  set in which the number of vertices in the two parts differ by at most $1$, and 
  its size is the number of edges which go across the two parts):
\begin{cor}\label{cor:cut}
Let $H=(U,F)$ be the giant component of
  $G=(V,E)$ chosen according to $\poimod_{\alpha,C}(n)$. 
Then, the following statements hold:
\begin{itemize}[(i)]
\item\label{it:bisection1}
W.e.p., the minimum bisection of $H$ is 
   $b(H)=\Omega\big(n^{2(1-\alpha)}/D)$, where $D:=D(H)$ is the 
   diameter of $H$.
\item\label{it:bisection2}   For any $\xi> 0$,
  with probability at least $1-o(n^{-1+\xi})$, the maximum bisection of $H$ is $B(H)=\Theta(n)$.
\end{itemize}
\end{cor}

\subsection{Related work}
Although the random hyperbolic graph model was relatively
  recently introduced~\cite{KPKVB10}, 
  quite a few papers have analyzed
  several of its properties.
However, none of them deals
  with the spectral gap of these graphs. In~\cite{GPP12}, the degree distribution, the maximum degree and global 
  clustering coefficient were determined.
The already mentioned paper of~\cite{BFM15} characterized the existence of
  a giant component as a function of $\alpha$; very recently, more
  precise results including a law of large numbers for the largest
  component in these networks was established
  in~\cite{FMLaw}. 
The threshold in terms of $\alpha$ for the connectivity of random
  hyperbolic graphs was given in~\cite{BFM13b}. 
Concerning diameter and graph distances,
  except for the aforementioned papers of~\cite{km15} and~\cite{fk15},
  the average distance of two points belonging to the giant component
  was investigated in~\cite{ABF}. 
Results on the global clustering coefficient of the so called
  binomial model of random hyperbolic graphs were obtained
  in~\cite{CF13}, and on the evolution of graphs on more general
  spaces with negative curvature in~\cite{F12}.

The model of random hyperbolic graphs in the regime where $\frac12 < \alpha < 1$, is very similar to two different models
studied in the literature: the model of inhomogeneous long-range
percolation in $\mathbb{Z}^d$ as defined in~\cite{Remco}, and the
model of geometric inhomogeneous random graphs, as introduced
in~\cite{BKL2}. In both cases, each vertex is given a weight, and
conditionally on the weights, the edges are independent (the presence
of edges depending on one or more parameters). In~\cite{Remco} the
degree distribution, the existence of an infinite component and the
graph distance between remote pairs of vertices in the model of
inhomogeneous long-range percolation are analyzed. On the other hand,
results on typical distances, diameter, clustering coefficient,
separators, and existence of a 
  giant component in the model of geometric
inhomogeneous graphs were given in~\cite{BKL1,BKL2}, and bootstrap
percolation in the same model was studied in~\cite{KL}. Both models
are very similar to each other, and similar results were obtained in
both cases; since the latter model assumes vertices in a toroidal
space, it generalizes random hyperbolic graphs.  

\subsection{Organization}
In Section~\ref{sec:overview}, we give an overview of the general proof strategy. In Section~\ref{sec:prelim}, we collect some known general useful
  results and establish a couple of 
  new ones  concerning random hyperbolic graphs
  that we later rely on.
In Section~\ref{sec:gap}, we determine up to polylogarithmic factors 
  both the conductance and
  the eigenvalue gap of the giant component of
  random hyperbolic graphs.
In Section~\ref{sec:conductance}, we essentially show that only linear
  size vertex sets $S$ of the giant component of random hyperbolic graphs
  can induce ``small bottlenecks''  measured in terms of 
  conductance, i.e., if $h(S)$
  is approximately
  equal to the conductance of the giant component $H$, then $S$ must contain essentially
  a constant fraction of $H$'s vertices.
In Section~\ref{sec:bisec}, we derive results concerning related graph 
  parameters such as minimum and maximum bisection as well as maximum 
  cuts of random hyperbolic graphs.
Finally, in Section~\ref{sec:conclusion}, we discuss some questions our 
  result naturally raises as well as possible future research directions.



\section{Overview of the proof of the main theorems}\label{sec:overview}
The proof of Theorem~\ref{th:main} is based on the so
  called multicommodity flow method.
Specifically, it is based on the fact that $\lambda_1(H)$ can by its variational 
  characterization be bounded from below as a function of a suitably 
  defined multicommodity flow defined on $H$. 
Roughly speaking, we aim for finding a flow between all pairs of
  vertices consisting of not too long paths, and moreover these paths
  are defined in such a way that no single edge has too much flow
  going through it. 
We point out that the classical canonical path technique of routing
  the flow through one single path cannot give the claimed result, hence
  we have to split the flow through different edges. 
Our main task therefore consists in defining such a flow by
  exploiting properties of the hyperbolic model. 
In a nutshell, for pairs of vertices ``close'' to the center we
  route the flow evenly through paths of length $3$ all of whose
  vertices are also relatively close to the center.
We then extend the flow to pairs of vertices where at least one
  vertex is ``far'' from the center by attaching a 
  ``shortest'' path from
  each such vertex into the center area; from there on the same
  strategy of length $3$ paths as before is applied. 
A crucial ingredient on which the analysis relies 
  concerns properties of the mentioned ``shortest'' paths
  implied by the metric of the underlying hyperbolic space.
The corresponding
  upper bound of Lemma~\ref{lem:upperBound} is easier, by Cheeger's
  inequality it is enough to find an
  upper bound on the conductance of $H$. The latter can be obtained by
  considering the set of vertices of $H$ that belong to a half
    disk of
  $B_O({R})$.

In order to obtain Theorem~\ref{thm:conductancia}, we decompose the graph in a way that takes into account the underlying geometry.
Informally said, the decomposition establishes the existence
  of regions $\calR$ of $B_O({R})$ such that for sets of vertices $S$ whose volume is $O(n^\varepsilon)$ for some $0 < \epsilon < 1$ the following holds: 
  \begin{inparaenum}[(i).-]
  \item $\calR$ covers a significant fraction of the edges 
  incident to $S$, and 
  \item the fraction of vertices of $\calR$ that belong
  to $S\cap\calR$ and to $\overline{S}\cap\calR$ are both non-trivial,
  or both $\Vol(S\cap\calR)$ and $\Vol(\overline{S}\cap\calR)$ are 
  a non-trivial fraction of $\Vol(\calP\cap\calR)$. In either case, the number of cut edges of $\partial S$ within $\calR$ is relatively 
  large. 
  \end{inparaenum}
The main task is to classify sets $S$ according to their shape so that corresponding regions $\calR$ can be found.

Additional technical contributions are derived in the process of establishing both theorems.
We show that w.e.p.~the volume of $H$ is linear in $n$, and that moreover, the volume of a not too small sector is w.e.p.~at most proportional to its angle, provided that inside the sector there is no vertex very close to the origin (see Lemma~\ref{lem:VolUpperBound} for details). 
Whereas this result is not surprising, we hope that it will turn out to be useful in other applications as well.


\section{Preliminaries}\label{sec:prelim}
In this section we collect some of the known properties as well as derive some additional ones concerning 
  random hyperbolic graph model.
We also state for future reference some known approximations 
  for different terms concerning distances,
  angles, and measure estimates that are useful in their study.

\medskip
By the hyperbolic law of cosines~\eqref{eqn:coshLaw}, 
  the hyperbolic triangle formed by the geodesics 
  between points $p'$, $p''$, and  $p$, with opposing side segments of length 
  $\distp$, $\distpp$, and $\dist$ respectively,
  is such that the angle formed at $p$ is:
\begin{equation}\label{eqn:angle}
\theta_{\dist}(\distp,\distpp) = 
\arccos\Big(\frac{\cosh \distp\cosh \distpp-\cosh \dist}{\sinh \distp\sinh \distpp}\Big).  
\end{equation}
Clearly, $\theta_{\dist}(\distp,\distpp) = \theta_{\dist}(\distpp,\distp)$. 
Next, we state a very handy approximation for $\theta_{\dist}(\cdot,\cdot)$.
\begin{lemma}[{\cite[Lemma 3.1]{GPP12}}]\label{lem:aproxAngle}
If $0\leq\min\{\distp,\distpp\}\leq \dist\leq \distp+\distpp$, then
\[
\theta_{\dist}(\distp,\distpp) = 2e^{\frac{1}{2}(\dist-\distp-\distpp)}\big(1+\Theta(e^{\dist-\distp-\distpp})\big).
\]
\end{lemma}
\begin{remark}\label{rem:aproxAngle}
We will use the previous lemma also in this form: let
  $p'$ and $p''$ be two points at distance $\dist$ from each other
  such that $r_{p'},r_{p''} > \frac{R}{2}$ and $\min\{r_{p'},r_{p''}\} \leq \dist \leq R$.
Then, taking $\distp=r_{p'}$
  and $\distpp=r_{p''}$ in Lemma~\ref{lem:aproxAngle}, we get
\[
\theta_{\dist}(r_{p'},r_{p''}) 
  := 2e^{\frac{1}{2}(\dist - r_{p'} - r_{p''})}\big(1+\Theta(e^{\dist-r_{p'}-r_{p''}})\big).
\]
Note also that $\theta_{\dist}(r_{p'},r_{p''})$, for fixed $r_{p'},r_{p''} > \frac{R}{2}$, 
  is increasing as a function of $\dist$ (for $\dist$ satisfying the constraints). 
Below, when aiming for an upper bound, we always use $\dist=R$.
\end{remark}

\medskip
Throughout, we will need estimates for measures of 
  regions of the hyperbolic plane, and more specifically, for regions obtained by performing some set algebra
  involving a few balls.
For a point $p$ of the hyperbolic plane $\HH^2$, 
  the ball of radius $\rho$ centered at $p$ will be denoted by
  $B_{p}(\rho)$, i.e., $B_{p}(\rho) := \{q\in\HH^2 : \dist(p,q)\leq \rho\}$.

Also, we denote by $\mu(S)$ the measure of a set 
  $S \subseteq \HH^2$, i.e.,
\[
\mu(S) := \int_{S}f(r,\theta)drd\theta.
\]
Next, we collect a few results for such measures.
\begin{lemma}[{\cite[Lemma~3.2]{GPP12}}]\label{lem:muBall} 
If $0\leq \rho< R$, then
\begin{equation}\label{eq:muBall1}
\mu(B_{O}(\rho)) = e^{-\alpha(R-\rho)}(1+o(1)).
\end{equation}
Moreover, if $p\in B_{O}(R)$ is such that $r_p=r$, then for 
  $C_{\alpha}:=2\alpha/(\pi(\alpha-\frac{1}{2}))$,
\begin{equation}\label{eq:muBall2}
\mu(B_{p}(R)\cap B_{O}(R)) = C_{\alpha}e^{-\frac{r}{2}}\big(1+O(e^{-(\alpha-\frac{1}{2})r}+e^{-r})\big).
\end{equation}
\end{lemma}
A direct consequence of~\eqref{eq:muBall1} is:
\begin{cor}\label{cor:medidalb} 
If $0 \leq \rho'_O < \rho_O < R$, then
\[
\mu(B_{O}(\rho_{O})\setminus B_{O}(\rho'_{O}))
  = e^{-\alpha(R-\rho_{O})}(1-e^{-\alpha(\rho_{O}-\rho'_{O})}+o(1)).
\]
\end{cor}
Sometimes we will require the following stronger version 
  of~\eqref{eq:muBall2}.
\begin{lemma}[{\cite[Lemma 4]{km15}}]\label{lem:muBallInterGen}
If $r_{p}\leq \rho_{p}$ and $\rho_{O}+r_{p}\geq \rho_{p}$, then for 
    $C_{\alpha}:=2\alpha/(\pi(\alpha-\frac{1}{2}))$
\[
\mu(B_{p}(\rho_{p})\cap B_{O}(\rho_{O})) 
  = C_{\alpha}\big(e^{-\alpha(R-\rho_{O})-\frac{1}{2}(\rho_{O}-\rho_p+r_p)}\big)
        + O(e^{-\alpha(R-\rho_{p}+r_{p})}).
\]
\end{lemma}

\medskip
At several places in this paper we need the following concentration bound.
\begin{theorem}{\cite[Corollary~A.1.14]{as08}}\label{thm:Chernoff}
Let $Y$ be the sum of mutually independent indicator random variables, $\mu=\Ex(Y).$ 
For all $\epsilon > 0$, there is a $c_{\epsilon}>0$ that 
  depends only on $\epsilon$ such that
$$
\Pr(|Y-\mu| > \epsilon \mu) < 2e^{-c_{\epsilon}\mu}.
$$
\end{theorem}
For Poisson variables, we also need the following slightly stronger bound:
\begin{theorem}{\cite[Theorem~A.1.15]{as08}}\label{thm:AlonSpencer}
Let $P$ have Poisson distribution with mean $\mu$. For $0 < \epsilon < 1$, 
$$
\begin{array}{rcl}
\Pr(P \le \mu(1-\epsilon)) & \le & e^{-\epsilon^2 \mu/2},
\end{array}
$$
and for $\epsilon > 0$,
$$
\begin{array}{rcl}
\Pr(P \ge \mu(1+\epsilon)) & \le & \left( e^{-\epsilon}(1+\epsilon)^{-(1+\epsilon))}\right)^{\mu}.
\end{array}
$$
\end{theorem}

We immediately derive the following lemma:%
\begin{lemma}\label{lem:meanMeus}
Let $\calP$ be the vertex set of a graph chosen according to $\poimod_{\alpha,C}(n)$.
If $S\subseteq B_{O}(R)$ is such that $\mu(S) = \omega(\frac{1}{n}\log n)$, 
  then, w.e.p.~$|S\cap\calP|=n \mu(S) (1+o(1)).$
Otherwise, w.e.p.~$|S\cap\calP|\leq\upsilon\log n$.
\end{lemma}
Many of the proof arguments we will later put forth involve statements
  concerning sectors of the hyperbolic disk $B_{O}(R)$, in particular, 
  their size and volume.
The next two lemmas provide estimates for such quantities.
We first provide estimates 
  for the degree of vertices of~$G$ as a function of their radius.

\medskip 
Throughout the paper let $\nu':=2\log R+\omega(1) \cap o(\log R)$.

\begin{proposition}\label{prop:degreeBound}
Let $v$ be a vertex  of $G$ chosen according to $\poimod_{\alpha,C}(n)$. 
If $r_v\leq R-\nu'$, then 
  w.e.p.~$d(v)=\Theta(e^{\frac{1}{2}(R-r_v)})$, and if 
  $r_v> R-\nu'$, then w.e.p.~$d(v) \leq (\log n)^{1+o(1)}$. 
\end{proposition}
\begin{proof}
Assume first that $r_v\leq R-\nu'$. 
Note that $d(v)=|B_{v}(R)\cap \calP|$.
Since by~Lemma~\ref{lem:muBall} we have
  $\mu(B_{v}(R)\cap B_{O}(R))=\Theta(e^{-\frac{r_v}{2}})
     =\omega(\frac{\log n}{n})$, by~Lemma~\ref{lem:meanMeus} the first part of the claim follows.  
If $r_v\geq R-\nu'$, then
$\mu(B_v({R}) \cap B_O({R}))$ is bounded from above by 
  $\mu(B_w({R}) \cap B_O({R}))$ where $w$ is a point of $B_O({R})$ with $r_w=R-\nu'$. We have  
$\mu(B_w({R}) \cap B_O({R})) = \Theta(e^{\frac{\nu'}{2}}/n)=\omega(\frac{\log n}{n})$, and hence by~Lemma~\ref{lem:meanMeus}, w.e.p.
$d(v) \le n\mu(B_w({R}) \cap B_O({R})) = O(e^{\frac{\nu'}{2}})=(\log n)^{1+o(1)}.$
\end{proof}

  When working with a Poisson point process $\calP$, for a positive 
  integer $\ell$, we refer to the vertices of $G$ that belong to 
  $B_{O}(\ell)\setminus B_{O}(\ell-1)$ as the $\ell$-th \emph{band} or \emph{layer} and 
  denote it by $\calP_{\ell}:=\calP_{\ell}(G)$, i.e.,  
  $\calP_{\ell}=\calP(G)\cap B_{O}(\ell)\setminus B_{O}(\ell-1)$.
We also need estimates for the cardinality and the volume of the $\calP_\ell$'s.

Since our results are asymptotic, we may and will ignore floors in the following calculations, and assume that certain expressions such as $R-\frac{2\log R}{1-\alpha}$ or the like are integers, if needed.
In what follows, also let
  \begin{align*}
  \elllow & := \big\lfloor (1-\tfrac{1}{2\alpha})R\big\rfloor, \\
  \nu & :=\tfrac{1}{\alpha}\log R+\omega(1) \cap o(\log R).
  \end{align*}
\begin{proposition}\label{prop:bandBound}
  Let $G=(V,E)$ be chosen according
  to $\poimod_{\alpha,C}(n)$ and let $\calP_\ell:=\calP_{\ell}(G)$.
\begin{enumerate}[(i).-]
\item\label{prop:bandBoundPart1}
If $\ell\geq \elllow+\nu$,
  then w.e.p.~$|\calP_{\ell}|=\Theta(ne^{-\alpha(R-\ell)})$.
Moreover, if $\ell<\elllow+\nu$, then 
  w.e.p.~$|\calP_{\ell}|=O(e^{\alpha\nu})=(\log n)^{1+o(1)}$. 

\item\label{prop:bandBoundPart2}
If $\elllow+\nu\leq \ell\leq R-\nu'$, then 
  w.e.p.~$\Vol(\calP_{\ell})=\Theta(e^{\frac{1}{2}R-(\alpha-\frac{1}{2})(R-\ell)})$.
\end{enumerate}
\end{proposition}
\begin{proof}
Note that $e^{\alpha\nu}=(\log n)^{1+o(1)}\cap\omega(\log n)$.
Consider the first part of the claim.
By~Lemma~\ref{lem:muBall} we have
  $\mu(B_{O}(\ell)\setminus B_{O}(\ell-1))=e^{-\alpha(R-\ell)}(1-e^{-\alpha})(1+o(1))$,
  which is $\omega(\frac{\log n}{n})$ if $\ell\geq\elllow+\nu$,
  so the result follows by applying Lemma~\ref{lem:meanMeus}.
Assume now that $\ell<\elllow+\nu$.
By~Lemma~\ref{lem:muBall} we have  
  that $\mu(B_{O}(\elllow+\nu))
  = e^{-\alpha (R-(\elllow+\nu))}(1+o(1))
  = \Theta(\frac{e^{\alpha\nu}}{n})=\omega(\frac{\log n}{n})$, so applying again
  Lemma~\ref{lem:meanMeus},
  w.e.p., $|\calP_\ell|\leq |\calP\cap B_{O}(\elllow+\nu)|
      =O(e^{\alpha\nu}).$
  
Since $\Vol(\calP_\ell)=\sum_{v \in \calP_\ell} d(v)$, and for each such vertex $v$, by Proposition~\ref{prop:degreeBound}, its degree is, w.e.p., 
  $\Theta(e^{\frac12(R-r_v)})$, the second part of the claim then follows
  easily from the first part. 
\end{proof}
  
Since the introduction of the random hyperbolic 
  graph model~\cite{KPKVB10}, it was pointed out that it gives
  rise to sparse networks, specifically constant average degree
  graphs (a fact that was soon after rigorously established in~\cite{GPP12}).
It follows that the expected volume of 
  random hyperbolic graphs is $\Theta(n)$, and thus their center component
  has, in expectation, volume $O(n)$.
A close inspection of~\cite{BFM15} (see Theorem~\ref{thm:volTobias1} below) 
  actually yields that the volume of 
  the center component is $\Omega(n)$ w.e.p.
In this paper, we aim for results that hold w.e.p.~and will require
  very sharp estimates not only for the 
  volume of the center component of random hyperbolic graphs but also
  for collections of vertices restricted to some regions of $B_{O}(R)$.
Next, we describe the regions we will be concerned about. Let $\Phi$ be a $\phi$-sector, that is, 
  $\Phi$ contains all points in 
  $B_O(R)$ making an angle of at most $\phi$ at the origin with an arbitrary 
  but fixed reference point. For a vertex $v$, we say that a $\phi$-sector 
  $\Phi$ is centered at $v$ if
  $v$ lies on the bisector of $\Phi$. 
Moreover, for a $\phi$-sector $\Phi$ and a 
  vertex $v\in\Phi$, we say that $\Upsilon:=\Phi\setminus B_{O}(r_v)$
  is a \emph{sector truncated at $v$}, and 
  if in addition $\Phi$ is centered at $v$, then we say 
  it is a \emph{sector truncated and centered at $v$}.
Our next result gives precise estimates for the volume of the center 
  component vertices that belong to sectors and truncated sectors. 
Although the result is not surprising we believe it is useful to isolate
  it not only for ease of reference later in this work, but also for
  reference in follow up work. 
However, we suggest the reader skip the proof at first reading,
  due to its rather technical nature.
\begin{lemma}\label{lem:VolUpperBound}
Let $H=(U,F)$ be the center component of $G=(V,E)$ chosen according to 
  $\poimod_{\alpha,C}(n)$.
Then, w.e.p.~$\Vol(U)=O(n)$.
Moreover, let $v\in\calP_{\ell}$ be  
  such that $\ell\le (1-\xi)R$ 
  for some arbitrarily small $\xi > 0$.
If $\Upsilon$ is a sector truncated at $v$
  of angle $\phi=\Omega(e^{-\frac{\ell}{2}})$, 
  then w.e.p.~$\Vol(\Upsilon)=O(\phi n)$.
\end{lemma}
\begin{proof}
Consider the first part of the lemma.
Let $\varepsilon'=\varepsilon'(\alpha) > 0$ be a sufficiently small constant and let $r_0=(1-\frac{1}{2\alpha}-\varepsilon')R$. By Lemma~\ref{lem:muBall},
  $\mu(B_O(r_0))=(1+o(1))e^{-\alpha(R-r_0)}=\Theta(e^{-(\frac12+\alpha\varepsilon')R})$. 
Hence, 
$|\calP\cap B_O(r_0)|$ is a Poisson random variable with mean 
  $t=\Theta(n^{-2\alpha \varepsilon'})$. 
Thus, by Theorem~\ref{thm:AlonSpencer}, for every $C'>0$ there exists 
  a sufficiently large constant $C''=C''(\alpha)>0$, so that 
\[
\Pr\big(|\calP\cap B_O(r_0)| \ge \tfrac{C''}{t} \EE(|\calP\cap B_O(r_0)|)\big) \le \big(\tfrac{3C''}{t}\big)^{-C''}=\Theta(n^{-2\alpha \varepsilon' C''}) \le n^{-C'}.
\]
Hence, w.e.p.~$|\calP\cap B_O(r_0)|  \le C''=O(1)$.
Thus, by Proposition~\ref{prop:degreeBound}, 
  w.e.p.~$\Vol(\calP\cap B_O(r_0))=O(n).$ 
Recall that $\nu=\frac{1}{\alpha}\log R+\omega(1) \cap o(\log R)$.
By the same argument, using Corollary~\ref{cor:medidalb} and 
  Proposition~\ref{prop:degreeBound}, the total contribution to the 
  volume of vertices $v$ with $r_0 \le r_v \le \elllow+\nu$ is w.e.p.,
\begin{align*}
& O\big(|\calP\cap B_O(\elllow+\nu)|\max_{v\not\in B_{O}(r_0)} d(v)\big)
  = O(ne^{-\alpha(R-\elllow-\nu)}e^{\frac12(R-r_0)}) 
  = O(n^{\frac{1}{2\alpha}+\varepsilon'}e^{\alpha \nu})  
  = o(n),
\end{align*}
where the last equality follows 
  for sufficiently small $\varepsilon' > 0$, since $\alpha>\frac12$.
Similarly, for vertices $v$ with $\elllow +\nu \le r_v \le R-\nu'$, 
  by Proposition~\ref{prop:degreeBound} and 
  Proposition~\ref{prop:bandBound} part~(\ref{prop:bandBoundPart1}), 
  the total volume of these vertices, using the formula for the sum of a geometric 
  series, is w.e.p., 
\[
\sum_{\ell=\elllow+\nu}^{R-\nu'} O(ne^{-\alpha(R-\ell)}e^{\frac12(R-\ell)})=O(n^{2(1-\alpha)}) \sum_{\ell=\elllow+\nu}^{R-\nu'} e^{(\alpha-\frac12)\ell}=O(ne^{-(\alpha-\frac12)\nu'})=o(n).
\] 
For the remaining volume, we may at the expense of a factor $2$ assume 
  that all remaining edges are incident to pairs of vertices
  in $B_{O}(R)\setminus B_{O}(R-\nu')$.
Fix integers $R-\nu' \le i \le j \le R$ and assume $v\in\calP_{i}$ and $w\in\calP_j$.
Partition $B_{O}(R)$ into 
  $N:=\lceil \frac{2\pi}{\theta_R(i-1,j-1)}\rceil$ sectors
  denoted (in clockwise order) by 
  $\Phi_1, \Phi_2, \ldots,\Phi_{N}$.
Observe that if $vw$ is an edge of $G$
  then $(v,w)$ besides belonging to $\calP_i\times\calP_j$ 
  also belongs to $\Phi_k\times\Phi_{k'}$ for some $k,k'\in [N]$ where
  $|k-k'|\leq 1$.
For given $i,j$, let $\mu_j^i=\frac{1}{N}\EE|\calP_j|$ and $\mu_i^j=\frac{1}{N}\EE|\calP_i|$.
For an integer $c\geq 1$, for either $b=i$ and $a=j$, or $b=j$ and $a=i$,
  say $\Phi_k\cap\calP_b$ is \emph{$c$-regular} if 
  $2^c\mu_{b}^{a}\le |\Phi_k\cap\calP_b|\le 2^{c+1}\mu_b^{a}$. 
Note that $\mu_b^a=\Theta(n^{2(1-\alpha)}e^{(\alpha-\frac12)b-\frac12 a})$
  and by Theorem~\ref{thm:AlonSpencer}, $\Phi_k\cap\calP_b$
  is $c$-regular with probability $e^{-\Omega(c 2^{c}\mu_b^a)}$. 
For any ordered pair $(i,j)$ and any $a,b$ as before we have 
  $(\frac{1}{\log n})^{O(1)}\le \mu_b^a \le (\log n)^{O(1)}$.
Hence, w.e.p., for every $b$ and every $k$, 
$|\Phi_k \cap \calP_b| = (\log n)^{O(1)}$.

Let $c, \widetilde{c} \ge 1$ be integers. 
In expectation, for $i < j$, there are 
 $Ne^{-\Omega(c2^{c}\mu_i^j+\widetilde{c}2^{\widetilde{c}}\mu_j^i)}$
  pairs of sectors $(\Phi_k, \Phi_{k'})$ with $|k-k'|\leq 1$ such that 
  $\Phi_k\cap\calP_i$ is $c$-regular and $\Phi_{k'}\cap\calP_j$
  is $\widetilde{c}$-regular. 
Clearly, for a fixed value of $k-k'$ $\in \{-1,0,1\}$, disjoint pairs
  of sectors $(\Phi_k, \Phi_{k'})$ are independent. 
Hence, if this expectation is $\omega(\log n)$, 
  by Theorem~\ref{thm:Chernoff}, for $i < j$, w.e.p.~there are $2Ne^{-\Omega(c2^{c}\mu_i^j+\widetilde{c}2^{\widetilde{c}}\mu_j^i)}$
  such pairs of sectors $(\Phi_k, \Phi_{k'})$, and this also holds after taking a union bound over the three possible
  values of $k-k'$. Otherwise, if the expectation is $O(\log n)$, then
  w.e.p., by Theorem~\ref{thm:Chernoff}, the number of such pairs is 
  at most $\upsilon \log n$, 
  and since for every $k$ and $b$, we have w.e.p. $|\Phi_k \cap \calP_b| =
  (\log n)^{O(1)}$, 
  the total number of edges between such pairs of
  sectors is w.e.p.~$(\log n)^{O(1)}$.
Similarly, w.e.p., there are $2Ne^{-\Omega(c2^{c}\mu_i^j)}$
  pairs of sectors $(\Phi_k, \Phi_{k'})$ with $|k-k'|\leq 1$ such that 
  $\Phi_k\cap\calP_i$ is $c$-regular
  and $|\Phi_{k'}\cap\calP_j|\leq 2\mu_j^i$ 
  or the expected number of such pairs of sectors is $O(\log n)$, 
  and as before, the number of edges between such pairs of sectors 
  is w.e.p.~$(\log n)^{O(1)}$. 
A similar argument suffices for handling the 
  case of pairs of sectors $(\Phi_k, \Phi_{k'})$ with $|k-k'|\leq 1$ such that 
  $|\Phi_k\cap\calP_i|\leq 2\mu_{i}^j$
  and $\Phi_{k'}\cap\calP_j$ is $\widetilde{c}$-regular.
For the remaining pairs of sectors $(\Phi_k, \Phi_{k'})$ we have 
  $|\Phi_k\cap\calP_j|\le 2\mu_j^i$ and $|\Phi_{k'}\cap\calP_i|\le 2\mu_i^j$.
Hence, for the number of edges between $\calP_i$ and $\calP_j$, we obtain 
  that w.e.p.,
\begin{align*}
& |E(\calP_i, \calP_j)| \le 
  \sum_{\myatop{k,k'\in [N]}{|k-k'|\leq 1}} |E(\Phi_k\cap\calP_i,\Phi_{k'}\cap\calP_j)|
\\  & \mbox{} \ 
  = O(N\mu_i^j\mu_j^i) \big(2^2\!+\!\!\sum_{c \ge 1} 2^{c{+}2}e^{-\Omega(c2^{c}\mu_i^j)}  
  \!+\!\!\sum_{\widetilde{c} \ge 1} 2^{\widetilde{c}+2}e^{-\Omega(\widetilde{c}2^{\widetilde{c}}\mu_j^i)}
  +\!\!\!\!\sum_{c \ge 1, \widetilde{c} \ge 1} \!\!\! 2^{c+\widetilde{c}{+}2}e^{-\Omega(c2^{c}\mu_i^j+\widetilde{c}2^{\widetilde{c}}\mu_j^i)}\big)\!+\!(\log n)^{O(1)}
\\  & \mbox{} \
= O\big(Nn^{4(1-\alpha)}e^{-(1-\alpha)(i+j)}\big)
  \sum_{c\geq 0}2^{c} e^{-\Omega(c2^{c}\mu_i^j)}
  \sum_{\widetilde{c}\geq 0}2^{\widetilde{c}} e^{-\Omega(\widetilde{c}2^{\widetilde{c}}\mu_j^i)}+(\log n)^{O(1)}.
\end{align*}
Now, for $i < j$, observe that since $\alpha <1$, $\mu_j^i=\Omega(1)$, and hence 
  $\sum_{\widetilde{c}\geq 0} 2^{\widetilde{c}} e^{-\Omega(\widetilde{c}2^{\widetilde{c}}\mu_j^i)}=O(1)$. 
On the other hand, let  $c^*=c^*(i,j):=\min\{c\in\NN : 2^{c}\mu_i^j\ge\frac12\}.$ 
Observe that we may ignore 
  values of $c$ smaller than $c^*$, as for 
  such pairs of sectors $(\Phi_k, \Phi_{k'})$ no vertices in 
  $\Phi_{k'}\cap\calP_i$ are present, and hence no edges are counted. 
Then, 
$\sum_{c \ge c^*} 2^{c}e^{-\Omega(c2^{c}\mu_i^j)} \le
2^{c^*}\sum_{c' \ge 0} 2^{c'} e^{-\Omega(c^* 2^{c'})} =O\left((2(1-\delta))^{c^*}\right)$   for some $\delta > 0$.
Thus, w.e.p.,
\begin{align*}
|E(\calP_i, \calP_j)| =O\big(e^{(\alpha-\frac12)(i+j)}n^{3-4\alpha})(2(1-\delta))^{c^*}+(\log n)^{O(1)}.
\end{align*}
The same calculations can also be applied for $i=j$ and $k \neq k'$. For $i=j$ and $k=k'$, edges within the same sector are counted. 
Hence, since $\mu_i^i=\Omega(1)$ and thus $\sum_{c\geq 0} 2^{2c+2}e^{-\Omega(c2^c\mu_i^i)}=O(1)$, we obtain w.e.p.~$|E(\calP_i, \calP_i)|  = O(e^{(2\alpha-1)i}n^{3-4\alpha})+(\log n)^{O(1)}$.
Hence, w.e.p.,
 \begin{align*}
& \sum_{R-\nu' \le i \le j \le R} |E(\calP_i, \calP_j)|= (\log n)^{O(1)}+\sum_{R-\nu' \le i \le j \le R} O\big(e^{(\alpha-\frac12)(i+j)}n^{3-4\alpha}(2(1-\delta))^{c^*}\big).
\end{align*}
Now, in order to bound the second right hand side term,
 write $i=R-\overline{\imath}$, $j=R-\overline{\jmath}$ with $0 \le \overline{\jmath} \le \overline{\imath} \le \nu'$. Observe that since 
  $2^{c^*}=\Theta(1+n^{-2(1-\alpha)}e^{\frac12j - (\alpha-\frac12)i)})=\Theta(1+e^{(\alpha-\frac12)\overline{\imath}-\frac12 \overline{\jmath}})$. 
Consider first pairs $(i,j)$ with $c^*=O(1)$. For such pairs,
  \begin{align*}
& n^{3-4\alpha}\sum_{R-\nu' \le i \le j \le R} e^{(\alpha-\frac12)(i+j)}
 =O(n) \sum_{0 \le \overline{\jmath} \le \overline{\imath} \le \nu'} e^{(\alpha-\frac12)(-\overline{\imath}-\overline{\jmath})}=O(n),
 \end{align*}
 where we used the formula for a geometric series.
Consider then pairs $(i,j)$ with $c^*=\omega(1)$.
For such a pair, $2^{c^*}=\Theta(e^{(\alpha-\frac12)\overline{\imath}-\frac12 \overline{\jmath}})$, we have $(2(1-\delta))^{c^*}=\Theta(e^{(1-\delta')\left((\alpha-\frac12)\overline{\imath}-\frac12 \overline{\jmath}\right)})$ for some $0 < \delta' < 1$.
 Hence, 
\begin{align*}
& n^{3-4\alpha}\sum_{R-\nu' \le i \le j \le R} O\big((2(1-\delta))^{c^*}e^{(\alpha-\frac12)(i+j)}\big)
  =O(n) \sum_{0 \le \overline{\jmath} \le \overline{\imath} \le \nu'} e^{-\alpha \overline{\jmath}}e^{-\delta'((\alpha-\frac12)\overline{\imath}-\frac12\overline{\jmath})}
\\ \mbox{} \qquad
& =O(n)\sum_{0 \le \overline{\imath} \le \nu'} e^{-\delta'(\alpha-\frac12)\overline{\imath}} \sum_{\overline{\jmath}=0}^{\overline{\imath}} e^{(-\alpha+\frac{\delta'}{2})\overline{\jmath}}=O(n) \sum_{0 \le \overline{\imath} \le \nu'}e^{-\delta'(\alpha-\frac12)\overline{\imath}} =O(n),
\end{align*}
where we again used the formula for a geometric series, thus finishing 
  the proof of the first part of the claimed result.

Now, consider the second part of the lemma, and let 
  $v\in\calP_{\ell}$ with $\ell=\lambda R \le (1-\xi)R$ for 
  some arbitrarily small $\xi > 0$. 
Since $\phi=\Omega(n^{-\lambda})$, we may
  partition $\Upsilon$ into $t=\Theta(\frac{\phi}{n^{-\lambda}})$
  subsectors $T_1,\ldots, T_t$ of angle $\Theta(n^{-\lambda})$ and
  bound the volume of each subsector $T_k$ separately. Let
  $\widehat{\lambda}$ be such that
  $1-\lambda-2\alpha(1-\widehat{\lambda})=-\varepsilon'$ for
  sufficiently small $\varepsilon'=\varepsilon'(\alpha) > 0$. Note
  that since $\alpha > \frac12$, for $\varepsilon'$ small enough, we
  have $1>\widehat{\lambda} > \lambda$. 
For a fixed $T_{k}$, consider first vertices $w\in\calP\cap T_k$ 
  with $\ell\le r_w\le\widehat{\ell}:=\lfloor\widehat{\lambda}R\rfloor$. 
Since the expected number
  of vertices of such radius inside $T_k$, by Lemma~\ref{lem:muBall} and the choice of angle for defining 
  $T_k$,
  is 
  $O(n^{-\varepsilon'})$, by
  the same reasoning as in the first part of the lemma, w.e.p.~there
  are $O(1)$ such vertices, and their total volume is, by
  Proposition~\ref{prop:degreeBound}, 
  w.e.p.~$O(1)e^{\frac12(R-\ell)}=O(n^{1-\lambda})$. 
Next, let  $\overline{\lambda}$ be such that
  $1-\lambda-2\alpha(1-\overline{\lambda})=\varepsilon'$. 
Note that $1>\overline{\lambda} >\widehat{\lambda}$ and 
  consider vertices $w \in \calP \cap T_k$ with 
  $\widehat{\ell}\le r_w\le \overline{\ell} := \lfloor \overline{\lambda}R \rfloor$. As in the first part of the lemma, the total contribution of these vertices to the volume of $T_k$ is, w.e.p.,
  $$
  O(e^{\frac12(R-\widehat{\ell})}n^{1-\lambda}e^{-\alpha(R-\overline{\ell})})=O(n^{1-\widehat{\lambda}+\varepsilon'})=O(n^{\frac{1-\lambda+\epsilon'}{2\alpha}+\varepsilon'})=o(n^{1-\lambda}),
  $$
  where the last equality follows by choosing $\varepsilon'=\varepsilon'(\alpha)$ sufficiently small.

 Next, let us  consider vertices $w \in \calP \cap T_k$ with 
  $\overline{\ell}\le r_w\le R-\nu'$. 
By the same argument as in the first part of the lemma, the total 
  volume of such vertices is w.e.p.,
\[
\sum_{\ell'=\overline{\ell}}^{R-\nu'} O(n^{1-\lambda} e^{-\alpha(R-\ell')}) e^{\frac12 (R-\ell')}=o(n^{1-\lambda}).
\]
As before, we may assume that the remaining edges 
  are incident to pairs of vertices 
  in $B_O(R) \setminus B_O(R-\nu')$, with at least one vertex inside $T_k$. 
Since most vertices indeed have all its neighbors inside $T_k$, we may in 
  fact also consider only pairs of vertices in $T_k\setminus B_O(R-\nu')$.
For these pairs, the argument is as before, we fix integers 
  $R-\nu'\le i\le j\le R$, and partition $T_k$ into 
  $\lceil\frac{\phi}{\theta_R(i-1,j-1)}\rceil$ sectors of equal angle. 
Since $\lambda< 1$, the same argument as in the first part, replacing 
  the number of sectors $N$ by $O(Nn^{-\lambda})$, shows that the number 
  of such edges is w.e.p.~$O(n^{1-\lambda})$. 
Hence, since $\Vol(\Upsilon)=\sum_{k}\Vol(T_k)$, and for each $k$, 
  w.e.p.~$\Vol(T_k)=O(n^{1-\lambda})$, we have 
  w.e.p.~$\Vol(\Upsilon)=O(tn^{1-\lambda})=O(\phi n)$, and the
  second part of the lemma 
  is finished as well.
\end{proof}
Recall that a $\pi$-sector is a $\phi$-sector with angle $\pi$, that 
  is a half disk.
Next, we combine our previous lemma with known facts about the giant 
  component of random hyperbolic graphs in order to observe that both
  the volume and the size of their center component are linear in $n$, 
  and that this holds even if one considers only the vertices 
  that belong to a fixed $\pi$-sector of $B_{O}(R)$.

  \begin{theorem}\label{thm:volTobias1}[Theorem~1.4 of~\cite{BFM15}]
  Let $H=(U,F)$ be the center component of $G=(V,E)$ chosen according to 
  $\poimod_{\alpha,C}(n)$. Let $\Pi$ is a $\pi$-sector, then
  w.e.p.~$|U\cap\Pi|=\Omega(n)$. Moreover, w.e.p. $H$ is the giant component of $G$.
  \end{theorem}
  \begin{proof}
  A close inspection of Theorem~1.4 part~(ii)
  of~\cite{BFM15} shows that it can also be performed in the model
  $\poimod_{\alpha,C}(n)$. 
Moreover, after suitably adapting the value of $C$ and thus of $T$ as defined in Section~4.2 of~\cite{BFM15}, equation (4.21) and then also Lemma~4.2 of~\cite{BFM15} in 
  fact hold w.e.p., and thus, the proof given there shows that 
  w.e.p.~$|U|=\Omega(n)$. 
The same proof holds also when restricting to one half of 
  $B_O(R)$, and hence w.e.p.~$|U \cap \Pi|=\Omega(n)$. For the second part of the corollary, once more a close inspection of the same theorem (Lemma~4.1, equations (4.3) and (4.21) of Theorem~1.4 of ~\cite{BFM15}) show that the claimed result holds in the Poisson model, and it holds w.e.p. 
  \end{proof}

An immediate consequence of Lemma~\ref{lem:VolUpperBound}
  and Theorem~\ref{thm:volTobias1} is the following:
\begin{cor}\label{cor:volTobias}
Let $H=(U,F)$ be the center component of $G=(V,E)$ chosen according to 
  $\poimod_{\alpha,C}(n)$.
Then, w.e.p.~$\Vol(U)=\Theta(n)$. 
Moreover, if $\Pi$ is a $\pi$-sector, then
  w.e.p.~$\Vol(U\cap\Pi)\ge |U \cap \Pi|=\Omega(n)$. 
\end{cor}
Regarding the diameter of the center component, we have the following result:
\begin{theorem}\label{thm:fk1Diam}[Theorem~1 and Theorem~3 of~\cite{fk15}]
Let $H=(U,F)$ be the center component of $G=(V,E)$ chosen according to 
  $\poimod_{\alpha,C}(n)$ and let $D=D(H)$ denote its diameter. Then, w.e.p., 
  $$
  D=\Omega(\log n) \cap O((\log n)^{\frac{1}{1-\alpha}}).
  $$
  \end{theorem}
  \begin{proof}
  Again, the results stated in~\cite{fk15} are stated with smaller probability, but a close inspection of them shows that they hold w.e.p. The original results are stated in the uniform model, but again, they hold in the Poissonized model as well.
  \end{proof}
The following lemma is implicit in~\cite{BFM15}, we make it explicit here. 
\begin{lemma}~\label{lem:newvolume}
Let $H=(U,F)$ be the center component of 
  $G=(V,E)$ chosen according to $\poimod_{\alpha,C}(n)$. 
If $\Phi$ is a $\phi$-sector with 
  $\phi=\omega(\frac{1}{n}(\log n)^{\frac{1+\alpha}{1-\alpha}})$,
  then, w.e.p.~$\Vol(U\cap\Phi)\ge |U\cap\Phi| 
     =\Omega(\phi n(\log n)^{-\frac{2\alpha}{1-\alpha}})$. 
\end{lemma}
\begin{proof}
Let $\ellbdr:=\lfloor R-\frac{2\log R}{1-\alpha}\rfloor$.
Since $d(v) \ge 1$ for any $v \in U$, the inequality 
  $\Vol(U \cap \Phi) \ge |U \cap \Phi|$ is trivial. 
In order to show that 
  $|U\cap\Phi|=\Omega(\phi n(\log n)^{\frac{2\alpha}{1-\alpha}})$, 
  note that, using the lower bound on $\phi$, by Lemma~\ref{lem:muBall} 
  and Lemma~\ref{lem:meanMeus}, the number of vertices in 
  $\Phi \cap\calP_{\ellbdr}$ is 
  w.e.p.~$\Theta(\phi n (\log n)^{-\frac{2\alpha}{1-\alpha}})$. 
Note also that for every vertex $v\in \calP_{\ell}$
  with $\frac{R}{2}\le\ell\le \ellbdr$, 
  by Remark~\ref{rem:aproxAngle} and Corollary~\ref{cor:medidalb}, 
  the expected number of neighbors of
  $v$ inside $\calP_{\ell-1}$ is 
  $\Theta(ne^{-\alpha(R-\ell)}e^{\frac12(R-2\ell)})=\Omega(\log^2 n)$,
  and hence, by Lemma~\ref{lem:meanMeus} this holds w.e.p. 
Thus, all vertices $v\in \calP_{\ell}$ connect through
  consecutive layers to vertices that belong to $B_{O}(\frac{R}{2})$ and thus are part of the center component $H$. Hence,
  $|U\cap\Phi|=\Omega(\phi n(\log n)^{-\frac{2\alpha}{1-\alpha}})$. 
\end{proof}

To conclude this section, we make a final important observation that simplifies arguing about the center component (and thus
  the giant component) of random hyperbolic graphs. 
\begin{remark}\label{rem:HtoG}
The previous lemma shows that w.e.p.~all vertices in 
  $\calP\cap B_O(R-\frac{2\log R}{1-\alpha})$ in fact belong to 
  the center component, and hence, for each $\ell \le R-\frac{2\log R}{1-\alpha}$, w.e.p.~$\calP_{\ell}(G)=\calP_{\ell}(H).$
We will use this without further mention throughout the paper.
\end{remark}


\section{Spectral gap}\label{sec:gap}
The purpose of this section is to bound from below the spectral gap of 
  the center component~$H$ of a random hyperbolic graph, i.e.,
  proving Theorem~\ref{th:main}.
As we show next, this result is essentially tight. 
Indeed, we first prove Lemma~\ref{lem:upperBound} 
  by showing a simple upper bound for $\lambda_1(H)$ obtained
  via Cheeger's inequality, that is, via an upper bound on the graph 
  conductance of~$H$. 
We include the bound mainly for completeness sake.

\medskip
\noindent\textit{Proof of Lemma~\ref{lem:upperBound}.}
Let $\Pi$ be a $\pi$-sector. We have to show that 
  $h(\Pi) \le \upsilon n^{-(2\alpha-1)}\log n$.
Let $\calP$ be the set of vertices (points) if $G$ is chosen according 
  to $\poimod_{\alpha,C}(n)$, and let $\calU$ be the set of vertices (points) if $G$ is chosen according to $\unifmod_{\alpha,C}(n)$.
First, observe that by Corollary~\ref{cor:volTobias}
  w.e.p.~$\Vol(\Pi)=\Theta(n)$, $\Vol(\calP\setminus \Pi)=\Theta(n)$. 
Since Corollary~\ref{cor:volTobias} holds w.e.p., the same results clearly 
  hold in the uniform model as well. 
Hence, it suffices to show that a.a.s.~$|E(\Pi, \calU\setminus\Pi)|,|E(\Pi, \calP\setminus\Pi)| =n^{2(1-\alpha)}O(\log n)$. Define $\calU_{N}$ as a uniformly distributed set of $N$ points in the hyperbolic disk of radius $R=2 \log n+C$, i.e., $\calU_N$ equals $\calP$ conditioned on $|\calP|=N$.
We first determine the expected value of $|E(\Pi,\calU_N\setminus \Pi)|$.
Clearly, 
\[
\Ex(|E(\Pi,\calU_N\setminus \Pi)|)
  =2\binom{N}{2}
     \Pr(u\in \Pi,v\in \calU_N\setminus \Pi,\dist(r_u,r_v,\theta_u-\theta_v)\leq R).
\] 
We divide the computation of the latter probability into two 
  cases depending on whether or not $r_u+r_v\leq R$, and denote
  the corresponding probabilities by $P'$ and $P''$. 
Recalling that $C(\alpha,R)=\cosh(\alpha R)-1$ and 
  since $2\sinh x\sinh y=\cosh(x+y)-\cosh(x-y)$,
\begin{align*}
P' & =\frac{1}{4}\iint_{r_u+r_v\leq R} f(r_u)f(r_v)dr_vdr_u
    =\frac{\alpha^2}{4(C(\alpha,R))^2}\iint_{r_u+r_v\leq R}
      \sinh(\alpha r_u)\sinh(\alpha r_v)dr_vdr_u \\
& = 
  \frac{\alpha}{8(C(\alpha,R))^2}R\sinh(\alpha R)
  - \frac{1}{4C(\alpha,R)} = O(R)e^{-\alpha R} = n^{-2\alpha}O(\log n).
\end{align*}
Now, in order to compute $P''$, observe that if for
  $u\in \Pi$, $v\in \mathcal{U}_N\setminus \Pi$ with $r_u+r_v\geq R$,  
  we have $uv \in F$, then 
  either $\theta_u+(2\pi-\theta_v)\leq \theta_{R}(r_u,r_v)$ or 
  $\theta_v-\theta_u\leq \theta_{R}(r_u,r_v)$, where 
  $\theta_{R}(\cdot,\cdot)$ is 
  as defined in~\eqref{eqn:angle}.
Clearly, for $(\theta_u,\theta_v)\in [0,\pi)\times [\pi,2\pi)$ the area of both triangles defined by the aforestated two inequalities is $\theta_R(r_u,r_v)$, and hence the probability that $(\theta_u,\theta_v)$ satisfies one of the two inequalities is $\frac{1}{4\pi^2}\theta^{2}_R(r_u,r_v)$.
Thus, by Lemma~\ref{lem:aproxAngle}, 
\begin{align*}
P''& =\frac{1}{4\pi^2}\iint_{r_u+r_v\geq R} \theta^2_{R}(r_u,r_v)f(r_u)f(r_v)dr_udr_v \\
& 
  = \frac{\alpha^2}{\pi^2(C(\alpha,R))^2}
    \iint_{r_u+r_v\geq R} e^{R-r_u-r_v}(1+\Theta(e^{R-r_u-r_v}))\sinh(\alpha r_u)\sinh(\alpha r_v)dr_udr_v
\\ & 
  = \frac{\alpha^2e^{R}}{4\pi^2(C(\alpha,R))^2}
    \iint_{r_u+r_v\geq R} e^{-(1-\alpha)(r_u+r_v)}(1+O(e^{R-r_u-r_v}+e^{-2\alpha r_u}+ e^{-2\alpha r_v}))dr_udr_v 
\\ & 
  = O(R)e^{-\alpha R} = n^{-2\alpha}O(\log n).
\end{align*}
Summarizing,  for $N=n$ and the model $\unifmod_{\alpha,C}(n)$, 
  we have $\Ex(|E(\Pi,\calU\setminus \Pi)|)= O(n^{2(1-\alpha)}\log n)$. 
For the model $\poimod_{\alpha,C}(n)$,  
\begin{align*}
\Ex(|E(\Pi,\calP \setminus \Pi)|)
  &= \sum_{N \ge 0} \Ex(|E(\Pi,\calU_N\setminus \Pi)|)\Pr(|\calP|=N) 
= O(n^{-2\alpha}\log n) \sum_{N \ge 0}\binom{N}{2}e^{-n}\frac{n^N}{N!} \\
& = O(n^{2(1-\alpha)}\log n) \sum_{N \ge 2} e^{-n}\frac{n^{N-2}}{(N-2)!}
  = O(n^{2(1-\alpha)}\log n).
\end{align*}
In either case, the 
  desired statement follows by Markov's inequality.
\qed

\medskip
We now undertake the more challenging task of establishing a lower bound
  on the spectral gap of the center component of random hyperbolic graphs.
By Theorem~\ref{thm:fk1Diam}, 
  w.e.p.~the diameter of the giant component of a graph chosen according to
  $\unifmod_{\alpha,C}(n)$ is $O((\log n)^{\frac{1}{1-\alpha}})$
  when $\frac{1}{2}<\alpha<1$.
A well known relation between the spectral gap and the diameter of graphs
  (see for example~\cite[Lemma 1.9]{chung97}) establishes 
  that for a connected graph $G$ with diameter $D$ it holds
  that $\lambda_1(G)\geq 1/(D\Vol(V(G)))$.
Thus, since by Corollary~\ref{cor:volTobias}, w.e.p.~$\Vol(V(H))=\Theta(n)$, we 
  get that $\lambda_1(H)=\Omega(\frac{1}{n}(\log n)^{-\frac{1}{1-\alpha}})$.
  Since by Lemma~\ref{lem:upperBound} we have 
  $h(H)\le \upsilon n^{-(2\alpha-1)}\log n$, 
  the 
  lower bound on $\lambda_1(H)\ge \frac{1}{2}h^2(H)$ obtained from
  Cheeger's inequality (see~\eqref{eqn:cheeger}) 
  cannot be asymptotically tight when
  $\alpha > \frac{3}{4}$. 
Below, we prove a lower bound on $\lambda_1(H)$ which in fact establishes
  that up to polylogarithmic (in $n$) factors, the upper bound 
  given by Cheeger's inequality is asymptotically tight.

\medskip
In order to bound $\lambda_1(H)$ from below we rely on the multicommodity
  flow technique developed in~\cite{ds91,sinclair92}. 
The basic idea is to consider a multicommodity flow problem in the 
  graph and obtain lower bounds on $\lambda_1(H)$ in terms of 
  a measure of flows. 
Formally, a \emph{flow} in $H$ is a function $f$ mapping 
  a collection of (oriented) simple paths 
  $\calQ:=\calQ(H)$ in $H=(U,F)$ 
  to the positive reals which satisfies, for all $s,t\in U$, $s\neq t$,
  the following \emph{flow demand} constraint:
\begin{equation}\label{eqn:flow}
\sum_{q\in \calQ_{s,t}}f(q) = \frac{d(s)d(t)}{\Vol(U)},
\end{equation}
where $\calQ_{s,t}$ is the set of all (oriented) 
  paths $q\in\calQ$ from $s$ to $t$.
Clearly, an extension of $f$ to a function on oriented edges of $H$ 
  is obtained by setting $f(e)$ equal to the total flow routed
  by $f$ through the oriented edge $e$, i.e., 
  $f(e) := \sum_{q \ni e} f(q)$.

In order to measure the quality of the flow $f$ a function 
  on oriented edges, denoted $\overline{f}$, is defined by
\begin{equation}\label{eqn:elongFlow}
\overline{f}(e):=\sum_{q\in\calQ: q\ni e}f(q)|q|,
\end{equation}
where $|q|$ is the length (number of edges) of the path $q$.
The term $\overline{f}(e)$ is referred to as the 
  \emph{elongated flow} through $e$.
The flow's quality is captured by the quantity
  $\overline{\rho}(f):=\max_{e} \overline{f}(e)$, where the maximum 
  is taken over oriented edges.
The following result is the cornerstone of the multicommodity 
  flow method.
We include the claim's proof for several reasons;
  \begin{inparaenum}[(i)-]
  \item for concreteness sake, 
  \item due to its elegance and conciseness, and
  \item for clarity of exposition, because in all instances known to us, 
    the result is stated in the language of reversible Markov chains,
    and its interpretation in graph theoretic terms might not be 
    straightforward for the reader.
  \end{inparaenum}
\begin{theorem}[Sinclair~\cite{sinclair92}]\label{thm:flowEmbed}
If $f$ is a flow in a connected graph $H=(U,F)$, then
\[
\lambda_1(H) \geq \frac{1}{\overline{\rho}(f)}.
\]
\end{theorem}
\begin{proof}
Recall (see e.g.~\cite[Eqn.~(1.5)]{chung97}) the following 
  characterization of
\[
\lambda_1:=\lambda_1(H) = \inf_{\psi} \frac{\sum_{s,t:st\in F}(\psi(s)-\psi(t))^2}{\sum_{s,t\in U}(\psi(s)-\psi(t))^2\frac{d(s)d(t)}{\Vol(U)}},
\]
where the infimum is taken over all non-constant functions $\psi:U\to\RR$.

For an oriented edge $e$, let $e^-$ and $e^+$ denote its start- and endvertices.
Note now that for any $\psi$ and any flow $f$ in $H$, 
  the denominator of the last displayed equation can be bounded from above as follows:
\begin{align*}
& \sum_{s,t\in U}(\psi(s)-\psi(t))^2\frac{d(s)d(t)}{\Vol(U)}
  = \sum_{s,t\in U}\sum_{q\in\calQ_{s,t}}f(q)\big(\sum_{e\in q}(\psi(e^-)-\psi(e^+))\big)^2 
\\ & \mbox{}\qquad
  \leq \sum_{q\in\calQ}f(q)|q|\sum_{e\in q}(\psi(e^-)-\psi(e^+))^2
  = \sum_{e}(\psi(e^-)-\psi(e^+))^2\overline{f}(e)
\\ & \mbox{}\qquad
  \leq \overline{\rho}(f)\sum_{e}(\psi(e^-)-\psi(e^+))^2
  = \overline{\rho}(f)\sum_{s,t:st\in F}(\psi(s)-\psi(t))^2,
\end{align*}
where the first inequality is by Cauchy-Schwarz, and 
  the second one by definition of $\overline{\rho}(f)$.
(Note that the first equality in the preceding displayed derivation
  requires that $\calQ_{s,t}$ is non-empty for all $s,t\in U$, which
  is indeed the case given that $H$ is connected.)
\end{proof}

A particular version of the multicommodity flow method, referred to 
  as the \emph{canonical path} method, consists in routing,
  for every pair of distinct vertices $s,t\in U$, 
  the required $d(s)d(t)/\Vol(U)$ flow demand via a single oriented
  path going from $s$ to $t$.
This simplified method cannot deliver as strong bounds on $\lambda_1(H)$ as the ones we claim.
Indeed, for the canonical path method, 
  the elongated flow on any edge used by a path carrying
  flow from $s$ to $t$ must be at least $d(s)d(t)/\Vol(U)$.
Taking $s$ and $t$ as the maximum degree vertices in $H$, 
  known results on the maximum degree of hyperbolic random graphs
  (see~\cite[Theorem~2.4]{GPP12}) 
  lead to bounds on elongated flows not smaller than
  $\Omega(n^{\frac{1}{\alpha}-1})$, and thence to bounds on $\lambda_{1}(H)$
  no better than $O(n^{1-\frac{1}{\alpha}})$, which would be worse than the claimed lower bound of $\Omega(n^{-(2\alpha-1)}/D)$ if $\alpha < \frac{1}{\sqrt{2}}$ (with some effort maybe one might be able to show that the method
  does not provide strong bounds even for larger values of $\alpha$).

\medskip
To simplify the exposition, we will use Theorem~\ref{thm:flowEmbed} in
  a slightly easily derived variant stated below. 
First, say that $\{\calQ',\calQ''\}$ is a 
  \emph{path consistent} partition of $\calQ:=\calQ(H)$ provided 
  there is a path oriented from $s$ to $t$ in $\calQ'$
  if and only if no such path is found in $\calQ''$, i.e.,
  for all $s,t\in U$, $s\neq t$,
  the set $\calQ'_{s,t}$ is non-empty if and only if $\calQ''_{s,t}$ is empty.
Moreover, for $\widetilde{\calQ}\subseteq\calQ$, we say 
  $\widetilde{f}:\calQ\to\RR_+$ is a $\widetilde{\calQ}$-flow provided
  $\widetilde{f}(q)=0$ if $q\not\in\widetilde{\calQ}$ and for every $s,t\in U$,
  $s\neq t$ such that $\widetilde{\calQ}_{s,t}$ is non-empty, the 
  following holds:
\begin{equation}\label{eq:flowCondition}
\sum_{q\in\widetilde{Q}_{s,t}} \widetilde{f}(q) = \frac{d(s)d(t)}{\Vol(U)}.
\end{equation}
We extend to $\widetilde{Q}$-flows,
  in the natural way, the notions of elongated flow and maximum
  elongated flow.
In order to more easily apply Theorem~\ref{thm:flowEmbed} we will
  construct a flow satisfying its hypothesis as a sum of 
  $\widetilde{Q}$-flows. 
Our next result validates such an approach.
\begin{cor}\label{cor:flowEmbed}
Let $H=(U,F)$ be a connected graph and $\{\calQ',\calQ''\}$ a path consistent
  partition of $\calQ:=\calQ(H)$.
Let $f',f'':\calQ\to\RR_+$ be such that $f'$ is a
  $\calQ'$-flow and $f''$ is a $\calQ''$-flow, then $f'+f''$ is a flow 
  in $H$ and
\[
\overline{\rho}(f'+f'') \leq \overline{\rho}(f')+\overline{\rho}(f'').
\]
\end{cor}
\begin{proof}
The result follows since
  $\overline{\rho}(f'+f'')= \max_{e}\big(\overline{f'}(e)+\overline{f''}(e)\big)
    \leq \overline{\rho}(f')+\overline{\rho}(f'')$.
\end{proof}

Key to our approach is the fact that w.e.p.~random hyperbolic graphs
  admit multicommodity flows of moderate maximum elongated flow.
To prove this assertion we associate to
  the center component $H$ of $G$ chosen according to 
  $\poimod_{\alpha,C}(n)$ a path consistent partition 
  $\{\calQ',\calQ''\}$ of 
  $\calQ:=\calQ(H)$.
The collection $\calQ'$ will consist of paths whose 
  endvertices are both ``sufficiently close'' to the origin $O$.
In contrast, $\calQ''$ will consist of the collection of paths one of 
  whose endvertices is not ``sufficiently close'' to the origin $O$.
We will fix the flow for path $q$ with endvertices $s$ and $t$, 
  so that it satisfies~\eqref{eqn:flow} while distributing 
  an equal amount of flow among all paths in $\calQ_{s,t}$.

In addition to the already defined quantities 
  $\elllow=\lfloor (1-\frac{1}{2\alpha})R\rfloor$ and 
  $\nu'=2\log R+\omega(1)\cap o(\log R)$, 
  the following quantities will also play an important role
  in the construction of $\calQ'$ and $\calQ''$:
\begin{align}
\ellmin & := \big\lceil(\alpha-\tfrac{1}{2})R+\nu'\big\rceil, 
  \label{def:lmin} \\
\ellmid & := \big\lfloor\tfrac{R}{2}\big\rfloor, 
  \label{def:lmid} \\
  \ellmax & := 
  \big\lfloor(\tfrac{3}{2}-\alpha)R-\nu'\big\rfloor. 
\end{align}
Observe that $\ellmin+\ellmax=R$.
For sufficiently large $n$, it always holds that 
  $\ellmin<\ellmid<\ellmax$ and $\ellmin<\elllow+\nu<\ellmid$.
From now on, we assume without further mention 
  that $n$ is large enough so that these inequalities hold.
Henceforth, for an integer $\ell\leq \ellmax$, we let 
\[
\widetilde{\ell}=\begin{cases}
     \ellmax, & \text{if $\ell < \ellmin$,} \\ 
     2\ellmid-\ell+1, & \text{if $\ellmin\leq\ell\leq\ellmid$,} \\ 
     \ellmid, & \text{if $\ell > \ellmid$}.
\end{cases}
\]
Note that $\frac{R}{2}\geq\ellmid=\frac{R}{2}+\Theta(1)$
  and $\ellmid\leq\widetilde{\ell}\leq\ellmax$.
Moreover, observe that $\ell\leq\ellmid$ if and only if 
  $\widetilde{\ell}>\ellmid$.
As before, often we shall ignore the floors/ceilings in the preceding 
  definitions, since it only introduces low order term approximations
  in our derivations. 
Recall that whenever referring to expressions such as 
  $R-\frac{\log R}{1-\alpha}$ or the like, when needed, we will also 
  assume that these are integers.

Details concerning $\calQ'$ as well as an associated $\calQ'$-flow
  are provided in the next section, and in the subsequent one analogous
  results concerning $\calQ''$ are discussed. 

\subsection{A $\calQ'$-flow}
For $s\in \calP_{k}$ and $t\in \calP_{k'}$ with $k,k'\leq\ellmax$, 
  let ${\calQ'}_{s,t}$ be the collection of length $3$ oriented paths
  from $s$ to $t$ whose first internal vertex belongs to $\calP_{\widetilde{k}}$ 
  and the other internal vertex is in $\calP_{\widetilde{k'}}$.
Also, let ${\calQ'}$ be the union of all such $\calQ'_{s,t}$'s.
We classify paths in $\calQ'$ as follows  (see Figure~\ref{fig:pathTypes}):
\begin{itemize}
\item \textbf{Type I:} both endvertices belong to $B_{O}(\ellmid)$
\item \textbf{Type II:} both endvertices belong to $B_{O}(\ellmax)\setminus B_{O}(\ellmid)$
\item \textbf{Type III:} one endvertex is in $B_{O}(\ellmid)$ and 
  the other one in $B_{O}(\ellmax)\setminus B_{O}(\ellmid)$
\end{itemize}

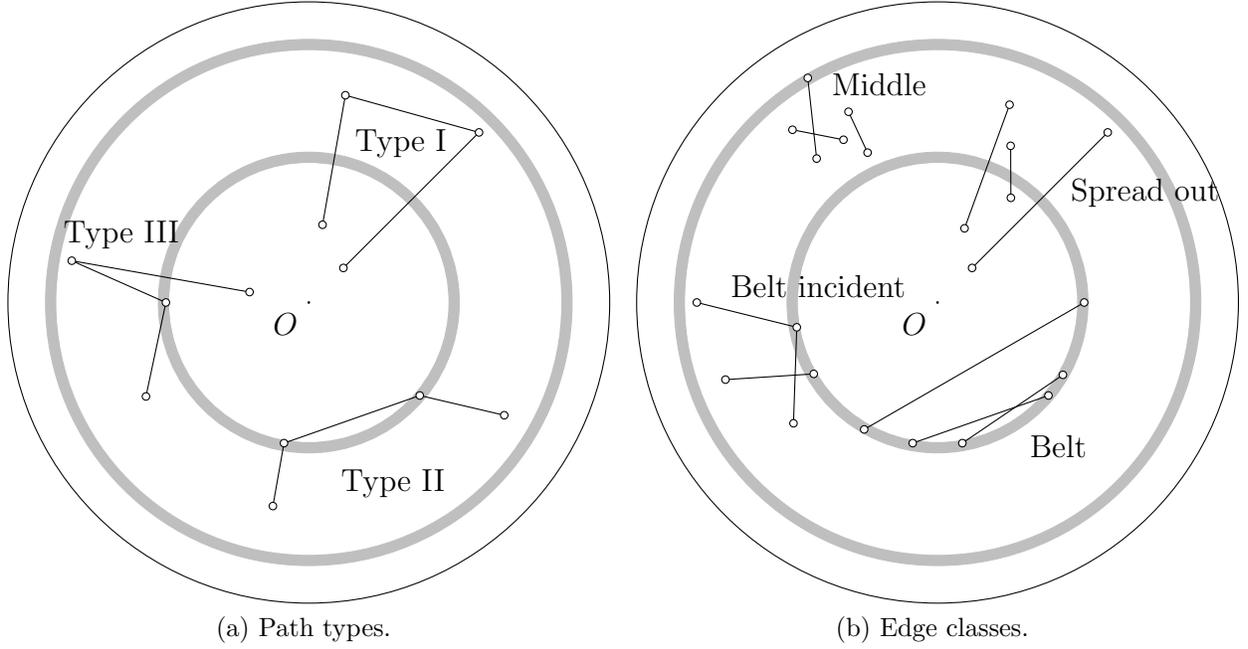
\begin{figure}[htbp!]
\centering
\subfloat[Path types.]{\label{fig:pathTypes}
\begin{tikzpicture}[scale=1.0]
\draw [-] (0,0) circle (4) {};
\draw [fill,gray!50] (0,0) circle (3.5) {};
\draw [fill,white] (0,0) circle (3.35) {};
\draw [fill,gray!50] (0,0) circle (2) {};
\draw [fill,white] (0,0) circle (1.85) {};
\draw [fill] (0,0) circle (0.01) {} node [below left] at (0,0) {$O$};

\draw [-] (45:0.65) -- (45:3.2) -- (80:2.8) -- (80:1.05);
\draw [fill=white] (45:0.65) circle (0.05) {};
\draw [fill=white] (45:3.2) circle (0.05) {};
\draw [fill=white] (80:2.8) circle (0.05) {};
\draw [fill=white] (80:1.05) circle (0.05) {};
\node at (60:2.45) {Type I};

\draw [-] (-30:3) -- (-40:1.925) -- (-100:1.9) -- (-100:2.75);
\draw [fill=white] (-30:3) circle (0.05) {};
\draw [fill=white] (-40:1.925) circle (0.05) {};
\draw [fill=white] (-100:1.9) circle (0.05) {};
\draw [fill=white] (-100:2.75) circle (0.05) {};
\node at (-65:2.65) {Type II};

\draw [-] (170:0.8) -- (170:3.2) -- (180:1.9) -- (210:2.5);
\draw [fill=white] (170:0.8) circle (0.05) {};
\draw [fill=white] (170:3.2) circle (0.05) {};
\draw [fill=white] (180:1.9) circle (0.05) {};
\draw [fill=white] (210:2.5) circle (0.05) {};
\node at (160:2.65) {Type III};
\end{tikzpicture}}\hfill
\subfloat[Edge classes.]{\label{fig:edgeTypes}
\begin{tikzpicture}[scale=1.0]
\draw [-] (0,0) circle (4) {};
\draw [fill,gray!50] (0,0) circle (3.5) {};
\draw [fill,white] (0,0) circle (3.35) {};
\draw [fill,gray!50] (0,0) circle (2) {};
\draw [fill,white] (0,0) circle (1.85) {};
\draw [fill] (0,0) circle (0.01) {} node [below left] at (0,0) {$O$};

\draw [-] (45:0.65) -- (45:3.2);
\draw [-] (70:2.8) -- (70:1.05);
\draw [-] (55:1.7) -- (65:2.3);
\draw [fill=white] (45:0.65) circle (0.05) {};
\draw [fill=white] (45:3.2) circle (0.05) {};
\draw [fill=white] (70:2.8) circle (0.05) {};
\draw [fill=white] (70:1.05) circle (0.05) {};
\draw [fill=white] (55:1.7) circle (0.05) {};
\draw [fill=white] (65:2.3) circle (0.05) {};
\node at (28:3.125) {Spread out};

\draw (-40:1.925) -- (-100:1.9);
\draw (-30:1.925) -- (-80:1.9);
\draw (0:1.95) -- (-120:1.95);
\draw [fill=white] (-30:1.925) circle (0.05) {};
\draw [fill=white] (-40:1.925) circle (0.05) {};
\draw [fill=white] (-100:1.9) circle (0.05) {};
\draw [fill=white] (-80:1.9) circle (0.05) {};
\draw [fill=white] (0:1.95) circle (0.05) {};
\draw [fill=white] (-120:1.95) circle (0.05) {};
\node at (-50:2.5) {Belt};

\draw [-] (180:3.2) -- (190:1.9) -- (220:2.5);
\draw [-] (200:3.0) -- (210:1.9);
\draw [fill=white] (180:3.2) circle (0.05) {};
\draw [fill=white] (190:1.9) circle (0.05) {};
\draw [fill=white] (220:2.5) circle (0.05) {};
\draw [fill=white] (200:3.0) circle (0.05) {};
\draw [fill=white] (210:1.9) circle (0.05) {};
\node at (172:1.6) {Belt incident};

\draw [-] (120:3.45) -- (130:2.5);
\draw [-] (115:2.8) -- (115:2.2);
\draw [-] (130:3.0) -- (120:2.5);
\draw [fill=white] (120:3.45) circle (0.05) {};
\draw [fill=white] (130:2.5) circle (0.05) {};
\draw [fill=white] (115:2.8) circle (0.05) {};
\draw [fill=white] (115:2.2) circle (0.05) {};
\draw [fill=white] (120:2.5) circle (0.05) {};
\draw [fill=white] (130:3.0) circle (0.05) {};
\node at (105:3.0) {Middle};
\end{tikzpicture}}
\caption{Illustration of path types and edge classes. 
Inner shaded rings correspond to 
  $B_{O}(\ellmid)\setminus B_{O}(\ellmid-1)$, outer shaded rings to $B_{O}(\ellmax)\setminus B_{O}(\ellmax-1)$ for
  $\alpha=5/8$.}
\end{figure}

Next, we relate the size of the $\calQ'_{s,t}$'s to the size of 
  certain collections of edges of $H=(U,F)$. 
This will be useful for estimating their size.
\begin{proposition}\label{prop:nonEmptyQ2}
If $v_g\in\calP_g$ and $v_h\in\calP_{h}$ with $g\leq h\leq \ellmax$, then 
\[
|\calQ'_{v_g,v_h}| = \begin{cases}
|E(\calP_{\widetilde{g}},\calP_{\widetilde{h}})|, 
  & \text{if $g,h\leq\ellmid$,} \\
|E(\calP_{\widetilde{g}},N_{\calP_{\ellmid}}(v_h))|, 
  & \text{if $g\leq \ellmid<h$,} \\
|N_{\calP_{\ellmid}}(v_g)|\cdot |N_{\calP_{\ellmid}}(v_h)|  
  & \text{if $g,h>\ellmid$.} 
\end{cases}
\]
\end{proposition}
\begin{proof}
The claim holds for $g,h\leq\ellmid$ because for each 
  edge $e\in E(\calP_{\widetilde{g}},\calP_{\widetilde{h}})$
  there is a path in $\calQ'_{v_g,v_h}$ with node set 
  $\{v_g,e^-,e^+,v_h\}$ and the middle edge of any path 
  in $\calQ'_{v_g,v_h}$ belongs to 
  $E(\calP_{\widetilde{g}},\calP_{\widetilde{h}})$.
The remaining cases are handled similarly.
\end{proof}

As already mentioned, we will evenly split the flow that needs 
  to be sent from a vertex~$s$ to another vertex~$t$  among all 
  oriented paths connecting~$s$ to~$t$. 
This partly explains, at least when $s,t\in U\cap B_{O}(\ellmax)$,
  why we next estimate the number of paths in $\calQ'_{s,t}$.
\begin{proposition}\label{prop:neighBound}
W.e.p., For $g, h\leq\ellmax$ where $g\geq\ellmid$ the following hold:
\begin{enumerate}[(i).-]
\item\label{prop:neighBoundPart1}
If $v_g\in\calP_g$, then~$|N_{\calP_{\widetilde{h}}}(v_g)|
     =\Theta(e^{-(\alpha-\frac12)(R-\widetilde{h})}e^{\frac{1}{2}(R-g)})
     =\Theta(e^{-(\alpha-\frac12)(R-\widetilde{h})}d(v_g))$.  
In particular,
  $|N_{\calP_{\ellmid}}(v_g)|
     =\Theta(n^{-(\alpha-\frac12)}e^{\frac{1}{2}(R-g)})
     =\Theta(n^{-(\alpha-\frac12)}d(v_g))$.

\item\label{prop:neighBoundPart2}
$|E(\calP_{g},\calP_{\widetilde{h}})|
  =\Theta(ne^{-(\alpha-\frac12)(R-\widetilde{h})}e^{-(\alpha-\frac12)(R-g)})$.

\item\label{prop:neighBoundPart3}
If $S\subseteq\calP_{\ellmid}$, then $|E(\calP_{g},S)|
  =\Theta(|S|\sqrt{n}e^{-(\alpha-\frac12)(R-g)})$.
\end{enumerate}
\end{proposition}
\begin{proof}
Consider the first part of the claim.
If $\widetilde{h}=g=\ellmid$, since $\calP_{\ellmid}$ induces 
  a clique in $H$, then $N_{\widetilde{h}}(v_g)=\calP_{\widetilde{h}}$.
Since by definition $\ellmid=R+\Theta(1)$
  and Proposition~\ref{prop:degreeBound} implies that 
  $d(v_g)=\Theta(\sqrt{n})$, the claim trivially holds by 
  Proposition~\ref{prop:bandBound} part~\eqref{prop:bandBoundPart1}.
Assume that $\widetilde{h}+g>2\ellmid\geq R$.  
Note that if a vertex in $\calP_{\widetilde{h}}$ is a neighbor 
  of $v_g\in\calP_g$ in $H$, then the small
  relative angle (in the interval $[0,\pi)$) between 
  such a vertex and $v_g$ is $O(\theta_{R}(g,\widetilde{h}))$,
  which by Lemma~\ref{lem:aproxAngle}, equals
  $\Theta(e^{\frac{1}{2}(R-g-\widetilde{h})})$.
Applying Lemma~\ref{lem:muBall}, we infer 
  that $\mu(B_{O}(\widetilde{h})\setminus B_{O}(\widetilde{h}-1))
    =e^{-\alpha(R-\widetilde{h})}(1-e^{-\alpha})(1+o(1))$.
Thus, for a sector $\Phi$ of $B_O(R)$ of 
  angle $\phi=\Theta(e^{\frac{1}{2}(R-g-\widetilde{h})})$,
\begin{align*}
& \mu(\Phi\cap B_{O}(\widetilde{h})\setminus B_{O}(\widetilde{h}-1)) 
  = \phi\mu(B_{O}(\widetilde{h})\setminus B_{O}(\widetilde{h}-1)) 
   = \Theta(\tfrac{1}{n})e^{-(\alpha-\frac{1}{2})(R-\widetilde{h})}e^{\frac{1}{2}(R-g)}. 
\end{align*}
Since $g\leq\ellmax$, $\widetilde{h}\geq\ellmid$ and
  because $\nu'=2\log R+\omega(1)$, recalling the definition 
  of $\ellmid$ and $\ellmax$, we deduce that
\[
e^{\frac12(R-g)-(\alpha-\frac12)(R-\widetilde{h})}
  \geq e^{\frac12(R-\ellmax)-(\alpha-\frac12)(R-\ellmid)}
  =\Omega(e^{\tfrac{\nu'}{2}})=\omega(\log n).
\]
We have established that 
  $\mu(\Phi\cap B_{O}(\widetilde{h})\setminus B_{O}(\widetilde{h}-1))
     =\omega(\frac{\log n}{n})$, so
  the desired conclusions follow
  by Proposition~\ref{prop:degreeBound} and Lemma~\ref{lem:meanMeus}.
The second part of~\eqref{prop:neighBoundPart1} follows immediately since 
  $\ellmid=\frac{R}{2}+\Theta(1)$.

Consider now the second part of the claim.
Note that $|E(\calP_{g},\calP_{\widetilde{h}})|
  =\sum_{v\in\calP_{g}}|N_{\calP_{\widetilde{h}}}(v)|$.
Since $g\geq\ellmid$, 
  the claim follows immediately 
  from the first part by a union bound and
  by Proposition~\ref{prop:bandBound} part~\eqref{prop:bandBoundPart1}.

For the last part of the claim, observe that 
  $|E(\calP_{g},S)|=\sum_{w\in S} |N_{\calP_{g}}(w)|$.
By part~\eqref{prop:neighBoundPart1}, a union bound over 
  the elements of $\calP_{\ellmid}$ yield that 
  w.e.p., for all $w\in S$
  it holds that $|N_{\calP_{g}(w)}|
    =\Theta(e^{-(\alpha-\frac12)(R-g)}e^{\frac12(R-\ellmid)})$.
The conclusion follows by definition of $\ellmid$.
\end{proof}

Next, we establish the main result of this section.
\begin{proposition}\label{prop:flowPart1}
Let $H=(U,F)$ be the center component of 
  $G=(V,E)$ chosen according to $\poimod_{\alpha,C}(n)$.
For all $q\in{\calQ'}_{s,t}$, let 
\[
{f'}(q):=
  \frac{d(s)d(t)}{\Vol(U)}\cdot\frac{1}{|{\calQ}'_{s,t}|}.
\]
Then, w.e.p.~${\calQ'}\subseteq\calQ(H)$, 
  ${f'}$ is a well defined ${\calQ'}$-flow and 
  $\overline{\rho}({f'})=O(n^{2\alpha-1})$.
\end{proposition}
\begin{proof}
For $s,t\in B_{O}(\ellmax)$, 
  Proposition~\ref{prop:nonEmptyQ2} and Proposition~\ref{prop:neighBound},
  imply that $|\calQ''_{s,t}|\neq 0$.
Thus, $f'$ is well defined.
Moreover, by the way in which $f'$ is prescribed, 
  $\sum_{q\in\calQ'_{s,t}} f'(q)=d(s)d(t)/\Vol(U)$, so $f'$ is a flow.

We need to bound the elongated flow in the edges traversed by paths
  in ${\calQ}'$.
First, we identify which edges $e$ of $H$ are traversed.
Paths in $\calQ'$ traverse edges of $H$ whose endvertices are 
  in $B_{O}(\ellmax)$.
Moreover,
the endvertices of $e$ are not both in $B_{O}(\ellmid-1)$, and a path in ${\calQ}'$ either starts or ends with $e$ if and only if
  at least one of the endvertices of $e$ is in $B_{O}(\ellmid)$.
If follows that an edge $e$ traversed by a path in 
  ${\calQ}'$ can belong to one  
  of four edge classes described forthwith.
An upper bound on the elongated flow of the members of each of these 
  classes is separately derived below
  (recall that for an oriented edge $e$, the 
  expressions $e^-$ and $e^+$ denote its start- and 
  endvertices).

Since $\calQ''_{s,t}=\calQ''_{t,s}$ for every distinct $s,t\in U$,
  the elongated flow $\overline{f''}$ is the same for both orientations of 
  a given edge.
Thus, in our ensuing discussion we fix (arbitrarily) one of 
  the two possible orientations of $e$ when bounding
  its elongated flow. 

\medskip\noindent\textbf{Spread out edges}
(one endvertex of $e$ is in $B_O(\ellmid)$ and the other 
  one in $B_{O}(\ellmax)\setminus B_O(\ellmid)$):
The only possibility is that for some $k\leq\ellmid$, the edge $e$ 
  is incident to a vertex in $\calP_{k}$ and to another one in 
  $\calP_{\widetilde{k}}$.
Fix the orientation of $e$ so $e^-\in\calP_k$ and $e^+\in\calP_{\widetilde{k}}$.
Necessarily, $e$ is the first 
   edge of a Type I path in ${\calQ}'$ that traverses it.
Also,
\begin{align*}
\frac{\overline{{f}'(e)}}{3} & = 
\frac{d(e^-)}{\Vol(U)}\Big(
  \sum_{\ell\leq\ellmid} \sum_{t\in\calP_{\ell}}
    \frac{d(t)}{|{\calQ}'_{e^-,t}|}|N_{\calP_{\widetilde{\ell}}}(e^+)|
  +
  \sum_{\ellmid<\ell\leq\ellmax} \sum_{t\in\calP_{\ell}}
    \frac{d(t)}{|{\calQ}'_{e^-,t}|}|E(\{e^+\},N_{\calP_{\ellmid}}(t))|  
\Big).
\end{align*}
Let $S_1$ and $S_2$ be the first and second summands inside the 
  parenthesis of the right hand side above.

First, we bound $S_1$.
Assume $\ell\leq\ellmid$ and $t\in\calP_{\ell}$.
By Proposition~\ref{prop:nonEmptyQ2}, 
  $|{\calQ}'_{e^-,t}|=|E(\calP_{\widetilde{k}},\calP_{\widetilde{\ell}})|$.
Since $\widetilde{k}>\ellmid$,
  part~\eqref{prop:neighBoundPart2} 
  of Proposition~\ref{prop:neighBound} 
  applies, implying that  w.e.p., 
  $|N_{\calP_{\widetilde{\ell}}}(e^+)|/|E(\calP_{\widetilde{k}},\calP_{\widetilde{\ell}})| 
      = O(\frac{1}{n}e^{\alpha(R-\widetilde{k})})$. 
Hence, w.e.p.,
\[
S_1 = O\big(\tfrac{1}{n}e^{\alpha(R-\widetilde{k})})
  \sum_{\ell\leq\ellmid}\Vol(\calP_{\ell}).
\]

\smallskip
We now bound $S_2$ from above.
Assume $\ellmid<\ell\leq\ellmax$ and $t\in\calP_\ell$.
By Proposition~\ref{prop:nonEmptyQ2}, 
  $|\calQ'_{e^-,t}|=|E(\calP_{\widetilde{k}},N_{\calP_{\ellmid}}(t))|$.
Moreover, since $\widetilde{k}\geq\ellmid$, Proposition~\ref{prop:neighBound}
  part~\eqref{prop:neighBoundPart3} yields that,
  w.e.p., $|\calQ'_{e^-,t}|
    =\Theta(|N_{\calP_{\ellmid}}(t)|\sqrt{n}e^{-(\alpha-\frac12)(R-\widetilde{k})})$.
By part~\eqref{prop:neighBoundPart1} of the same proposition, we get that
  w.e.p.~$d(t)/|\calQ'_{e^-,t}|=\Theta(n^{-(1-\alpha)}e^{(\alpha-\frac12)(R-\widetilde{k})}))$.
Also, $\sum_{t\in\calP_\ell} |E(\{e^+\},N_{\calP_{\ellmid}}(t))|
   =|E(N_{\calP_{\ellmid}}(e^+),\calP_\ell)|$,
  so by Proposition~\ref{prop:neighBound}
  part~\eqref{prop:neighBoundPart1} and part~\eqref{prop:neighBoundPart3}, 
  w.e.p.~$|E(N_{\calP_{\ellmid}}(e^+),\calP_\ell)|
    =\Theta(n^{1-\alpha}e^{-(\alpha-\frac12)(R-\ell)}e^{\frac12(R-\widetilde{k})})$.
Recalling that by Proposition~\ref{prop:bandBound}, we know 
  that w.e.p.~$\Vol(\calP_{\ell})=\Theta(ne^{-(\alpha-\frac12)(R-\ell)})$
  for $\ellmid<\ell\leq\ellmax$, it follows that w.e.p.,
\begin{align*}
S_2 & = 
   \Theta(n^{-(1-\alpha)}e^{(\alpha-\frac12)(R-\widetilde{k})})
   \sum_{\ellmid<\ell\leq\ellmax}
   |E(N_{\calP_{\ellmid}}(e^+),\calP_{\ell})| 
    = \Theta(\tfrac{1}{n}e^{\alpha(R-\widetilde{k})})
    \sum_{\ellmid<\ell\leq\ellmax} \Vol(\calP_\ell).
\end{align*}
Summarizing, $\overline{f''}(e)
  =\frac{d(e^-)}{\Vol(U)}\Theta(\frac{1}{n}e^{\alpha(R-\widetilde{k})}\sum_{\ell\leq\ellmax}\Vol(\calP_\ell))$.  
Since the summation in this last expression is clearly at most 
  $\Vol(U)$ and observing that by Proposition~\ref{prop:degreeBound}, 
  w.e.p.~$d(e^-)=\Theta(ne^{-\frac{k}{2}})$,  
  we conclude that w.e.p.~$\overline{f''}(e)
  =\Theta(e^{-\frac{k}{2}+\alpha(R-\widetilde{k})})$.
Finally, recall that $k\leq\ellmid$ and $\alpha>\frac12$, 
  so $\alpha(R-\widetilde{k})-\frac12 k
     \leq\max\{(\alpha-\frac12)k,\alpha(R-\ellmax)\}
     \leq \alpha(R-\ellmax)$.
By definition of $\ellmax$ and since $\alpha<1$, we infer that 
  w.e.p.~$\overline{f''}(e)=O(e^{\alpha(R-\ellmax)})
    =O(n^{\alpha(2\alpha-1)}e^{\alpha \nu'})=o(n^{2\alpha-1})$.

\medskip\noindent\textbf{Belt edges} 
(both endvertices of $e$ in $B_{O}(\ellmid)\setminus B_{O}(\ellmid-1)$):
The only possibility is that $e$ is the middle edge of 
  a path in ${\calQ}'$ of Type II.
In particular,
\[
\frac{\overline{{f'}}(e)}{3} =   
  \frac{1}{\Vol(U)}\sum_{\ellmid<\ell,\ell'\leq\ellmax}\sum_{s\in N_{\calP_{\ell}}(e^-)}\sum_{t\in N_{\calP_{\ell'}}(e^+)}\frac{d(s)d(t)}{|{\calQ}'_{s,t}|}.
\]
By Proposition~\ref{prop:nonEmptyQ2}, 
  if $s\in\calP_\ell$ and $t\in\calP_{\ell'}$ with 
  $\ellmid<\ell,\ell'\leq\ellmax$, then
  $|{\calQ}'_{s,t}|=|N_{\calP_{\ellmid}}(s)|\cdot |N_{\calP_{\ellmid}}(t)|$.
By Proposition~\ref{prop:neighBound} 
  part~\eqref{prop:neighBoundPart1}, for $w\in\calP_{\ell}\cup\calP_{\ell'}$, 
  expressions like $d(w)/|N_{\calP_{\ellmid}}(w)|$ 
  equal, w.e.p., $\Theta(n^{\alpha-\frac12})$.
Since a vertex cannot have more neighbors than its degree, w.e.p.,
\begin{align*}
& \overline{{f'}}(e)= \frac{\Theta(n^{2\alpha-1})}{\Vol(U)}
  \sum_{\ellmid<\ell\leq\ellmax} |N_{\calP_{\ell}}(e^-)|
  \sum_{\ellmid<\ell'\leq\ellmax} |N_{\calP_{\ell'}}(e^+)|
  \leq \frac{\Theta(n^{2\alpha-1})}{\Vol(U)}d(e^-)d(e^+).
\end{align*}
By Proposition~\ref{prop:degreeBound}, w.e.p.~$d(e^-), d(e^+)=\Theta(\sqrt{n})$,
  so by Lemma~\ref{lem:VolUpperBound}, w.e.p., 
  $\overline{{f'}}(e)=O(n^{2\alpha-1})$.

\medskip\noindent\textbf{Middle edges}
(both endvertices of $e$ in $B_{O}(\ellmax)\setminus B_{O}(\ellmid)$):
Now, $e$ can only appear as the middle edge of a path in $\calQ'$ of 
  Type I.
Say  $e^-\in\calP_{\widetilde{k}}$ and $e^+\in\calP_{\widetilde{k'}}$ 
  for $k,k'\leq\ellmid$.
Note that if $e$ is traversed by some path in $\calQ'_{s,t}$, 
  then it must be the case that $s\in\calP_{\ell}$ for some $\ell$ such
  that $\widetilde{\ell}=\widetilde{k}$ (if $\widetilde{k}\neq\ellmax$
  there is only one such $\ell$, otherwise $\ell\leq\ellmin$).
A similar statement holds for $t$.
By Proposition~\ref{prop:nonEmptyQ2}, 
  for $s\in\calP_k$ and $t\in\calP_{k'}$, we have that 
  $|\calQ'_{s,t}|=|E(\calP_{\widetilde{k}},\calP_{\widetilde{k'}})|$,  and hence
\[
\frac{\overline{{f}'(e)}}{3} = 
  \frac{1}{\Vol(U)}\!
  \sum_{\ell:\widetilde{\ell}=\widetilde{k}}
  \sum_{\ell':\widetilde{\ell'}=\widetilde{k'}}
  \sum_{s\in\calP_\ell}\sum_{t\in\calP_{\ell'}}
  \frac{d(s)d(t)}{|{\calQ}'_{s,t}|}
  = \frac{1}{\Vol(U)}
    \cdot\frac{1}{|E(\calP_{\widetilde{k}},\calP_{\widetilde{k'}})|}
  \sum_{\ell:\widetilde{\ell}=\widetilde{k}}\Vol(\calP_{\ell})\!\!
  \sum_{\ell':\widetilde{\ell'}=\widetilde{k'}}\Vol(\calP_{\ell'}).
\]
Since $\widetilde{k},\widetilde{k'}>\ellmid$, 
  by Proposition~\ref{prop:neighBound} 
  part~\eqref{prop:neighBoundPart2}, recalling that $\ellmin+\ellmax=R$, 
  since $\ellmin < \elllow+\nu$ (where $\nu=\frac{1}{\alpha}\log R+\omega(1)\cap o(\log R)$) and the way in which 
  $\widetilde{k}$ is defined, w.e.p.,
\begin{align*}
&  |E(\calP_{\widetilde{k}},\calP_{\widetilde{k'}})|
  =\Theta\big(ne^{-(\alpha-\frac{1}{2})(R-\widetilde{k})}e^{-(\alpha-\frac12)(R-\widetilde{k'})}\big) 
  =\Omega\big(ne^{-(\alpha-\frac{1}{2})(\max\{k,\elllow+\nu\}+\max\{k',\elllow+\nu\})}\big).
\end{align*}
Also, by Proposition~\ref{prop:bandBound} part~\eqref{prop:bandBoundPart2}
  and definition of $\elllow$, w.e.p., 
\[
e^{(\alpha-\frac{1}{2})\max\{k,\elllow+\nu\}}\sum_{\ell:\widetilde{\ell}=\widetilde{k}}
  \Vol(\calP_\ell)
  = \begin{cases}
    O(ne^{-(\alpha-\frac12)(R-2k)}), 
    & \text{if $k\geq\elllow+\nu$,} \\
    O(n^{\frac{(2\alpha-1)^2}{2\alpha}}e^{(\alpha-\frac12)\nu}\Vol(U)), 
    & \text{if $k<\elllow+\nu$.}
    \end{cases}
\]
Since $k\leq\ellmid\leq\frac{R}{2}$
  and $\frac{2\alpha-1}{\alpha}<1$ (given that $\alpha<1$),
  by Lemma~\ref{lem:VolUpperBound},
  w.e.p., the case that dominates above is when $k<\elllow+\nu$, which
  in turn is $o(n^{\alpha+\frac12})$.
Hence, again using Lemma~\ref{lem:VolUpperBound},  w.e.p., $\overline{{f'}}(e)
  =o\big(\tfrac{1}{\Vol(U)}\cdot\tfrac{1}{n}\cdot (n^{\alpha+\frac12})^2\big)
  =o\big(\tfrac{n^{2\alpha}}{\Vol(U)}\big)=o(n^{2\alpha-1})$.

\medskip\noindent\textbf{Belt incident edges}
(one endvertex of $e$ in $\calP_{\ellmid}$ and the other one in 
  $B_{O}(\ellmax)\setminus B_{O}(\ellmid)$):
Fix the orientation of $e$ so $e^-\in\calP_k$ for $\ellmid<k\leq\ellmax$
  and $e^+\in\calP_{\ellmid}$.
Note that $e$ can be the first edge of either a 
  Type II or Type III path, or the middle edge of a Type III path.
  Each alternative gives rise to 
  one of the terms on the right hand side of the following identity:
\begin{align*}
\frac{\overline{{f'}}(e)}{3}  & =  
  \frac{d(e^-)}{\Vol(U)}
  \sum_{\ell\leq\ellmax}\sum_{t\in\calP_{\ell}}
  \frac{d(t)}{|{\calQ}'_{e^-,t}|}|E(\{e^+\},N_{\calP_{\widetilde{\ell}}}(t))| \\
& \qquad 
  + \frac{1}{\Vol(U)}\sum_{\ell\leq\ellmid:\widetilde{\ell}=k}\sum_{s\in\calP_{\ell}}
    \sum_{\ellmid<\ell'\leq\ellmax}\sum_{t\in N_{\calP_{\ell'}}(e^+)}
    \frac{d(s)d(t)}{|{\calQ}'_{s,t}|}.
\end{align*}
Let $S_1$ and $S_2$ be the first and second terms on the right hand
  side above.

First, we bound $S_1$.
Let $t\in\calP_{\ell}$ for $\ell\leq\ellmax$.
By Proposition~\ref{prop:nonEmptyQ2}, 
  if $\ell\leq\ellmid$, then
  $|{\calQ}'_{e^-,t}|
    =|E(N_{\calP_{\ellmid}}(e^-),\calP_{\widetilde{\ell}})|$
  and $N_{\calP_{\widetilde{\ell}}}(t)=\calP_{\widetilde{\ell}}$
  (in particular, $E(\{e^+\},N_{\calP_{\widetilde{\ell}}}(t))
     =N_{\calP_{\widetilde{\ell}}}(e^+)$).
Moreover, if $\ell>\ellmid$, then 
  $|{\calQ}'_{e^-,t}|=|N_{\calP_{\ellmid}}(e^-)|\cdot |N_{\calP_{\ellmid}}(t)|$
  and $\widetilde{\ell}=\ellmid$. 
Since vertices in $\calP_{\ellmid}$ induce a clique in $H$, we have
  $|E(\{e^+\},N_{\calP_{\widetilde{\ell}}}(t))|=|N_{\calP_{\ellmid}}(t)|$.
Thus,
\[
S_1 = 
  \frac{d(e^-)}{\Vol(U)}\Big(
  \sum_{\ell\leq\ellmid}\frac{|N_{\calP_{\widetilde{\ell}}}(e^+)|\cdot\Vol(\calP_{\ell})}{|E(N_{\calP_{\ellmid}}(e^-),\calP_{\widetilde{\ell}})|}
  + \frac{1}{|N_{\calP_{\ellmid}}(e^-)|}
  \sum_{\ellmid<\ell\leq\ellmax}\Vol(\calP_{\ell})
\Big).
\]
By parts~\eqref{prop:neighBoundPart1} and~\eqref{prop:neighBoundPart3}
  of Proposition~\ref{prop:neighBound}
  if $\ell\leq\ellmid$,
  then w.e.p.~$|E(N_{\calP_{\ellmid}}(e^-),\calP_{\widetilde{\ell}})|
     =|N_{\calP_{\ellmid}}(e^-)|\cdot |N_{\calP_{\widetilde{\ell}}}(e^+)|$.
By part~\eqref{prop:neighBoundPart1} of the same proposition, 
  w.e.p.~$d(e^-)/|N_{\calP_{\ellmid}}(e^-)|=\Theta(n^{-(\alpha-\frac12)})$.
It follows that w.e.p.,
\[
S_1 =
  \frac{\Theta(n^{-(\alpha-\frac12)})}{\Vol(U)}
  \sum_{\ell\leq\ellmax}\Vol(\calP_{\ell}).
\]
Since the $\calP_{\ell'}$'s are disjoint and contained in $U$, 
  we clearly have $\sum_{\ell'\leq\ellmax}\Vol(\calP_{\ell'})\leq\Vol(U)$.
Hence, w.e.p.~$S_1=O(n^{-(\alpha-\frac12)})=o(n^{2\alpha-1})$.

Now, we bound $S_2$.
Assume $t\in B_{O}(\ellmax)\setminus B_{O}(\ellmid)$
  and $s\in\calP_\ell$ with $\ell\leq\ellmid$.
By Proposition~\ref{prop:nonEmptyQ2},  it holds that 
  $|{\calQ}'_{s,t}|=|E(\calP_{\widetilde{\ell}},N_{\calP_{\ellmid}}(t))|$.
By Proposition~\ref{prop:neighBound} part~\eqref{prop:neighBoundPart3},
  w.e.p., $|E(\calP_{\widetilde{\ell}},N_{\calP_{\ellmid}}(t))|
     =\Theta(n^{1-\alpha}e^{-(\alpha-\frac12)(R-\widetilde{\ell})}d(t))$.
Hence, w.e.p.,
\begin{align*}
S_2 = 
& \frac{\Theta(n^{-(1-\alpha)})}{\Vol(U)}\sum_{\ell\leq\ellmid:\widetilde{\ell}=k}
    \Theta(e^{(\alpha-\frac12)(R-\widetilde{\ell})})\Vol(P_{\ell})
    \sum_{\ellmid<\ell'\leq\ellmax}|N_{\calP_{\ell'}}(e^+)|.
\end{align*}
Since the number of neighbors of a vertex is at most its degree and 
  given that, by Proposition~\ref{prop:degreeBound}, 
  w.e.p.~$d(e^+)=\Theta(\sqrt{n})$, we infer that w.e.p.,
\begin{align*}
S_2 & = 
    \frac{1}{\Vol(U)}\Theta(n^{\alpha-\frac12}e^{(\alpha-\frac12)(R-k)})
    \sum_{\ell\leq\ellmid:\widetilde{\ell}=k}\Vol(P_{\ell})
\end{align*}
Clearly,  $\sum_{\ell\leq\ellmid:\widetilde{\ell}=k}\Vol(P_{\ell})\leq\Vol(U)$.
Recalling that $k>\ellmid$, the definition of $\ellmid$ 
  and since $\alpha>\frac12$, 
  we conclude that w.e.p.~$S_2=O(n^{2\alpha-1})$.
\end{proof}

\subsection{A $\calQ''$-flow}
The collection $\calQ''$ will contain 
  paths between distinct vertices $s$ and $t$ of the 
  center component $H$ if and only if
  at most one of $s$ and $t$ belongs to $B_{O}(\ellmax)$.
Paths in $\calQ''$ will have a similar structure as in $\calQ'$; we informally describe
  it first for paths both of whose endvertices $s$ and $t$ belong to 
  $B_{O}(R)\setminus B_{O}(\ellmax)$.
Specifically, such paths will consist of three segments.
The first segment connects $s$ 
  to a vertex $s'$ in $\calP_{\ellmax}$.
We denote this segment by $q_{s,s'}$.
The last segment, connects a vertex $t'$ in $\calP_{\ellmax}$
  to $t$.
We denote it by $q_{t',t}$. 
The middle segment will be a path from $s'$ to $t'$
  belonging to $\calQ'_{s',t'}$ as defined in the previous section.
In fact, the collection of paths from $s$ to $t$, i.e., $\calQ''_{s,t}$, 
  will be paths that first traverse $q_{s,s'}$, then a path
  in $\calQ'_{s',t'}$ and finally the path $q_{t',t}$.
For $q\in\calQ''_{s,t}$, 
  we refer to $q_{s,s'}$ and $q_{t',t}$ as \emph{end segments} of $q_{s,t}$
  and to $q_{s',t'}$ as the \emph{middle segment} of $q$.
If only $s$ belongs to 
  $B_{O}(R)\setminus B_{O}(\ellmax)$, we let 
  $t'=t$ and $q_{t',t}$  be the length $0$ path of the single vertex
  $t$.
We define $s'$ and $q_{s',s}$ similarly if $t$ is 
  in $B_{O}(R)\setminus B_{O}(\ellmax)$.

In order to specify how $s'$ and $t'$ are chosen and paths 
  $q_{s,s'}$ and $q_{t',t}$ defined, we borrow from~\cite{fk15} 
  the following useful concept of ``betweenness''
  (recall that $\Delta\varphi_{p_0,p_1}$ denotes the small
  relative angle in $[0,\pi)$ between $p_0,p_1\in\HH^2$):
  say that vertex $p'$ lies \emph{between} vertices $p$ and $p''$ 
  if $\Delta\varphi_{p,p'}+\Delta\varphi_{p',p''}=\Delta\varphi_{p,p''}$.
Also, given a finite set $\calS\subseteq\HH^2$ and $p,p'\in\calS$
  we say that 
  \emph{$p''$ follows $p$ in $\calS$}, 
  if there is no $p'\in \calS\setminus\{p,p''\}$
  such that $p'$ is between $p$ and $p''$. 
Now, let $u_0, u_1\in\calP_{\ellmax+1}$ be such that 
  $u_1$ follows $u_o$ in $\calP_{\ellmax+1}$ and $s$ is between $u_0$ and $u_1$.
Consider a shortest path in $H$ (ties broken 
  arbitrarily) between $s$ and an element of $\{u_0,u_1\}$
  -- denote the latter element by $u_b$.
We will show that, w.e.p.~$u_b$ has a neighbor in $\calP_{\ellmax}$.
We denote by~$q_{s,s'}$ the oriented path that 
  starts at $s$, traverses the aforementioned shortest path 
  up to $u_b$ and ends in $u_b$'s closest neighbor, 
  henceforth denoted by $s'$, that belongs to $\calP_{\ellmax}$.
Similarly, define $t'$ and $q_{t,t'}$.
Let $q_{t',t}$ equal the latter but with the reverse orientation.

An important fact concerning the just described end segments of paths in 
  $\calQ''$ arises from a 
  key property of geometric graphs, which depending on the model,
  precludes the existence of some vertex-edge configurations. 
In~\cite{fk15}, for hyperbolic geometric graphs,
  two very simple forbidden configurations are identified
  (each one obtained as the contrapositive of the two claims stated 
  in the following result).
\begin{lemma}[{\cite[Lemma~9]{fk15}}]\label{lem:fk}
Let $G=(V,E)$ be a hyperbolic geometric graph.
Let $u,v,w\in V$ be vertices such that $v$ is between $u$ and $w$, and 
  let $uw\in E$.
\begin{enumerate}[(i).-]
\item\label{lem:fkPart1} 
  If $r_v\leq \min\{r_u,r_w\}$, then $\{uv, vw\}\subseteq E$.
\item\label{lem:fkPart2}
  If $r_w\leq r_v\leq r_u$, then $vw\in E$.
\end{enumerate}
\end{lemma}
Our two following results establish, first,
  that w.e.p.~$q_{s,s'}$ with the 
  stated properties does indeed exist in $H$, and second,
  show that the end segment of a path in $\calQ''$ 
  exhibits a very useful property: it is essentially
  contained in a small angular sector to which $s'$ belongs to  
  and except for potentially one internal vertex the path is completely
  contained in $B_{O}(R)\setminus B_{O}(\ellmax)$.
\begin{lemma}\label{lem:links2}
Let $\ell\in\{\ellmax,\ellmax+1\}$.
W.e.p., for any two points $u_0,u_1\in\calP_{\ell}$ 
  such that $u_1$ follows $u_0$ in $\calP_{\ell}$
  it holds that 
  $\Delta\varphi_{u_0,u_1}\leq
     \frac{\upsilon}{n}e^{\alpha(R-\ell)}\log n$.
Moreover, w.e.p., every $u\in\calP_{\ellmax+1}$ has a neighbor
  $v\in\calP_{\ellmax}$ such that
  $\Delta\varphi_{u,v}\leq\frac{\upsilon}{n}e^{\alpha(R-\ellmax)}\log n$.
\end{lemma}
\begin{proof}
Fix $u_0\in\calP_{\ell}$.
Let $R_{u_0}$ be the collection of points 
  $u\in B_{O}(\ell)\setminus B_{O}(\ell-1)$ 
  such that $0<\Delta\varphi_{u_0,u}\leq\frac{\upsilon}{n}e^{\alpha(R-\ell)}\log n$.
By~Lemma~\ref{lem:muBall} and by definition of $\ell$, 
\begin{align*}
& \mu(R_{u_0})
  =\frac{\upsilon}{n}e^{\alpha(R-\ell)}(\log n) e^{-\alpha (R-\ell)}(1-e^{-\alpha})(1+o(1))
  =\omega\big(\tfrac{\log n}{n}\big).
\end{align*}
Hence, by Lemma~\ref{lem:meanMeus} together with a union bound over all 
  $u_0\in\calP_{\ell}$, w.e.p., 
  $R_{u_0}$ is not empty for each $u_0\in\calP_{\ell}$. 

Consider now the second part of the claim.
Let $v_0,v_1\in\calP_{\ellmax}$ be such that $v_1$ follows $v_0$ in 
  $\calP_{\ellmax}$ and $u$ is between $v_0$ and $v_1$.
From the first part, we know that w.e.p.~$\Delta\varphi_{u,v_0}, 
  \Delta\varphi_{u,v_1}\leq\frac{\upsilon}{n}e^{\alpha(R-\ellmax)}\log n$.
By Lemma~\ref{lem:aproxAngle}, we have 
  $\theta_{R}(\ellmax,\ellmax-1)=\Theta(e^{\frac12{(R-2\ellmax)}})
    =\Omega(n^{-2(1-\alpha)})$.
By definition of $\ellmax$ it holds that 
  $\Delta\varphi_{u,v_b}\leq\frac{\upsilon}{n}e^{\alpha(R-\ellmax)}\log n
    =\upsilon n^{-(1-\alpha)(2\alpha+1)}e^{\alpha \nu'}\log n$.
Since $2\alpha+1>2$, we conclude that 
  w.e.p.~$\Delta\varphi_{u,v_b}=o(\theta_{R}(\ellmax,\ellmax-1))$, 
  implying that $u$ and $v_b$ are neighbors in $H$.
\end{proof}

The following result establishes the existence of end segments 
  with certain useful characteristics.
\begin{proposition}\label{prop:charactShortPath}
Let $H=(U,F)$ be the center component of a graph chosen 
  according to $\poimod_{\alpha,C}(n)$.
Let $D$ be the diameter of $H$.
W.e.p.~for every vertex $s\in B_{O}(R)\setminus B_{O}(\ellmax)$ of $H$,
  there is a path in $H$ of length at most $D+1$ with endvertices 
  $s$ and $s'\in\calP_{\ellmax}$
  all of whose  internal vertices, 
  except for at most one, lie outside $B_{O}(\ellmax)$ and 
  determine together with~$s'$ an angle at the origin which is 
  at most $\phi_{\max}:=(1+\frac{1}{e})\frac{\upsilon}{n}e^{\alpha(R-\ellmax)}\log n$.
\end{proposition}
\begin{proof}
Let $u_0$ and $u_1$ be as described in the beginning of this section,
  i.e., $u_0,u_1\in\calP_{\ellmax+1}$ such that $u_1$ follows $u_0$ in 
  $\calP_{\ellmax+1}$ and $s$ is between $u_0$ and $u_1$.
Consider a shortest path in $H$ between $s$ and an element
  of $\{u_0,u_1\}$, say $\widetilde{q}$.
Clearly, $\widetilde{q}$ exists because $H$ is connected.
The length of $\widetilde{q}$ is at most $D$.
Suppose that some internal vertex of $\widetilde{q}$ belongs 
  to $B_{O}(\ellmax)$.
Say $w$ is the first such vertex one encounters when moving along 
  $\widetilde{q}$ beginning at $s$. 
Assume first that $w$ is between $u_0$ and $u_1$.
By Lemma~\ref{lem:links2}, we know that $u_0u_1$ is an edge of $H$, so
  by Lemma~\ref{lem:fk} part~\eqref{lem:fkPart1}, $wu_b$ is an edge of 
  $G$ (and thus of $H$) for any $b\in\{0,1\}$.
Assume then that $w$ is not 
  between $u_0$ and $u_1$ (in particular $w\not\in\{u_0,u_1\}$).
Let $\widetilde{w}$ be the vertex right before $w$ when moving along 
  $\widetilde{q}$ from $s$ to $w$. 
Note that by the choice of $w$, we have that 
  $\widetilde{w} \not\in B_{O}(\ellmax)$. 
Moreover, we may assume that $\widetilde{w}$ and all other vertices before $\widetilde{w}$ when moving along $\widetilde{q}$ beginning at $s$ are between $u_0$ and $u_1$, as otherwise, in the path $\widetilde{q}$, instead of moving to the first vertex  not between $u_0$ and $u_1$, one could by Lemma~\ref{lem:fk} part~\eqref{lem:fkPart1} directly move to $u_b$ for some $b \in \{0,1\}$, contradicting the fact that $\widetilde{q}$ is a shortest path.
Let $b\in\{0,1\}$ be such that $u_b$ is between $w$ and $\widetilde{w}$. 
By Lemma~\ref{lem:fk} part~\eqref{lem:fkPart2}, the edge $wu_b$ belongs to 
  $G$, hence also to $H$.
In summary, all but at most one of $\widetilde{q}$'s internal
  vertices lie outside $B_{O}(\ellmax)$ and in between $u_0$ and $u_1$.
By Lemma~\ref{lem:links2}, it follows that 
  all but one of the vertices of $\widetilde{q}$ determine 
  an angle at the origin with $u_b$ which is at most 
  $\frac{\upsilon}{en}e^{\alpha(R-\ellmax)}\log n$.
Again by Lemma~\ref{lem:links2}, if we concatenate $\widetilde{q}$ with 
  the edge $u_bv$ where $v\in\calP_{\ellmax}$ is as in the statement
  of Lemma~\ref{lem:links2}, we obtain a path $q_{s,v}$ with the desired 
  properties.
\end{proof}
An immediate consequence of the previous result is that
  every path in $\calQ''$ has length at most $2D+5$ where $D$ is 
  the diameter of $H$.

For future reference, we next derive some useful volume estimates,
  one of which involves a natural extension of our neighborhood
  definition.
Specifically, for $w\in U$ consider the set of neighbors $W$
  that belong to $\calP_{\ell}$, i.e., $W=N_{\calP_{\ell}}(w)$.
Denote by $N_{\calP_{\ell'}}(W)$ the set of neighbors of vertices in 
  $W$ that belong to $\calP_{\ell'}$, i.e.,
  $N_{\calP_{\ell'}}(W):=\bigcup_{w\in W}N_{\calP_{\ell'}}(w)$.
\begin{lemma}\label{lem:upsilonRepl}
Let $H=(U,F)$ be the center component of $G=(V,E)$ chosen 
  according to $\poimod_{\alpha,C}(n)$.
Then 
  the following holds w.e.p.:
\begin{enumerate}[(i).-]
\item\label{lem:upsilonRepl1}
If $w\in\calP_{\ellmax}$, then~$\displaystyle\sum_{t\in U\setminus B_{O}(\ellmax): t'=w} d(t)
      = O(\upsilon e^{\alpha(R-\ellmax)}\log n)$.


\item\label{lem:upsilonRepl2}
If $w\in P_{g}$ for some $\ellmid\leq g\leq\ellmax$, then
\[
\sum_{t\in U\setminus B_{O}(\ellmax): t'\in N_{\calP_{\ellmax}}(W)} d(t) 
  = \begin{cases}
  \Theta(\sqrt{n}e^{\frac12(R-\ellmax)}), & 
     \text{if $W=\{w\}$ and $g=\ellmid$,} \\
  \Theta(\sqrt{n}e^{\frac12(R-g)}), & 
     \text{if $W=N_{\calP_{\ellmid}}(w)$.} \\
  \end{cases}
\]
\end{enumerate}
\end{lemma}
\begin{proof}
For the first part, assume $t\in U\setminus B_{O}(\ellmax)$
  is such that $t'=w$.
By Proposition~\ref{prop:charactShortPath} (for $\phi_{\max}$ as defined there),
  w.e.p.~the angle at the origin determined by $t$ and $w$ 
  is at most $\phi:=\phi_{\max}$.
Thus, w.e.p.~$t$ must belong 
  to the $\phi$-sector centered at $w$, henceforth denoted by~$\Phi$,
  and thence to the truncated sector $\Phi\setminus B_{O}(\ellmax)$.
By Lemma~\ref{lem:VolUpperBound} we conclude that w.e.p.,
\[
\sum_{t\in U\setminus B_{O}(\ellmax): t'=w} d(t)
  \leq \sum_{t\in U\cap\Phi\setminus B_{O}(\ellmax)} d(t)
  = \Theta(\phi n) = O(\upsilon e^{\alpha(R-\ellmax)}\log n).
\]

For the second part, let 
  $\phi_W=\inf\{\phi : \text{there is a $\phi$-sector $\Phi\supseteq W$}\}$.
We proceed as in the first part.
Consider $t\in U\setminus B_{O}(\ellmax)$
  such that $t'$ is a neighbor of a vertex in $W$.
Note that the angle between a vertex in $W$ and one of its neighbors 
  in $\calP_{\ellmax}$ is $O(\theta_{R}(\ellmid,\ellmax))$.
As in the first part, the angle at the origin
  determined by $t\in U\setminus B_{O}(\ellmax)$ and $t'$ 
  is at most $\phi_{\max}$.
Hence, the angle at the origin 
  determined by $t$ and $w$ is $\phi
    :=\Theta(\phi_{\max}+\phi_W+\theta_R(\ellmid,\ellmax)).$
If $W=\{w\}$, then  $\phi_W=0$, and hence $\phi=\Theta(\theta_{R}(\ellmid,\ellmax))$, and the first result of the second part follows as before. Similarly, if
    $W=N_{\calP_{\ellmid}}(w)$, then $\phi_w=\Theta(\theta_R(\ellmid,g))$, and hence in this case, $\phi=\Theta(\theta_R(\ellmid,g))$. The argument is once again as in the first part.
\end{proof}

\medskip
The main result of this section is the following:
\begin{proposition}\label{prop:flowPart2}
Let $H=(U,F)$ be the center component of $G=(V,E)$ chosen 
  according to $\poimod_{\alpha,C}(n)$.
For all $q\in\calQ''_{s,t}$, let 
\[
f''(q) :=\frac{d(s)d(t)}{\Vol(U)}\cdot\frac{1}{|\calQ''_{s,t}|}.
\]
Then, w.e.p.~$\calQ''\subseteq\calQ(H)$, $f''$ is a well defined
  $\calQ''$-flow and $\overline{\rho}(f'')
     =O(D'n^{2\alpha-1})$.
\end{proposition}
\begin{proof}
Since $|\calQ''_{s,t}|=|\calQ'_{s',t'}|$,
  when at most one of $s$ and $t$ belongs to $B_{O}(\ellmax)$, 
  Proposition~\ref{prop:nonEmptyQ2} and Proposition~\ref{prop:neighBound},
  imply that $|\calQ''_{s,t}|\neq 0$.
Thus, $f''$ is well defined.
Moreover, by definition $\sum_{q\in\calQ''_{s,t}} f'(q)=d(s)d(t)/\Vol(U)$,
  so $f''$ is a flow.
  
\medskip
We next bound $\overline{\rho}(f'')$, i.e., the elongated
  flow $\overline{f''}(e)$ for each oriented edge $e$ traversed by some path 
  in $\calQ''$.
To facilitate the argument, we classify oriented edges $e$ of $H$ used
  by paths in $\calQ''$ and bound their elongated flows separately.
The edges traversed by middle segments of paths in $\calQ''$ are grouped 
  as in the proof of Proposition~\ref{prop:flowPart1}, i.e., 
  into spread out, belt and belt incident edges (so called middle edges, 
  i.e., edges with both endvertices in $B_{O}(\ellmax)\setminus B_{O}(\ellmid)$,
  are ignored because they are not traversed by paths in
  $\calQ''$).
The edges traversed by end segments of paths in $\calQ''$ will be
  referred to as \emph{remote edges}. 
These edges have at least one endvertex
  in $B_{O}(R)\setminus B_{O}(\ellmax)$.

For bounding elongated flows we use a trivial bound on the length
  of paths in $\calQ''$.
Specifically, we note that by construction end segments of paths 
  in $\calQ''$ have length at most $D+1$ where $D$ is the diameter of 
  the center component $H$.
Since every path in $\calQ'$ has length $3$, it follows 
  that, every path in $\calQ''$ has length
  at most $D':=2D+5$.

Let $e$ be an edge of $H$.
Since $\calQ''_{s,t}=\calQ''_{t,s}$ for every distinct $s,t\in U$,
 the elongated flow $\overline{f''}$ is the same for both orientations of $e$.
Thus, in our ensuing discussion we fix arbitrarily one of 
  the two possible orientations of $e$. 

\medskip\noindent
\textbf{Spread out edges} (one endvertex of $e$ in $B_{O}(\ellmid)$
  and the other one in $B_{O}(\ellmax)\setminus B_{O}(\ellmid)$):
Fix the orientation of $e$ so
  $e^-\in B_{O}(\ellmid)$ and $e^+\not\in B_{O}(\ellmid)$.
The only paths $q\in\calQ''$ that could traverse $e$ are those 
  whose middle segment traverses $e$.
This can happen only when $q$ is of Type I and its first
  edge is $e^-$ (in particular, the initial end segment of $q$ is the length $0$ path $\{e^-\}$).
Assume now that $s$ and $t$ are start- and endvertices of $q$.
Observe that 
  $t\not\in B_{O}(\ellmax)$ since otherwise
  $s,t\in B_{O}(\ellmax)$ contradicting the fact that $q\in\calQ''$. 
Moreover, it must be that 
\begin{inparaenum}[(i).-]
\item $s=e^-\in\calP_k$ for some $k\leq\ellmid$,
\item $q$'s middle segment must be a length $3$ path 
  with  $e^-$ and $t'$ as endvertices, and
\item one internal vertex of $q$'s middle segment is 
  $e^+$ and the other internal vertex belongs to 
  $N_{\calP_{\ellmid}}(e^+)\cap N_{\calP_{\ellmid}}(t')$.
\end{inparaenum}
Hence, there are at most 
  $|N_{\calP_{\ellmid}}(e^+)\cap N_{\calP_{\ellmid}}(t')|
    \leq |N_{\calP_{\ellmid}}(t')|$ feasible middle segments of~$q$.
Thence,
\begin{align*}
\overline{f''}(e) & \leq \frac{D'd(e^-)}{\Vol(U)}
   \sum_{t\in U\setminus B_{O}(\ellmax)}
   \sum_{q\in\calQ''_{e^-,t}: q\ni e} \frac{d(t)}{|\calQ''_{e^-,t}|} \\
&  \le \frac{D'd(e^-)}{\Vol(U)}
   \sum_{t\in U\setminus B_{O}(\ellmax): \exists q\in\calQ''_{e^-,t}, q\ni e}
      \frac{d(t)}{|\calQ''_{e^-,t}|}|N_{\calP_{\ellmid}}(t')|.
\end{align*}
Assume $t\in U\setminus B_{O}(\ellmax)$.
The way we built $\calQ''$, Proposition~\ref{prop:nonEmptyQ2}
  and Proposition~\ref{prop:neighBound} part~\eqref{prop:neighBoundPart3}, 
  imply that
  $|\calQ''_{e^-,t}|=|\calQ'_{e^-,t'}|
      =|E(\calP_{\widetilde{k}},N_{\calP_{\ellmid}}(t'))|
      =\Theta(|N_{\calP_{\ellmid}}(t')|
      e^{-(\alpha-\frac12)(R-\widetilde{k})}\sqrt{n})$.
Now, observe that if $q\in\calQ''_{e^-,t}$
  traverses $e$, then $t'$ is a neighbor of some vertex
  in $W:=N_{\calP_{\ellmid}}(e^+)$.
It follows that, w.e.p.,
\[
\overline{f''}(e) = \Theta\big(\tfrac{D'd(e^-)}{\Vol(U)}\cdot 
   \tfrac{1}{\sqrt{n}}
   e^{(\alpha-\frac12)(R-\widetilde{k})}\sum_{t\in U\setminus B_{O}(\ellmax): t'\in N_{\calP_{\ellmax}}(W)} d(t)\big).
\]
Also, by Proposition~\ref{prop:degreeBound}, 
  w.e.p.~$d(e^-)=\Theta(e^{\frac12(R-k)})$, so applying Lemma~\ref{lem:upsilonRepl} we deduce that w.e.p.~$\overline{f''}(e) 
   = O\big(\tfrac{D'}{\Vol(U)}ne^{-\frac{k}{2}}e^{\alpha(R-\widetilde{k})}\big)$.
Furthermore, by definition of $\ellmax$ and since $\alpha>\frac12$, 
  we have $\alpha(R-\widetilde{k})-\frac{k}{2}\leq \max\{(\alpha-\frac12)k,\alpha(R-\ellmax)\}=\alpha(R-\ellmax)$.
Since by Lemma~\ref{lem:VolUpperBound}, w.e.p.~$\Vol(U)=\Theta(n)$, 
  recalling that $\alpha<1$ and the definition of $\ellmax$,
  we conclude that w.e.p.~$\overline{f''}(e)=O(D'e^{\alpha(R-\ellmax)})
  = O(D'n^{\alpha(2\alpha-1)}e^{\alpha\nu'})=o(D'n^{2\alpha-1})$.

\medskip\noindent
\textbf{Belt edges} (both endvertices of $e$ in $\calP_{\ellmid}$):
The only paths $q\in\calQ''$ that could traverse $e$ are those 
  whose middle segment have $e$ as a middle edge.
This can happen only if $q$ is a Type II path.
Assume $q\in\calQ''_{s,t}$ traverses $e$.
Then, $s'$ must be a neighbor of $e^-\in\calP_{\ellmid}$ (in particular, 
  $s\not\in B_{O}(\ellmid)$).
Similarly, it must 
  be that $t'$ is a neighbor of $e^+\in\calP_{\ellmid}$ (in particular, 
  $t\not\in B_{O}(\ellmid)$).
By definition of $\calQ''$ and Proposition~\ref{prop:nonEmptyQ2},
  we have $|\calQ''_{s,t}|=|\calQ'_{s',t'}|=
     |N_{\calP_{\ellmid}}(s')|\cdot |N_{\calP_{\ellmid}}(t')|$.
Applying Proposition~\ref{prop:neighBound} 
  part~\eqref{prop:neighBoundPart1} and recalling 
  the definition of $\ellmid$, we get that 
  w.e.p.~$|\calQ''_{s,t}|=\Theta(n^{-(2\alpha-1)}d(s')d(t'))$.
Hence, w.e.p.,
\begin{align*}
\overline{f''}(e) & \leq \tfrac{D'}{\Vol(U)}
  \sum_{s,t\in U\setminus B_{O}(\ellmid)}\sum_{q\in\calQ'': q\ni e} 
  \frac{d(s)d(t)}{|\calQ''_{s,t}|} \\
&  
  \leq \Theta\big(\tfrac{D'}{\Vol(U)}n^{2\alpha-1}\big)
  \sum_{s\in U\setminus B_{O}(\ellmid) : s'e^-\in F}\frac{d(s)}{d(s')} 
  \sum_{t\in U\setminus B_{O}(\ellmid) : e^+t'\in F}\frac{d(t)}{d(t')}. 
\end{align*}
Note that $s'=s$ if $s\in U\cap B_{O}(\ellmax)$ and $s'\in\calP_{\ellmax}$
  otherwise.
Assume $s\in U\setminus B_{O}(\ellmax)$ is such that 
  $s'$ is a neighbor of $e^-$ in $H$.
By Proposition~\ref{prop:degreeBound}, 
  w.e.p.~$d(e^-)=\Theta(\sqrt{n})$
  and $d(s')=\Theta(e^{\frac12(R-\ellmax)})$.
By definition of $\ellmax$ and considering the 
  cases $s=s'$ and $s \neq s'$ separately (applying 
  Lemma~\ref{lem:upsilonRepl} in the latter), it follows that w.e.p.,
\begin{align*}
\sum_{s\in U\setminus B_{O}(\ellmid) : s'e^-\in F}\frac{d(s)}{d(s')}
  & \leq 
  d(e^-)
  +
  \Theta(e^{-\frac12(R-\ellmax)})\sum_{s\in U\setminus B_{O}(\ellmax): s'e^-\in F} d(s)
    = O(\sqrt{n}).
\end{align*}
The same argument shows that w.e.p.~$\sum_{t\in U\setminus B_{O}(\ellmid) : e^+t'\in F}\frac{d(t)}{d(t')}=O(\sqrt{n})$.
Applying Lemma~\ref{lem:VolUpperBound},
  w.e.p.~$\Vol(U)=\Theta(n)$, we conclude that 
  w.e.p.~$\overline{f''}(e) = O(D'n^{2\alpha-1})$.

\medskip\noindent
\textbf{Belt incident edges}
(one endvertex of $e$ in $\calP_{\ellmid}$ and the other one in 
  $B_{O}(\ellmax)\setminus B_{O}(\ellmid)$):
Let us fix the orientation of $e$ so $e^+\in\calP_{\ellmid}$.
Let $k>\ellmid$ be such that $e^-\in\calP_{k}$.

Let  $q\in\calQ''$ be a path that traverses $e$.
Since $e$ has both its endvertices in $B_{O}(\ellmax)$, $e$ must 
  belong to the middle segment of $q$.
By definition of $\calQ'$, one of the following must hold:
\begin{inparaenum}[(i).-]
\item\label{it:beltinc1} $e$ is the first edge of a Type II path, or
\item\label{it:beltinc2} $e$ is the first edge of a Type III path, or 
\item\label{it:beltinc3} $e$ is the middle edge of a Type III path.
\end{inparaenum}
Assume $q\in\calQ''_{s,t}$ where $s,t\in U$.
We make the following observations concerning each one of the 
  three situations just identified:
\begin{enumerate}[(i).-]
\item It must hold that $s'=e^-$ and $t\not\in B_{O}(\ellmid)$
  (otherwise, $q$ can not be of Type II).
By Proposition~\ref{prop:nonEmptyQ2}
  we have $|\calQ''_{s,t}|
    =|\calQ'_{e^-,t'}|=|N_{\calP_{\ellmid}}(e^-)|\cdot |N_{\calP_{\ellmid}}(t')|$.
Note also, that the paths in $\calQ''_{s,t}$ that traverse $e$
  are in one to one correspondence with $E(\{e^+\},N_{\calP_{\ellmid}}(t'))$,
  so there are $|N_{\calP_{\ellmid}}(t')|$ of them (since $\calP_{\ellmid}$
  induces a clique in $H$).
By Proposition~\ref{prop:neighBound} part~\eqref{prop:neighBoundPart1}, 
  we infer that, w.e.p.~the fraction of paths in $\calQ''_{s,t}$
  that traverse $e$ is $\Theta(n^{\alpha-\frac12}e^{-\frac12(R-k)})$.

\item 
It must hold that $s'=e^-$ and 
  $t\in\calP_{\ell}$ for some $\ell\leq\ellmid$.
In fact, $s\not\in B_{O}(\ellmax)$ so $e^-$ must belong to $\calP_{\ellmax}$
  (since otherwise both $s,t\in B_{O}(\ellmax)$ contradicting the
  fact that $q\in\calQ''$).
By Proposition~\ref{prop:nonEmptyQ2}, we now have
    $|\calQ''_{s,t}|=|\calQ'_{e^-,t}|
       =|E(N_{\calP_{\ellmid}}(e^-),\calP_{\widetilde{\ell}})|$.
So, by Proposition~\ref{prop:neighBound} part~\eqref{prop:neighBoundPart1}
  and~\eqref{prop:neighBoundPart3},
  w.e.p.~$|\calQ''_{s,t}|
     =\Theta(n^{1-\alpha}e^{-(\alpha-\frac12)(R-\widetilde{\ell})}e^{\frac12(R-\ellmax)})$.
Note also that the paths in $\calQ''_{s,t}$ that traverse $e$
  are in one to one correspondence with 
  $N_{\calP_{\widetilde{\ell}}}(e^+)$, so 
  by Proposition~\ref{prop:neighBound} part~\eqref{prop:neighBoundPart1},
  w.e.p., there are $\Theta(\sqrt{n}e^{-(\alpha-\frac12)(R-\widetilde{\ell})})$
  of them.

\item 
Now, it must hold that $s\in\calP_{\ell}$ for some $\ell\leq\ellmid$
  such that $\widetilde{\ell}=k$ and $t\in U\setminus B_{O}(\ellmax)$
  (since otherwise both $s,t\in B_{O}(\ellmax)$ contradicting the
  fact that $q\in\calQ''$).
By Proposition~\ref{prop:nonEmptyQ2} 
  we have that w.e.p.~$|\calQ''_{s,t}|=|\calQ'_{s,t'}|
     =|E(\calP_{k},N_{\calP_{\ellmid}}(t'))|$.
Hence, by Proposition~\ref{prop:neighBound} part~\eqref{prop:neighBoundPart1}
  and part~\eqref{prop:neighBoundPart3}, 
  w.e.p.~$|\calQ''_{s,t}|=\Theta(n^{1-\alpha}e^{-(\alpha-\frac12)(R-k)}d(t'))$.
  Moreover, 
if $t'e^+ \in F$, then there is exactly one path in $\calQ''_{s,t}$ that traverses $e$.
\end{enumerate}

\smallskip
The contribution of case~\eqref{it:beltinc1} to $\overline{f''}(e)$
  is, w.e.p.,
\begin{align*}
S_1 & :=  \frac{D'}{\Vol(U)} \sum_{s \in U\setminus B_{O}(\ellmid), s'=e^-} d(s) \sum_{t \in U \setminus B_O(\ellmid)}\frac{d(t)}{|\calQ''_{s,t}|}|N_{\calP_{\ellmid}}(t')|\\
&=\frac{D'}{\Vol(U)}O\big(
       n^{\alpha-\frac12}e^{-\frac12(R-k)}\sum_{s\in U\setminus B_{O}(\ellmid):s'=e^-}d(s)\sum_{t\in U\setminus B_{O}(\ellmid)}d(t)\big).
\end{align*}
Clearly, $\sum_{t \in U \setminus B_O(\ellmid)}d(t) \le \Vol(U).$ 
If $e^-\not\in\calP_{\ellmax}$, by Proposition~\ref{prop:degreeBound}, 
  w.e.p.~we have that $\sum_{s \in U\setminus B_{O}(\ellmid):s'=e^-}d(s)=d(e^-)=\Theta(e^{\frac12(R-k)})$. 
Hence, in this case, 
  $S_1=O(D'n^{\alpha-\frac12})=o(n^{2\alpha-1})$, since $\alpha > \frac12$.
Otherwise, that is,  if $e^-\in\calP_{\ellmax}$ (thence, $k=\ellmax$), by Proposition~\ref{prop:degreeBound}, Lemma~\ref{lem:upsilonRepl}
  and given that $\alpha >\frac12$,
  w.e.p.~$\sum_{s\in U\setminus B_{O}(\ellmid):s'=e^-}d(s)
    = d(e^-)+\sum_{s\in U\setminus B_{O}(\ellmax):s'=e^-}d(s)
    = O(\upsilon e^{\alpha(R-\ellmax)}\log n)$.
Hence, in this case, using that $\frac12<\alpha<1$, w.e.p.,
\[
S_1 = O(D'  n^{\alpha-\frac12}e^{(\alpha-\frac12)(R-\ellmax)}\upsilon \log n)=O(D'n^{\alpha(2\alpha-1)}e^{(\alpha-\frac12)\nu'}\upsilon \log n)=o(D'n^{2\alpha-1}).
\]

\smallskip
The contribution of case~\eqref{it:beltinc2} to $\overline{f''}(e)$
 is, w.e.p.,
\begin{align*}
S_2 & := \frac{D'}{\Vol(U)}O\Big(
       n^{\alpha-\frac12}e^{-\frac12(R-\ellmax)}
       \sum_{\ell\leq\ellmid}\Vol(\calP_{\ell})
       \sum_{s\in U\setminus B_{O}(\ellmax): s'=e^-}d(s)\Big).
\end{align*}
Clearly, $\sum_{\ell\leq\ellmid}\Vol(\calP_\ell)\leq\Vol(U)$.
By definition of $\ellmax$ and Lemma~\ref{lem:upsilonRepl},
  we get that 
  w.e.p.~$$S_2=O(D'n^{\alpha-\frac12}\upsilon e^{(\alpha-\frac12)(R-\ellmax)}\log n)
     = O(D'\upsilon n^{\alpha(2\alpha-1)}e^{(\alpha-\frac12)\nu'}\log n).$$
Since $\frac12<\alpha<1$, we conclude that w.e.p.~$S_2=o(D'n^{2\alpha-1})$.

\smallskip
The contribution of case~\eqref{it:beltinc3} to $\overline{f''}(e)$
  is, w.e.p.,
\begin{align*}
S_3 & := \frac{D'}{\Vol(U)}O\Big(
       n^{-(1-\alpha)}e^{(\alpha-\frac12)(R-k)}\sum_{\ell:\widetilde{\ell}=k}\Vol(\calP_{\ell})\sum_{t\in U\setminus B_{O}(\ellmax): e^+t'\in F}\frac{d(t)}{d(t')}\Big).
\end{align*}
By Proposition~\ref{prop:degreeBound}, w.e.p.~$d(t')=\Theta(e^{\frac12(R-\ellmax)})$.
So, by Lemma~\ref{lem:upsilonRepl}, it follows that w.e.p.,
\begin{align*}
\sum_{t\in U\setminus B_{O}(\ellmax) : e^+t'\in F}\frac{d(t)}{d(t')}
  &    = O(\sqrt{n}).
\end{align*}
Clearly, $\sum_{\ell:\widetilde{\ell}=k}\Vol(\calP_{\ell})\leq\Vol(U)$.
Since  $k>\ellmid$, we conclude that 
  w.e.p.~$S_3=O(D'n^{2\alpha-1})$.

\medskip\noindent
\textbf{Remote edges} (at least one endvertex of $e$ belongs to 
  $B_{O}(R)\setminus B_{O}(\ellmax)$):
Assume $q\in\calQ''$ traverses $e$.
Since no path in $\calQ'$ uses a vertex not in $B_{O}(\ellmax)$, 
  edge $e$ must be traversed by one of the end segments of $q$.
Note that there is an endvertex in $\calP_{\ellmax}$, say $v$, which is common
  to all end segments of paths in $\calQ''$ that traverse $e$. 
Since for $s\in U\setminus B_{O}(\ellmax)$
  and $t\in U$, the fraction of paths in $\calQ''_{s,t}$ that traverse $e$ is 
  trivially at most $1$, we infer that w.e.p.,
\begin{align*}
& \overline{f''}(e)
  \leq \frac{2D'}{\Vol(U)}\sum_{s\in U\setminus B_{O}(\ellmax)}
  \sum_{t\in U}\sum_{q\in\calQ''_{s,t}:q\ni e} \frac{d(s)d(t)}{|\calQ''_{s,t}|}
  \leq 2D'\sum_{s\in U\setminus B_{O}(\ellmax):s'=v}d(s).
\end{align*}
(The factor $2$ above follows from the fact that $v$ belongs to either the start- or end segment of a $\calQ''$-path that 
  traverses $e$.
By Lemma~\ref{lem:upsilonRepl}, the 
  definition of $\ellmax$ and since $\frac12<\alpha<1$, it follows
  that w.e.p.~$\overline{f''}(e)=O(D'\upsilon e^{\alpha\nu'}
     n^{\alpha(2\alpha-1)}\log n)=o(D'n^{2\alpha-1})$.
\end{proof}

\subsection{A $\calQ$-flow of moderate elongated length}
Below we derive the main theorem and a corollary that follows easily from
  the results of the previous sections and some results 
  found in the literature.

\medskip  
\noindent\textit{Proof of Theorem~\ref{th:main}}.
Let $H=(U,F)$ be the center component of 
  $G=(V,E)$ chosen according to $\poimod_{\alpha,C}(n)$. 
By~Corollary~\ref{cor:volTobias}, w.e.p.~$\Vol(U)=\Omega(n)$, 
  so the stated lower
  bound is a direct consequence of Corollary~\ref{cor:flowEmbed},
  Proposition~\ref{prop:flowPart1} and 
  Proposition~\ref{prop:flowPart2}.
\qed

\medskip
By Theorem~\ref{th:main}
  and Theorem~\ref{thm:fk1Diam} we immediately obtain the following:
\begin{cor}\label{cor:main}
If $J$ is the giant component of
  $G$ chosen according to $\poimod_{\alpha,C}(n)$, then w.e.p.,
\[
\lambda_1(J) =\Omega\big(n^{-(2\alpha-1)}/(\log n)^{\frac{1}{1-\alpha}}\big).
\]
\end{cor}



\section{Lower bound on the conductance}\label{sec:conductance}
In this section we will establish that the lower bound on the conductance 
  obtained in Section~\ref{sec:gap} can only be attained by relatively
  large sets.
In other words, our goal is to show Theorem~\ref{thm:conductancia}.
In order to derive the theorem, we first prove a few auxiliary lemmas. 
We begin by establishing that if for a fixed set $S\subseteq U$ there are two
  bands, both being relatively 
  far from the boundary of $B_{O}(R)$, one of them having a
  large fraction of $S$, and the other having a large fraction of
  $\overline{S}$, then $|\partial S|$ must be fairly large.

Henceforth, for $b\in\{0,1\}$ and $S\subseteq U$, 
  denote $S$ and $\overline{S}$ by $S^0$ and $S^1$, respectively. 
We fix the following parameter:
  $$\ellbdr := \lfloor R-\tfrac{2\log R}{1-\alpha}\rfloor.$$ 
Recall that Remark~\ref{rem:HtoG} guarantees that
  all vertices in $B_{O}(\ellbdr)$ are,
  w.e.p., part of the center component.

\medskip
\begin{lemma}\label{lem:earlylayersfull}
Let $H=(U,F)$ be the center component of $G=(V,E)$ chosen according to 
  $\poimod_{\alpha,C}(n)$.

Let $\omega_0$ be a function tending to infinity so that $\omega_0=e^{o(\log \log n)}$ but also $\omega_0=\omega(\upsilon)$,\footnote{The condition of $\omega_0=e^{o(\log \log n)}$ while at the same time $\omega_0=\omega(\upsilon)$ clearly implies a corresponding upper bound on $\upsilon$. Nevertheless, all previous results still hold.} and define
  $\epsilon:=\frac{1}{\omega_0} (\log n)^{- \frac{1+\alpha}{1-\alpha}}$.
  Let $\Phi$ be a sector of $B_{O}(R)$ of angle
  $\phi \ge \frac{\upsilon}{\epsilon\omega_0}=\frac{\upsilon}{n}(\log n)^{\frac{1+\alpha}{1-\alpha}}$, and
 let $\ell_{\phi}:= \lceil\frac{1}{2}R+\log\frac{1}{\phi}-2\rceil$. Let $\ell_{\phi} < \ell^* \le \ellbdr$. 
If for some $b \in \{0,1\}$,
\[
\frac{|S^b\cap\Phi\cap\calP_{\ell^*}|}{\EE|\Phi\cap\calP_{\ell^*}|} - \frac{|S^b\cap\Phi\cap\calP_{\ell_{\phi}}|}{\EE|\Phi\cap\calP_{\ell_{\phi}}|} \ge \epsilon,
\] 
then w.e.p.~$|E(S^b \cap \Phi, S^{1-b} \cap \Phi)|=\Omega((\phi n)^{2(1-\alpha)}\frac{\upsilon}{\omega_0} (\log n)^{-\frac{2}{1-\alpha}})$.
The same conclusion holds if in the hypothesis the roles 
  of $\ell_{\phi}$ and $\ell^*$ are interchanged.
\end{lemma}

\begin{proof}
Define $\tilde{\epsilon}:=\frac{\epsilon}{R}.$
First, for some $\ell$ with $\ell_{\phi} < \ell \le \ellbdr$, we bound 
  from below $|E(S^b \cap \Phi, S^{1-b} \cap \Phi)|$ 
  under the assumption 
\begin{equation}\label{eqn:earlylayersfull}
\frac{|S^b\cap\Phi\cap\calP_{\ell}|}{\EE|\Phi\cap\calP_{\ell}|} - \frac{|S^b\cap\Phi\cap\calP_{\ell-1}| }{\EE|\Phi\cap\calP_{\ell-1}|} \ge \tilde{\epsilon}.
\end{equation}
Consider an angle equipartition 
  $\Phi_1,\ldots,\Phi_N$ of $\Phi$ 
 where 
  $N:=\big\lceil \frac{\phi}{\theta_R(\ell,\ell)}\upsilon^{-1}\log n\big\rceil$.
Since 
\begin{align}
& |E(S^b\cap\Phi, S^{1-b} \cap \Phi)|
     \geq |E(S^b\cap\Phi\cap (\calP_{\ell-1}\cup\calP_{\ell}), 
     S^{1-b}\cap\Phi\cap (\calP_{\ell-1}\cup\calP_{\ell}))| \nonumber
\\ & \mbox{} \qquad \label{eq:earlyLayerCutBound}
  \geq\sum_{i\in [N]} |E\left(S^b\cap\Phi_i\cap\left(\calP_{\ell-1}\cup\calP_{\ell}\right),S^{1-b}\cap\Phi_i\cap\left(\calP_{\ell-1}\cup\calP_{\ell}\right)\right)|,
\end{align}
it suffices to bound from below the summation in the latter expression.

For $i\in [N]$, let $m_i:=|\Phi_i\cap\calP_{\ell-1}|$.
Also, let $m$ be the expected number of elements of 
  $\calP_{\ell-1}$ that belong
  to a given $\frac{2\pi}{N}$-sector of $\Phi$.
Define $m'_i$ and $m'$ similarly but replacing $\ell-1$ by $\ell$.
By Remark~\ref{rem:aproxAngle}, Corollary~\ref{cor:medidalb} 
  and Lemma~\ref{lem:meanMeus} and our upper bound on $\ell$, 
  since $$\EE m_i=\Theta\big(\tfrac{\phi n}{N}e^{-\alpha (R-\ell+1)}\big)
    =\Theta\big(\tfrac{\upsilon}{\log n}e^{(1-\alpha)(R-\ell)}\big)
    =\Omega(\upsilon\log n)=\omega(\log n)$$
  for every $i$, w.e.p., 
  $m_i=(1+o(1))m$ and $m'_i=(1+o(1))m'$.
Also, let $\delta_{i}$ denote the fraction of vertices in $S^b$ that 
  belong to $\Phi_i\cap\calP_{\ell-1}$, i.e.,
  $\delta_i=\frac{1}{m_i}|S^b\cap\Phi_i\cap\calP_{\ell-1}|$, 
  and define $\delta'_i$ similarly again replacing $\ell-1$ by $\ell$.
Since each $\Phi_i$ is a sector of angle $\frac{2\pi}{N}\le\theta_{R}(\ell,\ell)$, 
  if a pair of vertices belongs
  to $\Phi_i\cap (\calP_{\ell-1}\cup\calP_{\ell})$, 
  then they must be neighbors in $G$ (and thus also in $H$).
Hence, w.e.p., the $i$-th term of the summation in~\eqref{eq:earlyLayerCutBound}
  is $(1+o(1))(\delta_im+\delta'_im')((1-\delta_i)m+(1-\delta'_i)m')$.

Moreover, observe that the constraint in~\eqref{eqn:earlylayersfull}
  is equivalent to 
  $$
  \frac{\sum_{i \in [N]}\delta_i'm_i'}{\EE|\Phi\cap\calP_{\ell}|} - \frac{\sum_{i \in [N]}\delta_i m_i }{\EE|\Phi\cap\calP_{\ell-1}|} \ge \tilde{\epsilon},
  $$
and w.e.p.~it is stricter than the constraint
  $\frac{1}{N}\sum_{i \in [N]}(\delta_i'-\delta_i)\ge \tilde{\epsilon}(1+o(1))$.
Thus, a lower bound as the one we seek can be derived by bounding from below
  the optimum of the following problem:
\begin{equation*}
\begin{aligned}
& \min & &
  \sum_{i\in [N]}(\delta_im+\delta'_im')((1-\delta_i)m+(1-\delta'_i)m')
\\
& \text{s.t.}
& &
\frac{1}{N}\sum_{i \in [N]}(\delta_i'-\delta_i)
   \geq \tilde{\epsilon}(1+o(1)).
\end{aligned}
\end{equation*}
The minimum of a concave function over a bounded polyhedral 
  domain is attained at a vertex of the polytope.
It is not hard to see that any vertex of the polytope obtained by
  intersecting the hypercube and a half-space has all its
  coordinates equal to $0$ or $1$, except for at most one 
  coordinate.
It follows that the minimization problem stated above attains 
  its minimum when at most one among
  $\delta_1,\ldots,\delta_N,\delta'_1,\ldots,\delta'_N$ is 
  distinct from $0$ or $1$.

Now, if $\tilde{\epsilon}N \ge 2$, there must exist 
  at least $\tilde{\epsilon}N-1$ indices $i$ such that for these indices 
  $\delta_i$ is set to $1$ and $\delta'_i$ is equal to $0$. 
If $\tilde{\epsilon}N  < 2$,
  there exists one index $i$ such that $\delta_i - \delta'_i \ge (1+o(1))\tilde{\epsilon}N/2$. 
Since the function to be optimized is concave in each $\delta_i$ and $\delta'_i$, under this restriction the minimum is attained when for this index $i$ we have $\delta_i=(1+o(1))\tilde{\epsilon}N/2$ and $\delta'_i=0$, or $\delta_i=1$ and $\delta'_i=1-(1+o(1))\tilde{\epsilon}N/2$. 
%
In all cases, the value of the optimization problem is 
  $\Omega(\tilde{\epsilon}mm'N)$.
To conclude, note that $N=\Theta\big(\phi n\upsilon^{-1}e^{-(R-\ell)}\log n\big).$ 
By Corollary~\ref{cor:medidalb}, 
  we have $m'N=\Theta(\phi n e^{-\alpha(R-\ell)})$.
Moreover, $m=\Theta(m')$. 
Thus, w.e.p., $|E(S^b\cap\Phi, S^{1-b} \cap \Phi)|
    =\Omega(\tilde{\epsilon}\phi n e^{-(2\alpha-1)(R-\ell)}\frac{\upsilon}{\log n})$. 
 The conclusion of the lemma then
  follows from noting that $\ell^*-\ell_{\phi}\leq R=O(\log n)$, 
  and hence there must exist two consecutive values of $\ell-1$ and $\ell$ 
  whose difference in terms of the fractions of $S^b$ is at least 
  $\tilde{\epsilon}$.
Recalling that $\ell > \ell_{\phi}$ and that our lower bound on 
  $|E(S^b\cap\Phi, S^{1-b} \cap \Phi)|$ is increasing in $\ell$, we are done 
  for the first part. To conclude, observe that the roles of $\ell_{\phi}$ and $\ell^*$ can be 
  interchanged in the proof above.
\end{proof}

We extend the definition of $h(S)$ as follows: for a region $\calR \subseteq B_O({R})$ 
and a set $S$ with $\Vol(S)=O(n^{1-\varepsilon})$ for some $\varepsilon > 0$,  we set
$$
h_{\calR}(S)=\frac{|E(S \cap \calR, \overline{S})|+|E(\overline{S} \cap \calR, S)|}{\Vol(S \cap \calR)}.
$$

Suppose now that given a fixed set $S\subseteq U$ 
  we could find a collection $\calA$ of regions 
  of $B_{O}(R)$ such that 
  \begin{inparaenum}[(i).-]
  \item $h_{\calR}(S)$
  is moderately large for all $\calR\in\calA$,
  \item $\Vol(S\cap\cup_{\calR\in\calA}\calR)$ is a reasonably large 
  fraction of $\Vol(S)$, and
  \item no edge in $\partial S$ is counted more than $O(1)$ times 
  in $\sum_{\calR\in\calA}(|E(S\cap\calR,\overline{S})|+|E(\overline{S}\cap\calR, S)|)$.
  \end{inparaenum} 
Then, since w.e.p.~$\Vol(S) \le \Vol(\overline{S})$
  (note that by Corollary~\ref{cor:volTobias}, $\Vol(U)=\Omega(n)$, and by
  assumption $\Vol(S)=O(n^{\varepsilon})$), and noting 
  that for any positive numbers $a,b,c,d$ we have 
  $\frac{a+b}{c+d} \geq \min\{\frac{a}{b}, \frac{c}{d}\}$, 
  it will then follow that
\begin{align}\label{eq:conduct}
  & h(S)=\frac{|\partial S|}{\Vol(S)} =
    \Omega\Big(\frac{\sum_{\calR\in\calA}\Vol(S\cap\calR)}{\Vol(S)}\Big)\cdot
    \frac{\sum_{\calR\in\calA}\big(|E(S\cap\calR,\overline{S})|+|E(\overline{S}\cap\calR,S)|\big)}{\sum_{\calR\in\calA}\Vol(S\cap\calR)} \\ & \mbox{} \qquad \nonumber =
\Omega\Big(\frac{\sum_{\calR\in\calA}\Vol(S\cap\calR)}{\Vol(S)}\Big)\cdot
\min_{\calR\in\calA} h_{\calR}(S).
\end{align}
If we can do as above for an arbitrary set $S$ 
  such that $\Vol(S)=O(n^{1-\epsilon})$, then we would be done. 
Below, we develop such an approach.

\medskip
Next, we show that if there is a sufficient quantity of 
  vertices of a fixed set $S$ in a certain sector $\Phi$ of $B_{O}(R)$, 
  and all such vertices are relatively close to the 
  boundary of $B_{O}(R)$ (henceforth referred to as simply the boundary),
  then there must be a large (relative to $\Vol(S)$)
  number of edges between $S\cap\Phi$
  and $\overline{S}\cap\Phi$. 
The intuitive reason for this is the following: in most small
  angles inside the sector there must exist some vertex a bit further
  away from the boundary belonging to $\overline{S}$, and therefore
  within every such angle we find already one cut edge, therefore
  yielding a large total number of cut edges.

\begin{lemma}\label{lem:cuna}
Let $H=(U,F)$ be the center component of 
  $G=(V,E)$ chosen according to $\poimod_{\alpha,C}(n)$. 
Let $\epsilon$ and $\phi$ be as in Lemma~\ref{lem:earlylayersfull}.
If $S\subseteq U$ 
  and a $\phi$-sector $\Phi$ of $B_{O}({R})$  are such that $|S^b\cap\Phi|=\Omega(\EE|\Phi\cap\calP_{\ellbdr}|)$ 
  and $|S^{b}\cap\Phi\cap\calP_{\ellbdr}| 
      \leq \epsilon\EE|\Phi\cap\calP_{\ellbdr}|$
  for some $b\in\{0,1\}$, then 
  w.e.p.~$|E(S^b\cap\Phi, S^{1-b})|=\Omega(\epsilon\EE|\Phi\cap\calP_{\ellbdr}|)$. 
\end{lemma}
\begin{proof}
%
Recall that we say that 
  $v$ follows $w$ in $P_{\ellbdr}$ if $v,w\in\calP_{\ellbdr}$ 
  and there is no other vertex in $\calP_{\ellbdr}$ between $v$ and $w$. 
Our first goal is to find sufficiently many pairs $v,w \in \Phi\cap \calP_{\ellbdr}$ such that $v$ follows $w$, and moreover, $v$ and $w$ are in 
  $S^{1-b}$.
Note that w.e.p.~(again by Corollary~\ref{cor:medidalb} and 
  Lemma~\ref{lem:meanMeus}) we have $\Delta\varphi_{v,w} 
    \le \theta_R(\ellbdr-1,\ellbdr-1)\le \frac{\upsilon}{n}(\log n)^{\frac{2\alpha}{1-\alpha}+1}$.
Thus, w.e.p., by Lemma~\ref{lem:meanMeus}, the number of vertices in 
  $\calP\setminus\calP_{\ellbdr}$ 
  between $v$ and $w$ is $\upsilon (\log n)^{\frac{2\alpha}{1-\alpha}+1}
    =\upsilon(\log n)^{\frac{1+\alpha}{1-\alpha}}=\frac{\upsilon}{\epsilon\omega_0}$. 
Hence, since by hypothesis 
  $|S^b\cap\Phi\cap\calP_{\ellbdr}|\le\epsilon\EE|\Phi\cap P_{\ellbdr}|$,
  w.e.p.~there are $O(\epsilon\EE|\Phi\cap\calP_{\ellbdr}|)$ 
  pairs $v,w$ in $\Phi \cap \calP_{\ellbdr}$ so that $v$ follows $w$ and moreover both $v,w \in S^b$,
  each pair defining a region of $B_{O}(R)$ corresponding to 
  a sector with $v,w$ on its boundary.
Thus, by our choice of $\epsilon$ (recall that $\omega_0=\omega(\upsilon)$), 
  the number of vertices that belong to 
  $\calP\cap\Phi$ which are
  between two vertices in $S^b\cap\Phi\cap\calP_{\ellbdr}$ 
  is $o(\EE|\Phi\cap\calP_{\ellbdr}|)$.
The same holds also for those pairs $v,w$ where one belongs to 
  $S^b$ and the other to $S^{1-b}$. 
However, since $|S^b\cap\Phi|=\Omega(\EE|\Phi\cap\calP_{\ellbdr}|)$,
  most of the vertices in $S^b\cap\Phi$
  must be in regions between two vertices belonging
  to $S^{1-b}\cap\Phi\cap\calP_{\ellbdr}$.

Note also that, since by Lemma~\ref{lem:aproxAngle}, 
  $\theta_{R}(\ellbdr,\ellbdr)=\Theta(\frac{1}{n}(\log n)^{\frac{2}{1-\alpha}})$, and
  $\Delta\varphi_{v,w}\le \frac{\upsilon}{n}(\log n)^{\frac{1+\alpha}{1-\alpha}}
     =o(\frac{1}{n}(\log n)^{\frac{2}{1-\alpha}})$, 
  w.e.p.~vertices $v$ and $w$ are neighbors in $G$ and thus also in $H$.

Assume now that $v$ and $w$ belong to 
  $S^{1-b}\cap\Phi\cap\calP_{\ellbdr}$.
Suppose there exists $u\in S^{b}$ between $v$ and $w$ 
  with $r_v, r_w < r_u$ so that one of the following happens:
  \begin{inparaenum}[(i).-] 
  \item $u$ is adjacent to a vertex in $S^{1-b}$ 
    \item $u$ is adjacent to a vertex $z\in S^b\cap B_O(\ellbdr-1)$ 
  between $v$ and $w$, in which case, since $v$ and $w$ are adjacent, 
  by Lemma~\ref{lem:fk} part~\eqref{lem:fkPart1},
  the edges $vz$ and $wz$ 
  must also be present, or \item $u$ is adjacent to a vertex
  $z\in S^b\cap B_O(\ellbdr-1)$ 
  with $\theta_z \not\in [\theta_w, \theta_v]$ (since we assume $v$ follows $w$, we assume $\theta_v \ge \theta_w$),
  in which case, by Lemma~\ref{lem:fk} part~\eqref{lem:fkPart2},
  the edge $wz$ or the edge $vz$ also has to be present.
  \end{inparaenum} 
In all cases, for each of the aforementioned pair of vertices $v,w$ we obtain at
  least one edge going from $S^b\cap\Phi$ to $S^{1-b}$,
  and since, w.e.p., there are at least $\frac{\epsilon\omega_0}{\upsilon}\EE|\Phi\cap\calP_{\ellbdr}|$
  regions and every edge between $S^b$ and $S^{1-b}$ is counted at most
  twice, w.e.p.~$|E(S^b \cap \Phi, S^{1-b})|
     =\frac{\epsilon\omega_0}{\upsilon}\EE|\Phi\cap\calP_{\ellbdr}|
     =\Omega(\epsilon\EE|\Phi\cap\calP_{\ellbdr}|)$.
\end{proof}
The next lemma shows that if for a fixed choice of $S$, in a certain
  sector there is an important quantity of both $S$ and $\overline{S}$,
  then the sector's conductance is large. 
Intuitively, this can occur either because there exists one band
  having both large fractions of $S$ and $\overline{S}$, or there are
  two bands, one having a large fraction of $S$, the other having a
  large fraction of $\overline{S}$, or because most of $\overline{S}$
  is relatively close to the center, and most of $S$ is concentrated close 
  to the boundary, in which case we can apply Lemma~\ref{lem:cuna}.
\begin{lemma}\label{lem:Conetreatment}
Let $H=(U,F)$ be the center component of 
  $G=(V,E)$ chosen according to $\poimod_{\alpha,C}(n)$.
Let $\omega_0$, $\epsilon$, $\phi$ and $\ell_\phi$ be as in 
  Lemma~\ref{lem:earlylayersfull}.
Let $\Phi'$ be a  $(2\phi)$-sector of $B_O({R})$. 
  If $S\subseteq U$ is such that  
  $|S\cap\Phi'|,|\overline{S}\cap\Phi'|=\Omega(\EE|\Phi'\cap\calP_{\ellbdr}|)$,
  then for some $b \in \{0,1\}$, w.e.p.~$|E(S^b \cap \Phi', S^{1-b})|
      =\Omega((\log n)^{-\frac{4}{1-\alpha}}(\phi n)^{2(1-\alpha)}).$
\end{lemma}
\begin{proof}
Note that by Remark~\ref{rem:aproxAngle} every vertex $v\in\calP_{\ell_{\phi}}$ 
  is adjacent to every other vertex $v'\in\calP_{\ell_{\phi}}$ satisfying 
  $\Delta\varphi_{v,v'}\le\theta_{R}(\ell_{\phi},\ell_{\phi})$.
Thus, since $\theta_R(\ell_{\phi},\ell_{\phi})\ge(2+o(1))e\phi\ge 2\phi,$ 
  in particular any two vertices in $\calP_{\ell_{\phi}}\cap\Phi'$ are adjacent. 
By choice of $\ell_{\phi}$ and the lower bound on $\phi$, 
  w.e.p.~$|\Phi' \cap \calP_{\ell_{\phi}}|=(1+o(1))\EE|\Phi' \cap \calP_{\ell_{\phi}}|$. 
Thus, if for both $b=0$ and $b=1$ it holds that 
  $|S^b\cap\Phi'\cap\calP_{\ell_{\phi}}|
      \ge{\epsilon}\EE|\Phi'\cap\calP_{\ell_{\phi}}|$,
  then w.e.p.~$|E(S^b \cap \Phi', S^{1-b})|
    =\Omega(({\epsilon}\EE|\Phi'\cap\calP_{\ell_{\phi}}|)^2)$.
Otherwise, for some $b\in\{0,1\}$ we have 
  $|S^b\cap\Phi'\cap\calP_{\ell_{\phi}}|\leq {\epsilon}\EE|\Phi'\cap\calP_{\ell_{\phi}}|$. 
If there exists some $\ell_{\phi}\le\ell \le \ellbdr$ with $|S^b\cap\Phi'\cap\calP_{\ell}|\geq 2\epsilon\EE|\Phi'\cap\calP_{\ell}|,$
  by Lemma~\ref{lem:earlylayersfull}
  (applied with $\ell^*=\ell$),
  we get that 
  w.e.p.~$|E(S^b\cap\Phi', S^{1-b} \cap \Phi')| =
     \Omega((\phi n)^{2(1-\alpha)}\frac{\upsilon}{\omega_0}(\log n)^{-\frac{2}{1-\alpha}})=\Omega((\log n)^{-\frac{4}{1-\alpha}}(\phi n)^{2(1-\alpha)})$.
If not, then $|S^{1-b}\cap\Phi'\cap\calP_{\ellbdr}|
  \geq (1-2\epsilon)\EE|\Phi'\cap\calP_{\ellbdr}|.$
We apply Lemma~\ref{lem:cuna} (which we may since $|S\cap\Phi'|=\Omega(\EE|\Phi'\cap\calP_{\ellbdr}|)$), we obtain that
  w.e.p.~$|E(S^b \cap\Phi', S^{1-b})|
     =\Omega(\epsilon\EE|\Phi'\cap\calP_{\ellbdr}|).$
     
To conclude, observe that by our choice of $\ell_{\phi}$ and 
  Corollary~\ref{cor:medidalb}, we have
  $(\epsilon\EE|\Phi'\cap\calP_{\ell_{\phi}}|)^2=\Omega(\epsilon^2 (\phi n)^{2(1-\alpha)})=\Omega({\frac{1}{\omega_0^2} (\log n)^{-\frac{2(1+\alpha)}{1-\alpha}}}(\phi n)^{2(1-\alpha)})=\Omega((\log n)^{-\frac{4}{1-\alpha}}(\phi n)^{2(1-\alpha)})$, 
  where the latter equality holds by our assumption on $\omega_0$.
 Also, again by Corollary~\ref{cor:medidalb}, our choice of $\ellbdr$ and $\epsilon$, 
   we infer that $\epsilon\EE|\Phi'\cap\calP_{\ellbdr}|
   =\Omega(\frac{1}{\omega_0}(\phi n)(\log n)^{-\frac{1+3\alpha}{1-\alpha}})=\Omega((\log n)^{-\frac{4}{1-\alpha}}(\phi n)^{2(1-\alpha)})$, where the
latter equality follows from the fact that $\frac12 < \alpha < 1$ and by our assumption on $\omega_0$.
\end{proof}
A very similar lemma is the following:
\begin{lemma}\label{lem:Conetreatment1}
Let $H=(U,F)$ be the center component of $G=(V,E)$ chosen according to 
  $\poimod_{\alpha,C}(n)$.
Let $\omega_0$, $\phi$, $\ell_{\phi}$, $\epsilon$ be 
  as in Lemma~\ref{lem:earlylayersfull} and
  let $\Phi$ be a $\phi$-sector of $B_{O}®$.
There is a sufficiently large $C_1=C_1(\alpha)$ such that if $S\subseteq U$ satisfies 
\[
\Vol(S\cap\Phi)\ge C_1\EE|\Phi\cap\calP_{\ellbdr}|(\log n)^{\frac{1}{1-\alpha}},
 \quad\text{and}\quad
|\overline{S}\cap\Phi|=\Omega(\EE|\Phi\cap\calP_{\ellbdr}|),
\]
then for some $b\in\{0,1\}$, w.e.p.~$|E(S^b \cap \Phi, S^{1-b})|=\Omega((\log n)^{-\frac{4}{1-\alpha}}(\phi n)^{2(1-\alpha)})$.
\end{lemma}
\begin{proof}
As in  the proof of Lemma~\ref{lem:Conetreatment},
  if for both $b=0$ and $b=1$ it holds that 
  $|S^b\cap\Phi\cap\calP_{\ell_{\phi}}|=\epsilon\EE|\Phi\cap\calP_{\ell_{\phi}}|$,
  then w.e.p.~$|E(S^b \cap \Phi, S^{1-b})|
                  =\Omega((\epsilon\EE|\Phi\cap\calP_{\ell_{\phi}}|)^2)=\Omega((\log n)^{-\frac{4}{1-\alpha}}(\phi n)^{2(1-\alpha)})$.

Otherwise, suppose that for some $b\in\{0,1\}$ we have
  $|S^b\cap\Phi\cap\calP_{\ell_{\phi}}|\le\epsilon\EE|\Phi\cap\calP_{\ell_{\phi}}|$. 
If there exists some $\ell_{\phi}\le\ell \le\ellbdr$ such
  that $|S^b\cap\Phi\cap\calP_{\ell}|\geq 2\epsilon\EE|\Phi\cap\calP_{\ell}|,$ 
  by Lemma~\ref{lem:earlylayersfull}, 
  w.e.p.~$|E(S^b \cap \Phi, S^{1-b})|
     = \Omega((\log n)^{-\frac{4}{1-\alpha}}(\phi n)^{2(1-\alpha)})$.
If not and $b=1$, then Lemma~\ref{lem:cuna} can be applied 
  and hence, w.e.p.~$|E(S^b \cap \Phi, S^{1-b})|
    =\Omega(\epsilon\EE|\Phi\cap\calP_{\ellbdr}|)$.
So, assume $|S\cap\Phi\cap\calP_{\ell}|\le\epsilon\EE|\Phi\cap\calP_{\ell}|$
 for all $\ell_{\phi}\le\ell\le\ellbdr$
  and $|S\cap\Phi\cap\calP_{\ellbdr}|\le 2\epsilon\EE|\Phi\cap\calP_{\ellbdr}|$. 
If there exists a
  $v\in S\cap\Phi\cap B_O(2\log\frac{1}{\phi}+R-\ellbdr)$,
  then by Lemma~\ref{lem:aproxAngle}, the vertex $v$ is adjacent to every 
  vertex in $\calP\cap\Phi\cap B_O(\ellbdr)$.
By just counting edges between $v$ and 
  $\overline{S}\cap \Phi\cap B_O(\ellbdr)$ 
  we obtain for $b=0$ w.e.p.~$|E(S^b \cap \Phi, S^{1-b})|
     \geq |\overline{S}\cap \Phi\cap B_O(\ellbdr)|
     \geq (1-2\epsilon)\EE|\Phi\cap \calP_{\ellbdr}|$.
If no such vertex $v$ exists, then by Lemma~\ref{lem:muBall}, 
  Lemma~\ref{lem:meanMeus} and Proposition~\ref{prop:degreeBound}, 
  w.e.p.~the volume of $S\cap\Phi\cap B_{O}(\ellbdr)$ is at most 
  $\frac{C_1}{2}\phi n (\log n)^{-\frac{2\alpha-1}{1-\alpha}}
     \leq\frac{C_1}{2}\EE|\Phi\cap\calP_{\ellbdr}|(\log n)^{\frac{1}{1-\alpha}}$ 
  for $C_1$ large enough: 
indeed, by Lemma~\ref{lem:meanMeus} and Proposition~\ref{prop:degreeBound}, 
the volume is, w.e.p., at most
  \begin{align*}
 & \sum_{\ell=2\log\frac{1}{\phi}+R-\ellbdr}^{\ellbdr} \max\{\upsilon \log n, O(\phi ne^{-\alpha(R-\ell)}) \}\Theta(e^{\frac12(R-\ell)}).
  \end{align*}
Using $\max\{x,y\}\leq x+y$, $\alpha < 1$ and 
  the formula for a geometric series, we obtain a 
  $O(\phi n (\log n)^{\frac{1-2\alpha}{1-\alpha}}+\upsilon \phi n(\log n)^{-\frac{\alpha}{1-\alpha}})=O(\phi n (\log n)^{\frac{1-2\alpha}{1-\alpha}})$ bound on the volume.
Since every other vertex, once more by Proposition~\ref{prop:degreeBound},
  w.e.p.~has degree $O((\log n)^{\frac{1}{1-\alpha}})$, by our 
  assumption on $\Vol(S\cap\Phi)$, 
  w.e.p.~$|S\cap\Phi|= \Omega(\Vol(S\cap\Phi)(\log n)^{-\frac{1}{1-\alpha}})
      =\Omega(\EE|\Phi\cap\calP_{\ellbdr}|)$.
Applying Lemma~\ref{lem:cuna} with $b=0$ we get that 
  w.e.p.~$|E(S^b \cap \Phi, S^{1-b})|
=\Omega(\epsilon\EE|\Phi\cap\calP_{\ellbdr}|)$. The previous discussion and 
   similar observations as those in the last paragraph of the proof 
   of Lemma~\ref{lem:Conetreatment} yield the claim. 
\end{proof}

We use the previous lemma in roughly the following way: for a fixed 
  $S\subseteq U$, we start by applying the lemma with $\Phi$ a 
  sector with a relatively large angle so that inside it
  we cannot have only $S$ (the existence of such an angle follows from 
  the fact that we are interested solely in the cases where
  $\Vol(S)$ is sublinear in $n$), and then, in case we have not
  found dense spots of $S$, we half the previous sector, and continue
  recursively. 
Thus, we either detect subsectors of $S$, in
  which case the previous lemmas imply a large conductance, or conclude that there is no relatively large angle containing only $S$.
\begin{lemma}\label{lem:Conetreatment2}
Let $H=(U,F)$ be the center component of $G=(V,E)$ chosen according to 
  $\poimod_{\alpha,C}(n)$. 
Let  $\Phi$ be a sector of $B_{O}(R)$ of angle 
  $\phi$ with $\phi \ge \phi_0:=\frac{\upsilon}{n}(\log n)^{\frac{1+\alpha}{1-\alpha}}$. 
Let $j \ge 0$ be the largest integer such that 
  $2^{-j}\phi\ge\phi_0$ and, for $0\le i\le j$, let 
  $\Phi^{(i)}_{1},\ldots,\Phi^{(i)}_{2^i}$
  be an angular equipartition of $\Phi$.
Then, there is a constant $0 < C_2 < 1$ such 
  that w.e.p.~$|U\cap\Phi^{(j)}_{k}|\geq C_2\frac{\phi n}{2^j}(\log n)^{-\frac{2\alpha}{1-\alpha}}$ 
  for every $1\leq k\leq 2^j$.
Moreover, let $C_1>0$ be as in Lemma~\ref{lem:Conetreatment1} and consider 
  $S\subseteq U$ such that $|S\cap\Phi|
    \le \frac{C_2}{3}\EE|\Phi\cap\calP_{\ellbdr}|$
  and $\Vol(S\cap\Phi)\le C_1\phi n$. Then, w.e.p., for each $\Phi^{(j)}_{k}$ 
  one of the following holds:
\begin{enumerate}[(i).-]
\item\label{it:Conetreatment2Part1} 
there is $0\le i\le j$ and a $k'$ for which 
  $h_{\Phi^{(i)}_{k'}}(S)=\Omega\big((\log n)^{-\frac{4}{1-\alpha}}(\frac{2^i}{\phi n})^{2\alpha-1}\big)$ and   $\Phi^{(j)}_{k} \subseteq \Phi^{(i)}_{k'}$ 
  or
\item\label{it:Conetreatment2Part2} 
  $|S\cap\Phi^{(j)}_{k}|\le\frac{C_2}{3}\EE|\Phi^{(j)}_{k}\cap\calP_{\ellbdr}|$.
\end{enumerate}
\end{lemma}
\begin{proof}
The existence of $C_2$ is a direct consequence of 
  Lemma~\ref{lem:newvolume} and the fact that, 
  by Corollary~\ref{cor:medidalb} and Lemma~\ref{lem:meanMeus}, we have
  $\EE|U\cap\Phi^{(j)}_{k}|\geq \EE|\Phi^{(j)}_{k}\cap\calP_{\ellbdr}|
     =\Theta(\frac{\phi n}{2^j}(\log n)^{-\frac{2\alpha}{1-\alpha}})$.

We show, by induction on $i$, $0\leq i\le j$,  
  that at recursion depth $i$ we have for all $1\leq k\leq 2^i$ either
    $|S\cap\Phi_{k}^{(i)}|\le\frac{C_2}{3}\EE|\Phi^{(i)}_{k}\cap\calP_{\ellbdr}|$ 
  and  $\Vol(S\cap\Phi_{k}^{(i)})\leq 2C_1\frac{\phi n}{2^i}$, 
  or $\Phi_{k}^{(i)}\subseteq\Phi^{(i')}_{k'}$ and 
  $h_{\Phi^{(i')}_{k'}}(S)
     =\Omega\big (\log n)^{-\frac{4}{1-\alpha}}(\frac{2^{i'}}{\phi n})^{2\alpha-1}\big)$
  for some $0\leq i'<i$ and $1\leq k'\leq 2^{i'}$.
By hypothesis and since $\Phi^{(0)}_{1}=\Phi$, the claim holds for $i=0$.
Assume it is true for $i-1$.
Let $k',k$ be such that $\Phi^{(i)}_{k}\subseteq \Phi^{(i-1)}_{k'}$
  with 
  $|S\cap\Phi_{k'}^{(i-1)}|\le\frac{C_2}{3}\EE|\Phi_{k'}^{(i-1)}\cap\calP_{\ellbdr}|$ 
  and  $\Vol(S\cap\Phi_{k'}^{(i-1)})\leq 2C_1\frac{\phi n}{2^{i-1}}$.
If $|S\cap\Phi^{(i)}_{k}|\ge\frac{C_2}{3}\EE|\Phi^{(i)}_{k}\cap\calP_{\ellbdr}|$, 
  then also $|S\cap\Phi^{(i-1)}_{k}|=\Omega(|\Phi^{(i-1)}_{k}\cap\calP_{\ellbdr}|),$
  and hence by Lemma~\ref{lem:Conetreatment} applied 
  with  $\Phi'=\Phi^{(i-1)}_{k'}$
  we get that for some $b \in \{0,1\}$
  w.e.p.~$|E(S^{b}\cap\Phi^{(i-1)}_{k}, S^{1-b})|
      =\Omega((\log n)^{-\frac{4}{1-\alpha}}(\frac{\phi n}{2^{i-1}})^{2(1-\alpha)})$.
Since $\Vol(S\cap\Phi^{(i-1)}_{k'})
        \leq 2C_1\frac{\phi n}{2^{i-1}}$,
  it follows that, 
  w.e.p.~$h_{\Phi^{(i-1)}_{k}}(S)
      = \Omega\big( (\log n)^{-\frac{4}{1-\alpha}}(\frac{2^{i-1}}{\phi n})^{2\alpha-1} \big)$. 
Otherwise, if it happens that 
  $|S\cap\Phi^{(i)}_{k}|\le\frac{C_2}{3}|\Phi^{(i)}_{k}\cap\calP_{\ellbdr}|$
  and also $\Vol(S\cap\Phi^{(i)}_{k}) > 2C_{1}\frac{\phi n}{2^{i}}$, then 
  first note that still 
  $\Vol(S\cap\Phi^{(i)}_{k})\leq \Vol(S\cap\Phi_{k'}^{(i-1)})
    \leq 4C_1\frac{\phi n}{2^{i}}$ must hold.
In this case, applying Lemma~\ref{lem:Conetreatment1} to $\Phi^{(i)}_k$
  we get that for some $b \in \{0,1\}$ 
  w.e.p.~$|E(S^b\cap\Phi^{(i)}_{k}, S^{1-b})|
      =\Omega((\log n)^{-\frac{4}{1-\alpha}}(\frac{\phi n}{2^i})^{2(1-\alpha)})$ and thus  
  $h_{\Phi^{(i)}_{k}}(S)=\Omega\big( (\log n)^{-\frac{4}{1-\alpha}}(\frac{2^{i}}{\phi n})^{2\alpha-1} \big)$.
This completes the induction since the only remaining possibility is that
  $|S\cap\Phi_{k}^{(i)}|\le\frac{C_2}{3}\EE|\Phi_{k}^{(i)}\cap\calP_{\ellbdr}|$ 
  and $\Vol(S\cap\Phi_{k}^{(i)})\leq 2C_1\frac{\phi n}{2^i}$.
\end{proof}

Now we are ready to prove Theorem~\ref{thm:conductancia}. 
We show that every set $S \subseteq U$ with $\Vol(S)=O(n^{\varepsilon})$ has the desired conductance. Roughly speaking, the argument goes as follows. We start with sufficiently large angles that cannot contain only $S$. Either we find the desired number of cut edges for subsectors of these sectors via Lemma~\ref{lem:Conetreatment2} part~\eqref{it:Conetreatment2Part1},
  or for the remaining vertices we will find in a not too small angle around them sufficiently many vertices in $S$ and in $\overline{S}$, and hence we can also find relatively many edges between $S$ and $\overline{S}$.

\bigskip
\noindent\textit{Proof of Theorem~\ref{thm:conductancia}:}
We will show that w.e.p.~for all sets $S$ with
$\Vol(S)=O(n^{\varepsilon})$ for some $0 < \varepsilon < 1$ we have
$h(S)=\Omega(n^{-(2\alpha-1)\varepsilon+o(1)})$. 
We will consider an arbitrary, but fixed set $S$ and only at the very end of the proof take into account all possible sets $S$.

Let $\ell_0=(1-\xi)R$ for some $\xi=\xi(n)$ tending to $0$ sufficiently slowly with $n$.
Consider $C_1$ and $C_2$ as in Lemma~\ref{lem:Conetreatment1} and Lemma~\ref{lem:Conetreatment2}, respectively (recall that $C_1$ should be thought of as a sufficiently large and $C_2$ as a small constant).
Fix a set $S$ such that $\Vol(S)=O(n^{\varepsilon})$. 
Hence, there exists a sufficiently large $C'>0$ so that we can partition 
  $B_O(R)$ into $\phi$-sectors,
  $\phi:=C'n^{-(1-\varepsilon)}(\log n)^{\frac{2\alpha}{1-\alpha}}$  so that
  w.e.p.~in each such sector $\Phi$ we have 
  $|S\cap\Phi|\le\frac13|U\cap\Phi|$, 
  $|S\cap\Phi|\le\frac{C_2}{3}\EE|\Phi\cap\calP_{\ellbdr}|$ 
  and $\Vol(S\cap\Phi)\le C_1\phi n.$
To each of these sectors, we apply
  Lemma~\ref{lem:Conetreatment2} with $\phi_0:=\theta_{R}(\ell_0,\ell_0)$.
Thus, w.e.p., every sector $\Phi$ of angle 
  $2^{-j}\phi$, $\phi_0\leq 2^{-j}\phi<2\phi_0$ 
  arising from the application of the lemma is \emph{accounted for}, i.e., $h_\Phi(S)=\Omega( (\log n)^{-\frac{4}{1-\alpha}}(\phi n)^{1-2\alpha})=\Omega(n^{-(2\alpha-1)\varepsilon+o(1)})$, or 
$|S\cap\Phi|\le\frac{C_2}{3}|\Phi\cap\calP_{\ellbdr}|$. 
Let $\calO$ be the collection of all sectors $\Phi$ associated to $S$ which are accounted for. 
Similarly, we say that a truncated sector $\Upsilon_v$ centered at $v \in S$ is accounted for, if $h_{\Upsilon_v}(S)=\Omega(n^{-(2\alpha-1)\varepsilon+o(1)}).$

Next, we iteratively build two additional collections of regions,
  denoted by
  $\calA$ and $\calC$: $\calA$ will be the set of sectors (truncated or not) that are accounted for, and $\calC$ will be the set of regions that are ``compensated", i.e., these regions will not be accounted for, but we will show that their total volume is only slightly larger than the volume of the collection of regions that is accounted for.
Initially, $\calA=\calO$, i.e., 
  $\calR\in\calA$ if and only if $\calR$ is a $\Phi_k^{(j)}$ for 
  which the conditions of 
  part~\eqref{it:Conetreatment2Part1} hold
  and $\calC=\emptyset$.
The iterative process that updates $\calA$ and $\calC$ proceeds
  as described next: 

\medskip
\noindent\textsc{Sector-Accounting}
\begin{enumerate}[(i).-]
\item\label{it:iterativeProc}
  Stop if $S\setminus\mathlarger{\cup}_{\calR\in \calA\cup\calC}\calR=\emptyset$. 
  Otherwise, let $v$ be the vertex in $S\setminus\mathlarger{\cup}_{\calR\in \calA\cup\calC}\calR$ 
  closest to the origin and assume $\ell$ is such that $v\in\calP_\ell$. 
  \item\label{it:iterativeProc2} If $\ell \le \ell_0$, then 
  let $\Upsilon_v$ be the sector truncated and centered at $v$
  of angle $2\theta_{R}(\ell,R)$ 
  \begin{enumerate}
 \item\label{step2a}  If  $\mu(\Upsilon_v\cap \cup_{\calR\in\calA}\calR) < \frac12\mu(\Upsilon_v)$, then add $\Upsilon_v$ to $\calA$ and go to~Step~\eqref{it:iterativeProc}.
    \item\label{step2b} If $\mu(\Upsilon_v\cap \cup_{\calR\in\calA}\calR)\geq\frac12\mu(\Upsilon_v)$ and $\Vol(S\cap\Upsilon_v \cap \cup_{\calR\in\calA}\calR) =o\big((\log n)^{-\frac{2\alpha}{1-\alpha}}\Vol(S\cap\Upsilon_v)\big)$, then add $\Upsilon_v$
  to $\calA$ and go to~Step~\eqref{it:iterativeProc}.
  \item\label{step2c} If $\mu(\Upsilon_v\cap \cup_{\calR\in\calA}\calR)\geq\frac12\mu(\Upsilon_v)$ and $\Vol(S\cap\Upsilon_v \cap \cup_{\calR\in\calA}\calR) =\Omega\big((\log n)^{-\frac{2\alpha}{1-\alpha}} \Vol(S\cap\Upsilon_v)\big)$, then add $\Upsilon_v$ to $\calC$ and go to~Step~\eqref{it:iterativeProc}.
  \end{enumerate}
\item\label{it:iterativeProc3} If $\ell > \ell_0$, then let $\Upsilon_v$ be 
  the sector truncated and centered at $v$ of angle $2\theta_{R}(\ell_0,\ell_0)$.
\begin{enumerate}
  \item\label{step3a}  If $\Upsilon_v\cap \cup_{\calR\in\calA}\calR = \emptyset$, then add $\Upsilon_v$ to $\calA$ and go to~Step~\eqref{it:iterativeProc}.
  \item\label{step3b}  If  $\Upsilon_v\cap \cup_{\calR\in\calA}\calR \neq \emptyset$, then add $\Upsilon_v$ to $\calC$ and go to~Step~\eqref{it:iterativeProc}.
  \end{enumerate}
\end{enumerate}

We claim that if a region $\calR$ ends up in $\calA$, then 
  it is accounted for.
The claim holds at the start of the process by definition
  of $\calO$.

Now, if $v$ is such that $\Upsilon_v$ was added to $\calA$ in
Step~\eqref{step2a}, then at the moment $\Upsilon_v$ was added,
  at least a constant fraction of the sectors of angle
  $2^{-j}\phi$ intersecting $\Upsilon_v$ did
  not belong to $\calO$. 
For each such sector $\Phi_k^{(j)} \not\in\calO$, by 
  Lemma~\ref{lem:Conetreatment2} part~\eqref{it:Conetreatment2Part2},
  we have 
  $|S\cap\Phi_k^{(j)}|\le\frac{C_2}{3}\EE|\Phi_k^{(j)}\cap\calP_{\ellbdr}|\le \frac13\EE|\Phi\cap\calP_{\ellbdr}|$. 
Note that $v$ is adjacent to every vertex in $\calP_{\ellbdr}\cap\Upsilon_v$ 
  (since $\theta_R(\ell,\ellbdr)\geq\theta_R(\ell,R)$) and
  at least a constant fraction of these belong to $\overline{S}$. 
Hence, w.e.p.~we obtain $|E(\{v\},\overline{S})|
   =\Omega\big(n\theta_R(\ell,R)(\log n)^{-\frac{2\alpha}{1-\alpha}}\big)$.  
Also, since by Lemma~\ref{lem:VolUpperBound},
  w.e.p.~$\Vol(S\cap\Upsilon_v)=O(n\theta_R(\ell,R))$, we obtain
  w.e.p.~$h_{\Upsilon_v}(S)\ge (\log n)^{-\frac{2\alpha}{1-\alpha}}$, 
  and $\Upsilon_v$ is accounted for.

Similarly, consider a vertex $v$ such that $\Upsilon_v$ was added
  to $\calA$ in Step~\eqref{step2b}.
Let $\calA_v$ be the collection of regions belonging to $\calA$ just 
  before $\Upsilon_v$ was added to it.
Since by Lemma~\ref{lem:newvolume}, 
  w.e.p.~$|U \cap \Upsilon_v\cap\cup_{\calR\in\calA_v}\calR| 
     =\Omega\big((\log n)^{-\frac{2\alpha}{1-\alpha}} n\theta_R(\ell,R)\big)$
  and by assumption together with Lemma~\ref{lem:VolUpperBound}, 
  $\Vol(S\cap \Upsilon_v\cap \cup_{\calR\in\calA_v}\calR) 
    =o\big((\log n)^{-\frac{2\alpha}{1-\alpha}}\Vol(S\cap\Upsilon_v)\big)
    =o\big((\log n)^{-\frac{2\alpha}{1-\alpha}}n \theta_R(\ell,R)\big)$,
  at least a constant fraction of the vertices in 
  $U \cap\Upsilon_v\cap\cup_{\calR\in\calA}\calR$
  must belong to $\overline{S}$. 
Since these are all adjacent to $v$, by
  counting the edges from $v$ to these, by analogous calculations as
  in the previous case, we obtain w.e.p.~$h_{\Upsilon_v}(S)=\omega(1)$, 
  and $\Upsilon_v$ is accounted for.

Next, consider a vertex $v$ such that $\Upsilon_v$ was added to $\calA$ 
  in Step~\eqref{step3a}.
Again, let $\calA_v$ be the collection of regions belonging to $\calA$ just 
  before $\Upsilon_v$ was added to it.
Consider all vertices in $\calP_{\ellbdr}\cap\Upsilon_v$.
Recall that $\theta_R(\ellbdr,\ellbdr)=\Theta(\frac{1}{n}(\log n)^{\frac{2}{1-\alpha}})$. The expected number of vertices in $\calP_{\ellbdr}$ in 
  a sector of angle
  $\phi_1:=\upsilon \frac{1}{n}(\log n)^{\frac{2\alpha}{1-\alpha}+1}$ is 
  $\upsilon \log n$, and by Theorem~\ref{thm:Chernoff} this holds w.e.p. 
Hence, w.e.p.~the maximal angular distance between any two vertices 
  $v,w\in\calP_{\ellbdr}$ such that $v$ follows $w$ in $\calP_{\ellbdr}$ is at 
  most $\phi_1$. 
Since $\phi_1 < \theta_R(\ellbdr,\ellbdr)$, w.e.p.~any pair of such 
  vertices is adjacent.
Moreover, by Remark~\ref{rem:HtoG}, w.e.p., every vertex in 
  $\calP_{\ellbdr}$ belongs to $U$. 
Thus, $\calP_{\ellbdr}\cap\Upsilon_v$ induces a connected component in $H$. 
Also, since the expected number of vertices in $\calP_{\ellbdr}\cap\Upsilon_v$ 
  is 
  $\Theta(\theta_R(\ell_0,\ell_0)\EE|\calP_{\ellbdr}|)
    =\omega(\log n)$, 
  this holds w.e.p. 
By assumption of this case, $\Upsilon_v\cap\cup_{\calR\in\calA_v}\calR=\emptyset$, 
  and by Lemma~\ref{lem:Conetreatment2} part~\eqref{it:Conetreatment2Part2},
  at least a constant fraction of the 
  vertices in $\calP_{\ellbdr}\cap\Upsilon_v$ belongs to $\overline{S}$.
If at least one of the vertices in $\calP_{\ellbdr}\cap\Upsilon_v$
  belongs to $S$,  w.e.p.~we have that $\calP_{\ellbdr}\cap\Upsilon_v$ induces
  a connected component in $H$ with vertices both in $S$ and $\overline{S}$
  and $|E(S\cap\Upsilon_v,\overline{S})|\ge 1$,
  and since by Lemma~\ref{lem:VolUpperBound}, 
  w.e.p.~$\Vol(S\cap\Upsilon_v)=O(n\theta_{R}(\ell_0,\ell_0))=O(n^{2\xi})$,
  we obtain $h_{\Upsilon_v}(S)=\Omega(n^{-2\xi})$. 
The same argument applies if $v$ is adjacent to a vertex in $\overline{S}$.
If $\calP_{\ellbdr}\cap\Upsilon_v\subseteq \overline{S}$ 
  and $r_v \le \ellbdr$, 
  by Lemma~\ref{lem:fk} part~\eqref{lem:fkPart1},
  given that $\Upsilon_v$ is centered at $v$, w.e.p., $v$ lies between
  a pair of vertices of the connected 
  component of $H$ induced by 
  $\calP\cap\Upsilon_v\cap B_{O}(\ell)\setminus B_{O}(\ell-1)$, so
  $v$ is adjacent to a vertex in 
  $\overline{S} \cap \calP_{\ellbdr} \cap \Upsilon_v$ as well, and the same 
  conclusion holds. 
If $r_v > \ellbdr$, then since $v$ is in $U$, it must be connected by a path to 
  a vertex in $\overline{S}$, and either we find on this path,
  by Lemma~\ref{lem:fk} part~\eqref{lem:fkPart1} 
  (in case the path uses only vertices with radius larger than $\ellbdr$) or by  Lemma~\ref{lem:fk} part~\eqref{lem:fkPart2}
  otherwise,  an edge between vertices in 
  $S\cap\Upsilon_v$ and $\overline{S}$ or
  between vertices in $S$ and $\overline{S}\cap\Upsilon_v$.
In both cases we have
  w.e.p.~$h_{\Upsilon_v}(S) 
     = \Omega(n^{-2\xi})=\Omega(n^{-(2\alpha-1)\varepsilon+o(1)})$ 
  by our assumption on $\xi$ tending to $0$, and in all cases $\Upsilon_v$ is accounted for.
To conclude, note that each edge is counted at most six times for the conductance of different regions in $\calA$: in order for an edge to be counted for the conductance of a region $\calR$ belonging to $\calA$, by definition of $h(S)$ and $h_{\calR}(S)$ (see~\eqref{eqn:setConductance} and~\eqref{eq:conduct}), at least one of its endpoints must belong to it. First, since the sectors $\Phi$ which are accounted for by Lemma~\ref{lem:Conetreatment2} are disjoint, each point $p \in B_O({R})$ can appear in at most one such sector. Next, let $\calR \in \calA$ be the first region in which $p$ appears in~\textsc{Accounting-Sectors}: since $\calR$ is connected, it has a bisector, and we may 
  assume without loss of generality that $p$ is to the left of the bisector of $\calR$ (here and below ``to the left'' is understood as preceding in a counter-clockwise ordering; ``to the right'' is defined analogously).
Since no vertex $v$ with $v \in \calR$ is chosen in the algorithm after having added $\calR$ to $\calA$, and since 
  the measures of the regions added to $\calA$ are 
  non-increasing 
  during the algorithm (and hence at any radial distance the width of the next region is at most as big as the previous one), no region $\Upsilon_v$ added to $\calA$ after $\calR$, and with $v$ to the right of~$\calR$ can contain $p$. 
  If $v$ is to the left of $\calR$, 
  then $\Upsilon_v$ can contain 
  $p$, but $p$ is now to the right 
  of the bisector of $\Upsilon_v$. 
Hence, any region to the left of $\Upsilon_v$ cannot contain $p$ anymore. Summarizing, we may associate each point 
  $p \in B_O({R})$ to at most $1$ region in $\calO$ and $2$ regions in $\calA\setminus\calO$, i.e., to at most $3$ regions $\calR\in\calA$, 
  and hence a cut edge is counted at most six times.

Next, let $\calC'$ be the collection of regions added 
  to $\calC$ in Step~\eqref{step2c}. 
By definition, for every region $\calR'\in\calC'$ we have 
  $\Vol(S\cap \calR')
   =O((\log n)^{\frac{2\alpha}{1-\alpha}}\sum_{\calR\in\calA}\Vol(S\cap\calR))$. By the same argument as above, each point $p \in B_O({R})$ can be contained in at most two regions $\calR' \in \calC'$.
Thus, in particular any point $p \in \calR$ for $\calR \in \calA$ 
  is contained in
   at most two regions $\calR' \in \calC'$, and we obtain
  $\Vol(S\cap \cup_{\calR' \in \calC'} \calR')=O((\log n)^{\frac{2\alpha}{1-\alpha}}\sum_{\calR\in\calA}\Vol(S\cap\calR))$. 
The same argument also applies when adding regions to $\calC$ in 
  Step~\eqref{step3b}: as before, 
  every point $p\in\calR$ for $\calR\in\calA$ is contained
  in at most two regions $\calR' \in \calC$. 
Since for every such region $\calR'$ we have 
  $\Vol(S \cap \calR')=O(n^{2\xi})$, 
  and since $\xi$ tends to $0$ slowly enough so that 
  $n^{2\xi} \ge (\log n)^{\frac{2\alpha}{1-\alpha}}$, we obtain
\begin{align*}
& \Vol(S)
\leq \sum_{\calR\in\calA}\big(O((\log n)^{\frac{2\alpha}{1-\alpha}})\Vol(S\cap\calR)+O(n^{2\xi})\big) =
  O(n^{2\xi})\sum_{\calR\in\calA}\Vol(S\cap\calR).
\end{align*}
Hence, by~\eqref{eq:conduct}, since 
  $h_{\calR}(S)=\Omega(n^{-(2\alpha-1)\varepsilon+o(1)})$
  for $\calR\in\calA$, and since $n^{\xi}=n^{o(1)}$ by our assumption on $\xi$, 
  we are done for this set $S$. 

So far we have considered one single fixed set $S$. A close inspection of all probabilistic events in Lemma~\ref{lem:earlylayersfull}
  through Lemma~\ref{lem:Conetreatment2} 
  shows that they depend either on the angle chosen, or on single vertices or pairs of vertices, but not on the whole set of vertices belonging to $S$. The starting angles chosen in Lemma~\ref{lem:Conetreatment2} can also be chosen to be the same for all $S$, so that altogether for all $S$ only polynomially many angles are used. Hence, only a union bound over polynomially many events is needed, and all properties given in all lemmata for one $S$ hold simultaneously for all choices of $S$. The proof of the theorem is finished.
\qed

\section{Bisections and cuts}\label{sec:bisec}
In this section we derive some consequences of the previous 
  sections' results.

\medskip
\noindent\textit{Proof of Corollary~\ref{cor:cut}.} 
Let $H=(U,F)$ be the giant component of
  $G=(V,E)$ chosen according to $\poimod_{\alpha,C}(n)$. 
First, note that by Corollary~\ref{cor:volTobias}, w.e.p.~$|U|=\Theta(n)$, 
  and hence for any bisection of $\{S,U \setminus S\}$, w.e.p.~we 
  have $\Vol(S)=\Theta(n),\Vol(U \setminus S)=\Theta(n)$.
By~definition of conductance (see~\eqref{eqn:setConductance}) we have
  $h(S)=\Theta(\frac{1}{n}|\partial(S)|).$  
Recalling Cheeger's inequality (see~\eqref{eqn:cheeger}), for any
  graph $G$ its conductance $h(G)$ satisfies $h(G)\ge\frac{1}{2}\lambda_1(G).$
Therefore, by Corollary~\ref{cor:main}, 
  for any bisection $\{S,U \setminus S\}$, w.e.p.,
$$
\frac{|\partial(S)|}{n} = \Omega(h(H))
  =\Omega\Big(\frac{1/D}{n^{2\alpha-1}}\Big),
$$
and hence for any $S$ with $|S|=\lceil\frac12|U|\rceil$ we must have 
  $|\partial(S)|=\Omega(n^{2(1-\alpha)}/D),$ so the first 
  part of the claimed result follows. 

For the second part, observe that since by Lemma~\ref{lem:VolUpperBound}, 
  w.e.p.~$\Vol(U)=O(n)$, clearly $B(H)=O(n).$ On the other hand, consider the 
  bisection $\{S,U\setminus S\}$ with $S$ consisting of those 
  $\lceil \frac12|U|\rceil$ vertices of $H$ with minimal radial coordinate $r_u$. 
By Lemma~\ref{lem:muBall} and Lemma~\ref{lem:meanMeus}, there exists
  a large constant $C_1$ such that the number of vertices 
  in $B_{O}(R-C_1)$ is w.e.p.~smaller than $\varepsilon n\le\frac14|U|$ 
  for small enough $\varepsilon$. 
Thus, there exists $C_1' < C_1$ such that w.e.p.~all vertices 
  $v \in S$ belong to $B_{O}(R-C'_1)$.
Moreover, for every fixed $0 < \delta < \frac12$, by 
  Corollary~\ref{cor:medidalb}, w.e.p.~there exists a 
  constant $c_1=c_1(\delta)$ with $C_1 > c_1> C_1'$ 
  such that
  a $\delta$-fraction of the vertices in $S$ belong to 
  $B_{O}(R)\setminus B_{O}(R-c_1)$. 

Let now $\calB:=\calP\cap B_O(R-C_1)\setminus B_O(R-C_1-1)$ and 
  $\calB':=\calP\cap B_O(R-C_1')\setminus B_O(R-c_1)$. Recall that $\ellbdr:=\lfloor R-\frac{2\log R}{1-\alpha}\rfloor$.
By Lemma~7 of~\cite{km15}, for each vertex  $u \in \calB$ 
  there is a positive probability  to be connected through a path of vertices of decreasing radii 
  (with all internal vertices of the path belonging to $B_{O}(R-C_1-1)$)
  to a vertex  in
  $\calP\cap B_O(\ellbdr)$, and moreover, 
  by Remark~\ref{rem:HtoG},
  w.e.p.~every vertex in $\calP \cap B_O(\ellbdr)$
  belongs to $H$. 
W.e.p., $|\calB|=\Theta(n)$, and so~$\EE|U \cap \calB|=\Theta(n)$, 
 and since for any two vertices at angular distance 
  $\frac{1}{n}(\log n)^{\omega(1)}$ 
  the events of having such a
  path to a vertex in $\calP\cap B_O(\ellbdr)$ are
  independent, $\VV|U \cap \calB|=n(\log n)^{\omega(1)}$,   and hence, by Chebyshev's inequality, with probability 
   $1-O(n^{-1+\xi})$ for any small constant $\xi > 0$,  we have 
  $|U \cap \calB|=\Theta(n).$

Next, for each vertex $u \in \calB$, by 
  Lemma~\ref{lem:muBallInterGen},
  there exists a non-zero probability 
  $P$ 
  that $u$ has at least one neighbor in $\calP\cap B_O(R-C_1')$.
By applying
  Lemma~\ref{lem:muBallInterGen} one more time, 
  there is positive probability $P' < P$
  that it has at least one neighbor in $\calP\cap B_O(R-c_1)$,  and hence, for each $u \in\calB$, there is positive
  probability (at least $P-P'$) to have at least one neighbor in 
  $\calB'$. 
For any two vertices $u,u'\in\calB$ such that $\Delta\varphi_{u,u'}\ge\frac{C_2}{n}$ with $C_2$ 
  sufficiently large, the corresponding events of
  having at least one neighbor in~$\calB'$
  are independent.  
Therefore, by the same argument as before, by Chebyshev's inequality, 
  with probability at least $1-O(n^{-1+\xi})$, we have
  $|E(\calB,\calB')|=\Omega(n)$.
Moreover, since $C_1 > c_1 > C_1'$, for every vertex $u\in\calB$,  
  the events of having a path of vertices of decreasing radii starting from $u$ to 
  $\calP \cap B_O(\ellbdr)$ 
  with all internal vertices of the path inside $B_{O}(R-C_1-1)$
  and of having an edge 
  between $u$ and a vertex in $\calB'$ are independent.
Hence, recalling that we have already shown that $|U\cap\calB|=\Theta(n)$
  with probability at least $1-O(n^{-1+\xi})$ for any small constant $\xi > 0$, 
  we obtain with probability at least $1-O(n^{-1+\xi})$ that
  $|E(U \cap \calB,\calB')|=\Omega(n)$,
  and thus $B(H)=\Omega(n)$, so the second part of the statement follows. 
\qed

\medskip
The related questions regarding the minimum and maximum cut size of $H$
  (i.e., minimum and maximum number of edges between the two parts of a non-trivial partition of the vertex set of $H$, respectively) follow easily from 
  results proved here and in the literature.
For the minimum cut,  
   by the proof of Theorem~3
  of~\cite{fk15}, w.e.p.~there exists a path of length 
  $\Theta(\log n)$ starting at a vertex $u$ having no other neighbor. 
Hence, w.e.p.~there will be a leaf $u$ in $H$, and therefore, by
  considering the cut set $\{u\},$ we obtain $mc(H)=1$.
For the maximum cut, note that by Lemma~\ref{lem:VolUpperBound}, 
  w.e.p.~$\Vol(U)=\Theta(n)$, and hence $MC(H)=O(n)$. 
For a maximum bisection, as shown above, the bound is attained, and hence 
  $MC(H)=\Theta(n).$



\section{Conclusion and outlook}\label{sec:conclusion}
In this paper we have, up to a polylogarithmic factor, shown that the
  conductance of the giant component of a random hyperbolic graph is
  $\Theta(n^{-(2\alpha-1)})$, and the same holds for the spectral gap of
  the normalized Laplacian of the giant component of such a graph. 
We have established that there are relatively small bottlenecks that disconnect large fractions of vertices of the graph's giant component, and we also showed that for smaller sets of vertices, the conductance of such sets, is compared to larger sets, bigger. 

Given the fundamental nature of the two parameters studied in this 
  paper, i.e., spectral gap and conductance, their determination should
  contribute to the understanding of the random hyperbolic graph
  model, and in particular, to the understanding of issues concerning well 
  known related topics
  such as the spread of information,
  mixing time of random walks, and similar phenomena 
  in such a model.
It is widely believed that social networks are fast mixing (see 
  for example the discussion in~\cite{MYK10}) and that rumors spread
  fast in such networks.
Given the interest in random hyperbolic graphs as a model of 
  networks that exhibit common properties of social networks, 
  it is natural to ask whether fast mixing and rumor spreading
  does indeed occur. 
The low conductance and the 
  spectral gap we establish do not give evidence that it is so.

\bibliographystyle{alpha}
\bibliography{biblio} 


\end{document}